\documentclass{amsart}

\usepackage{amssymb}
\usepackage{amsthm}
\usepackage{amsmath}
\usepackage{mathtools}

\usepackage[usenames,dvipsnames]{xcolor}

\usepackage{enumitem}

\usepackage{hyperref}
\hypersetup{
    colorlinks=true,
    linkcolor=blue,
    filecolor=magenta,      
    urlcolor=cyan,
    citecolor=red
}

\theoremstyle{plain}
\newtheorem{proposition}{Proposition}[section]
\newtheorem{theorem}[proposition]{Theorem}
\newtheorem{lemma}[proposition]{Lemma}
\newtheorem{corollary}[proposition]{Corollary}
\theoremstyle{definition}
\newtheorem{example}[proposition]{Example}
\newtheorem{definition}[proposition]{Definition}
\newtheorem{observation}[proposition]{Observation}
\theoremstyle{remark}
\newtheorem{remark}[proposition]{Remark}

\DeclareMathOperator{\diam}{diam}
\DeclareMathOperator{\diag}{diag}

\DeclareMathOperator{\SL}{\mathsf{SL}}

\DeclareMathOperator{\PGL}{\mathsf{PGL}}

\DeclareMathOperator{\id}{id} 
 
\DeclareMathOperator{\Leb}{Leb}

\DeclareMathOperator{\Isom}{Isom}

\DeclareMathOperator{\dist}{d}

\DeclareMathOperator{\Bc}{\mathcal{B}}
\DeclareMathOperator{\Cc}{\mathcal{C}}
\DeclareMathOperator{\Fc}{\mathcal{F}}

\DeclareMathOperator{\Hc}{\mathcal{H}}

\DeclareMathOperator{\Lc}{\mathcal{L}}

\DeclareMathOperator{\Nb}{\mathbb{N}}

\DeclareMathOperator{\R}{\mathbb{R}}
\DeclareMathOperator{\Rb}{\mathbb{R}}

\DeclareMathOperator{\Zb}{\mathbb{Z}}

\DeclareMathOperator{\lox}{lox}

\newcommand{\abs}[1]{\left|#1\right|}

\newcommand{\norm}[1]{\left\|#1\right\|}

\newcommand{\wh}[1]{\widehat{#1}}

\newcommand{\mc}{\mathcal}
\newcommand{\mr}{\mathrm}
\newcommand{\ms}{\mathsf}
\newcommand{\mb}{\mathbb}
\newcommand{\mf}{\mathfrak}
\begin{document}

\title[Counting, mixing and equidistribution]{Counting, mixing and equidistribution for GPS systems with applications to relatively Anosov groups}
\author[Blayac]{Pierre-Louis Blayac}
\address{University of Michigan}
\author[Canary]{Richard Canary}
\address{University of Michigan}
\author[Zhu]{Feng Zhu}
\address{University of Wisconsin-Madison}
\author[Zimmer]{Andrew Zimmer}
\address{University of Wisconsin-Madison}
\thanks{Canary was partially supported by grant  DMS- 2304636 from the National Science Foundation.
Zhu was partially supported by an AMS--Simons travel grant.
Zimmer was partially supported by a Sloan research fellowship and grant DMS-2105580 from the National Science Foundation.}

\date{\today}
\subjclass[2010]{}

\begin{abstract} 
We establish counting, mixing and equidistribution results for finite BMS measures on flow spaces associated to geometrically finite convergence group actions.
We show that, in particular, these results apply to flow spaces associated to relatively Anosov groups.

\end{abstract}

\maketitle

\tableofcontents

\section{Introduction}

In previous work \cite{BCZZ}, we introduced the theory of Patterson--Sullivan measures for Gromov--Patterson--Sullivan (GPS) systems.
We recall that a GPS system consists of a pair of cocycles for a discrete convergence
groups action which are related by  a Gromov product.
Given a GPS system we constructed Patterson--Sullivan measures, a flow space and a Bowen--Margulis--Sullivan (BMS) measure.
We established a Hopf--Tsuji--Sullivan dichotomy for these flows and their BMS measures.
In this paper, we obtain mixing and equidistribution results when the BMS measure is finite and the length spectrum is non-arithmetic.

In the case when the convergence group action is geometrically finite, we are able to establish a stronger equidistribution result which will allow us to obtain counting results. 
We apply our abstract results to obtain equidistribution and counting results for relatively Anosov groups in semisimple Lie groups
which generalize earlier work of Sambarino \cite{sambarino-quantitative,sambarino-dichotomy} for Anosov groups.
In the relatively Anosov setting such counting results were previously known only for rank-one semisimple Lie groups, or for special classes of Anosov groups, see the discussion after Corollary \ref{cor:counting for rel Anosovs in intro} for more details.

\subsection{Main results}

We will assume through the paper that $M$ is a compact metrizable space and \hbox{$\Gamma\subset \mathsf{Homeo}(M)$} is a non-elementary convergence group. 
A continuous cocycle $\sigma \colon \Gamma \times M \to \Rb$ defines a natural  \emph{period}
$$
\ell_\sigma(\gamma) := \begin{cases} 
\sigma(\gamma, \gamma^+) & \text{if } \gamma \text{ is loxodromic} \\
0 & \text{otherwise}
\end{cases}.
$$ 
One of the primary aims of this paper is to study counting results for this quantity. To that end, we restrict our study to the case when $\sigma$ is {\em proper},
that is 
$$
\lim_{n \to \infty} \ell_\sigma(\gamma_n)=+\infty
$$
whenever $\{\gamma_n\} \subset \Gamma$ is a sequence of distinct loxodromic elements whose repelling / attracting fixed points $\{ (\gamma_n^-, \gamma_n^+)\}$ are relatively compact in $M^{(2)} \subset M \times M$, the space of distinct pairs.

We also assume our cocycle is part of a {\em continuous  Gromov--Patterson--Sullivan (GPS) system}, which is a triple $(\sigma, \bar{\sigma}, G)$ where $\sigma, \bar{\sigma} \colon \Gamma \times M \to \Rb$
are continuous proper cocycles  and $G\colon M^{(2)} \to \Rb$ is a continuous function such that
\begin{equation}\label{eqn:GPS property}
\bar\sigma(\gamma,x)+\sigma(\gamma,y) = G(\gamma x,\gamma y)-G(x,y)
\end{equation}
for all $\gamma \in \Gamma$ and $(x,y) \in M^{(2)}$. Many geometrically important cocycles appear as part of a GPS system, see~\cite{BCZZ} for examples. 

When $\sigma$ is part of a GPS system, there is a natural way to (coarsely) define the \emph{$\sigma$-magnitude} $\norm{\gamma}_\sigma$ of an element $\gamma \in \Gamma$ and then define a \emph{critical exponent} by
$$
\delta_\sigma(\Gamma) : = \limsup_{R \to \infty} \frac{1}{R} \log \#\{ \gamma \in \Gamma : \norm{\gamma}_\sigma \leq R\},
$$
see Section~\ref{sec:background on PS} for details. Further, we define the {\em length spectrum} as 
$$\mathcal L(\sigma,\bar\sigma,G):=\{\ell_\sigma(\gamma)+\ell_{\bar\sigma}(\gamma): \gamma\in\Gamma\text{ is loxodromic}\}$$ and we say that $\mathcal L(\sigma,\bar\sigma,G)$ is {\em non-arithmetic} if it generates a dense subgroup of $\mathbb R$.

In the case when $\Gamma$ is geometrically finite we obtain the following counting result for periods of elements in $[\Gamma_{\lox}]^w$, the set of weak conjugacy classes of loxodromic elements (defined in Section~\ref{sec:equi for geom finite} below). 
If $\Gamma$ is torsion-free, then $[\Gamma_{\lox}]^w$ is simply the set of conjugacy classes of loxodromic elements. 

\begin{theorem}[Corollary~\ref{cor:counting for geomfin actions}]\label{thm:counting for geomfin actions}
Suppose $(\sigma, \bar{\sigma}, G)$ is a continuous GPS system for a geometrically finite convergence group $\Gamma\subset\mathsf{Homeo}(M)$ where $\delta:=\delta_\sigma(\Gamma)<+\infty$. If 
\begin{enumerate}
\item $\mathcal L(\sigma,\bar\sigma,G)$ is non-arithmetic and 
\item  $\delta_\sigma(P)<\delta$ for all maximal parabolic subgroups $P \subset \Gamma$, 
\end{enumerate}
then
$$
\#\{[\gamma]^w\in [\Gamma_{\lox}]^w : 0<\ell_\sigma(\gamma) \leq R\} \sim \frac{e^{\delta R}}{\delta R},
$$
i.e.\ the ratio of the two sides goes to 1 as $R \to +\infty$.
\end{theorem}

Theorem \ref{thm:counting for geomfin actions}, in combination with the main result of \cite{CZZ4}, will allow us to prove counting results for relatively Anosov representations.
We will describe these results in Section~\ref{subsec:rel anosov} below after developing the appropriate terminology to state them.

\begin{remark}
The condition that $\mathcal{L}(\sigma,\bar\sigma,G)$ is non-arithmetic holds whenever  $\Gamma$ contains a parabolic element. Moreover, it can be replaced with the
weaker assumption that the cross ratio spectrum $\mathcal{CR}$, see Section~\ref{sec:mixing}, is non-arithmetic. This weaker assumption holds whenever some
infinite path component of the limit set $\Lambda(\Gamma)$ contains a conical limit point, see Section~\ref{sec:NALS}.
In particular, in the geometrically finite case, the cross ratio spectrum is non-arithmetic whenever $\Gamma$ is not a virtually free uniform convergence group. 
\end{remark}

Motivated by previous work concerning specific GPS systems, the proof of Theorem~\ref{thm:counting for geomfin actions} is based on studying the dynamical properties of the  flow space associated to the GPS system.
We briefly introduce this flow space and state the dynamical results we prove. For more precise definitions, see Section~\ref{sec:background}.

Let $\Lambda(\Gamma) \subset M$ denote the limit set of $\Gamma$ and let $\Lambda(\Gamma)^{(2)}$ denote the space of distinct pairs in $\Lambda(\Gamma)$. 
One may define a flow space 
$$\tilde U_\Gamma:=\Lambda(\Gamma)^{(2)}\times\mathbb R$$
with flow 
$$
\psi^t(x,y,s)=(x,y,s+t),
$$
which corresponds via the Hopf parametrisation to the (nonwandering part of  the) unit tangent bundle of $\mb H^n$ when $\Gamma\subset\Isom(\mb H^n)$, $M=\partial\mb H^n$ and $\sigma$ is the Busemann cocycle from the hyperbolic metric.
Using this analogy, one then defines an action of $\Gamma$ on $\widetilde U_\Gamma$ by
$$\gamma(x,y,s)=(\gamma(x),\gamma(y), s+\sigma(\gamma,y)).$$ 
When $\sigma$ is part of a GPS system, the action of $\Gamma$ on $\tilde U_\Gamma$ is properly discontinuous (see \cite[Prop.\,10.2]{BCZZ}) and commutes with the flow, so $\psi^t$ descends to a flow
on the quotient
$$U_\Gamma:=\Gamma\backslash \tilde U_\Gamma$$
which we also denote $\psi^t$.

In the case of GPS systems associated to Anosov groups this flow space was constructed by Sambarino~\cite{sambarino-quantitative,sambarino15,sambarino-dichotomy} and in the case of GPS systems associated to transverse groups this flow space was constructed by Kim--Oh--Wang~\cite{KOW}.
In the general setting considered here, one can easily adapt the proof of~\cite[Th.\,9.1]{KOW} to show that $\Gamma$ acts properly discontinuously on $\tilde U_\Gamma$.

In our previous paper \cite{BCZZ}, we showed that if the Poincar\'e series associated to $\sigma$, given by
$$Q_\sigma(s)=\sum_{\gamma\in\Gamma} e^{-s\norm{\gamma}_\sigma},$$
diverges at the critical exponent $\delta_\sigma(\Gamma)$, then $U_\Gamma$ has a unique \emph{(BMS) Bowen--Margulis--Sullivan measure} $m_\Gamma$ and the flow 
$\psi^t \colon (U_\Gamma, m_\Gamma) \to (U_\Gamma, m_\Gamma)$ is ergodic and conservative (see Section~\ref{subsec:background on GPS systems} for the definition of $m_\Gamma$ and precise statements). 

In this paper, we further show that the flow is mixing when the BMS measure is finite and the length spectrum is non-arithmetic.

\begin{theorem}[see Theorem~\ref{thm:mixing} below]
\label{main mixing}
Suppose $(\sigma, \bar{\sigma}, G)$ is a continuous GPS system for a convergence group $\Gamma\subset\mathsf{Homeo}(M)$ where $\delta:=\delta_\sigma(\Gamma) < +\infty$ and $Q_\sigma(\delta) = +\infty$. 
If the BMS measure  $m_\Gamma$ is finite and the length spectrum $\mathcal{L}(\sigma,\bar\sigma,G)$  is non-arithmetic, then
the flow $\psi^t \colon (U_\Gamma,m_\Gamma) \to (U_\Gamma, m_\Gamma)$ is mixing. 
\end{theorem}

In the geometrically finite setting, we are able to adapt arguments of Dal'bo--Otal--Peign\'e \cite{DOP} in the setting of
geometrically finite negatively curved manifolds to provide the following criterion for when the BMS measure to be finite. 

\begin{theorem}[see Theorem~\ref{BMS finite criterion} below]\label{thm:BMS finite criterion in intro}
Suppose  $(\sigma, \bar{\sigma}, G)$ is a continuous GPS system for a geometrically finite convergence group $\Gamma\subset\mathrm{Homeo}(M)$ where $\delta:=\delta_\sigma(\Gamma) < +\infty$.
If $\delta_\sigma(P)<\delta$ for any maximal parabolic subgroup $P$ of $\Gamma$, then $Q_\sigma(\delta) = +\infty$ and the BMS measure $m_\Gamma$  is finite.
\end{theorem}

As an application of mixing, we prove an equidistribution result for fixed points of loxodromic elements in terms of the Patterson--Sullivan measures associated to the cocycles in a GPS system. 
In our earlier work~\cite{BCZZ}, we proved that Patterson--Sullivan measures exist in the critical dimension and are unique when the Poincar\'e series diverges at its critical exponent, see Section~\ref{sec:background on PS} for details.
In the statement below,  $\Gamma_{\rm lox}$ denotes the set of loxodromic elements of $\gamma$ and
$\mathcal D_{\gamma^\pm}$ denotes the unit Dirac mass on $M^{(2)}$ based at the attracting/repelling fixed point of $\gamma\in\Gamma_{\rm lox}$.

\begin{theorem}[see Theorem~\ref{thm:equidistribution in paper} below] \label{thm:equidistribution}
Suppose $(\sigma, \bar{\sigma}, G)$ is a continuous GPS system for a convergence group $\Gamma\subset\mathsf{Homeo}(M)$ where $\delta:=\delta_\sigma(\Gamma) < +\infty$ and $Q_\sigma(\delta) = +\infty$. 
Let $\mu$ be the unique $\sigma$-Patterson--Sullivan measure of dimension $\delta$ and let $\bar\mu$ be the unique $\bar\sigma$-Patterson--Sullivan measure of dimension $\delta$.

If the BMS measure $m_\Gamma$ is finite and mixing, then
 \[
 \lim_{T \to \infty} \delta e^{-\delta T} \sum_{\substack{\gamma \in \Gamma_{\lox}\\ \ell_\sigma(\gamma)\leq T}}\mc D_{\gamma^-}\otimes \mc D_{\gamma^+} = \frac{1}{\norm{m_\Gamma}} e^{\delta G(x,y)}\bar\mu(x)\otimes \mu(y)
 \]
  in the dual of compactly supported continuous functions.
\end{theorem}

The above theorem can be expressed as an equidistribution result for closed orbits of the geodesic flow on $U_\Gamma$.
For every $R$, let $\tilde m_R$ be the sum of Lebesgue measures on axes of loxodromic elements of $\Gamma$ with period at most $R$, 
which is a locally finite measure on $M^{(2)}\times \Rb$.  
Denote by $m_R$ the quotient measure  on $U_\Gamma$.
Then the conclusion of Theorem~\ref{thm:equidistribution} can be reformulated as
$$
 \lim_{R \to \infty} \delta e^{-\delta R}\int fdm_R = \frac{1}{\norm{m_\Gamma}} \int fdm_\Gamma$$
 for any continuous function $f \colon U_\Gamma \to \Rb$ with compact support.

In the context of geometrically finite convergence groups, we establish the following stronger equidistribution result, which is needed to obtain our counting results.

\begin{theorem}[see Theorem~\ref{GF equidistribution} below]\label{GF equidistribution in intro}
Suppose $(\sigma, \bar{\sigma}, G)$ is a continuous GPS system for a geometrically finite convergence group  $\Gamma\subset\mathsf{Homeo}(M)$ where
$\delta : = \delta_\sigma(\Gamma)<+\infty$. If
\begin{enumerate}
\item $\delta_\sigma(P)<\delta$ for all maximal parabolic subgroups $P \subset \Gamma$, and
\item $\mathcal L(\sigma,\bar\sigma,G)$ is non-arithmetic,
\end{enumerate}
 then 
 $$
 \lim_{R \to \infty} \delta e^{-\delta R}\int fdm_R = \frac{1}{\norm{m_\Gamma}} \int fdm_\Gamma$$
 for any bounded continuous function $f \colon U_\Gamma \to \Rb$. 
\end{theorem}

The counting result in Theorem~\ref{thm:counting for geomfin actions} is obtained by applying a slightly modified version of Theorem \ref{GF equidistribution in intro} (Proposition \ref{prop:GF equidistribution'})
to the constant function with value 1.

\subsection{Applications to relatively Anosov groups}\label{subsec:rel anosov} 

We now develop the terminology necessary to explain how to obtain counting result for relatively Anosov subgroups of $\mathsf{SL}(d,\Rb)$ from Theorem \ref{thm:counting for geomfin actions}.
In Section~\ref{sec:semisimple Lie groups} we describe the same result  in the more general setting of semisimple Lie groups. 

Let $\mathfrak{sl}(d,\Rb) = \mathfrak k \oplus \mathfrak p$ denote the standard Cartan decomposition of the Lie algebra of $\SL(d,\Rb)$, 
where  $\mathfrak k$ is the Lie algebra of $\mathsf{SO}(d)$ and $\mathfrak p$ consists of symmetric matrices with trace zero.
 Let $\mathfrak{a} \subset \mathfrak{p}$ denote the standard Cartan subalgebra consisting of diagonal matrices with trace zero and let 
$$
\mathfrak{a}^+ = \{ {\rm diag}(a_1,\dots, a_d) \in \mathfrak a : a_1 \geq a_2 \geq \cdots \geq a_d \} 
$$
denote the standard choice of positive Weyl chamber. The associated simple roots are 
$$
\Delta = \{ \alpha_1,\dots, \alpha_{d-1}\}
$$
where $\alpha_j({\rm diag}(a_1,\dots, a_d)) = a_j - a_{j+1}$. The \emph{Cartan projection} $\kappa \colon \SL(d,\Rb) \to \mathfrak{a}^+$ is given by 
$$
\kappa(\gamma) = {\rm diag}\big( \log \sigma_1(\gamma), \dots, \log \sigma_d(\gamma)\big)
$$
where $\sigma_1(\gamma) \geq \cdots \geq \sigma_d(\gamma)$ are the singular values of $\gamma$ and the  \emph{Jordan projection} $\lambda \colon \SL(d,\Rb) \to \mathfrak{a}^+$ is given by 
$$
\lambda(\gamma) = {\rm diag}\big( \log \lambda_1(\gamma), \dots, \log \lambda_d(\gamma)\big)
$$
where $\lambda_1(\gamma) \geq \cdots \geq \lambda_d(\gamma)$ are the moduli of the generalized eigenvalues of $\gamma$.

Given $\phi \in \mathfrak{a}^*$, one can define the \emph{$\phi$-period} of $\gamma \in \SL(d,\Rb)$ by 
$$
\ell^\phi(\gamma) := \phi(\lambda(\gamma))
$$
and the \emph{$\phi$-magnitude} by $\phi(\kappa(\gamma))$. Also, given a discrete subgroup $\Gamma \subset \SL(d,\Rb)$ and $\phi \in \mathfrak{a}^*$, one can define a (possibly infinite) critical exponent 
$$
\delta^\phi(\Gamma) : = \limsup_{R \to \infty} \frac{1}{R} \log \#\{ \gamma \in \Gamma : \phi(\kappa(\gamma)) \leq R\}. 
$$

Given a subset $\theta \subset \Delta$, let $\Fc_\theta$ denote the partial flag manifold associated to $\theta$, i.e.\ $\Fc_\theta$ is the set of partial flags with subspaces of dimensions 
$\{ j : \alpha_j \in \theta\}$. When $\theta$ is symmetric (i.e.\ $\alpha_j \in \theta \Leftrightarrow \alpha_{d-j} \in \theta$), a discrete subgroup $\Gamma\subset \mathsf{SL}(d,\Rb)$ is {\em $\mathsf{P}_\theta$-relatively Anosov} if $\Gamma$ (as an abstract group) is relatively hyperbolic with respect to a finite collection $\mathcal P$ of finitely generated subgroups of $\Gamma$ and there exists a $\Gamma$-equivariant embedding of the
Bowditch boundary $\partial(\Gamma,P)$ into $\Fc_\theta$ with good dynamical properties.

Associated to $\theta \subset \Delta$ is a natural subspace of $\mathfrak a$ defined by 
$$\mathfrak{a}_\theta:=\{a\in\mathfrak{a} : \beta(a)=0\text{ if } \beta\in\Delta-\theta\}.$$  Then $\mathfrak{a}_\theta^*$ is generated by $\{\omega_j|_{\mathfrak{a}_\theta} :\alpha_j\in \theta\}$ where $\omega_j \in \mathfrak{a}^*$ is the fundamental weight associated to $\alpha_j$ and satisfies 
$$
\omega_j({\rm diag}(a_1,\dots, a_d)) = a_1+\dots + a_j. 
$$
Hence we can identify $\mathfrak{a}_\theta^*$ as a subspace of $\mathfrak a^*$.

In Section~\ref{sec:semisimple Lie groups} we will show that if $\Gamma\subset \mathsf{SL}(d,\Rb)$ is $\mathsf{P}_\theta$-relatively Anosov, $\phi \in \mathfrak a_\theta^*$ and $\delta^\phi(\Gamma) < +\infty$, then there exists a GPS system on the limit set of $\Gamma$ whose periods are exactly the $\phi$-lengths defined above (in fact we will show that this is more generally true for the wider class of $\mathsf P_\theta$-transverse groups in semisimple Lie groups).
We then use results in~\cite{CZZ4} to verify that the conditions of Theorem~\ref{thm:counting for geomfin actions} are satisfied and obtain the following counting result. 

\begin{corollary}[see Section~\ref{sec:semisimple Lie groups} below]\label{cor:counting for rel Anosovs in intro}
Suppose $\theta \subset \Delta$ is symmetric and \hbox{$\Gamma\subset \mathsf{SL}(d,\Rb)$} is  $\mathsf{P}_\theta$-relatively Anosov with respect to $\mathcal P$. If 
$\phi\in\mathfrak{a}_\theta^*$ and $\delta:=\delta^\phi(\Gamma)<+\infty$, then
$$
\#\{[\gamma]^w\in [\Gamma_{\lox}]^w :0<\phi(\lambda(\gamma))  \leq R\} \sim \frac{e^{\delta R}}{\delta R}.
$$
\end{corollary}

Corollary \ref{cor:counting for rel Anosovs in intro} was previously known only when $\ms G$ is rank-one (see Roblin~\cite{roblin}), when 
$\Gamma$ is the image of a relatively Anosov representation of a finitely generated torsion-free Fuchsian group \cite{BCKM}, and when $\Gamma$ is a $\{\alpha_1,\alpha_{d-1}\}$-Anosov group of $\PGL(d,\R)$ preserving a properly convex domain 
in $\mathbb P(\mathbb R^d)$ and $\phi=\omega_1+\omega_{d-1}$ \cite{BZ,zhu-ergodicity}.

\begin{remark} 
The concurrent work of Kim and Oh \cite{KO24} contains a proof, using different techniques, that the BMS measure associated to a relatively Anosov group is finite and that the flow is mixing with respect to BMS measure in this setting. 
\end{remark}

\subsection{Historical remarks}\label{sec:historial remarks}

There is a long history of using dynamical methods to obtain asymptotic counting results for the number of closed orbits of a flow.
The first counting result of the form of Theorem \ref{thm:counting for geomfin actions} was established by Huber \cite{huber} for cocompact Fuchsian groups, as an application of Selberg's trace formula.
In his Ph.D.\ thesis, Margulis \cite{margulis} established mixing, counting and equidistribution results for negatively curved manifolds, and his proofs provide the template for much subsequent work.
More generally, Margulis' work applies to all flows $\phi^t$ on closed manifolds $M$ that are \emph{Anosov}, i.e.\ for which there is an invariant splitting of the tangent bundle of $M$ into the flow direction, 
a ``stable'' sub-bundle which is exponentially contracted by the flow, and a ``unstable'' sub-bundle which is exponentially dilated.

Margulis' approach, combined with the theory of Patterson--Sullivan measures \cite{Patterson,sullivan}, was used in Roblin's work on (not necessarily cocompact) discrete isometry groups of ${\rm CAT}(-1)$ spaces \cite{roblin}.
This is the approach we use here.
In the context of discrete subgroups of Lie groups, which is our main application, this method was also used to obtain counting results for certain subgroups of $\PGL(d,\R)$ 
preserving convex domains of $\mb P(\R^d)$: these results use the geometry of these domains, and are for a certain length function called the Hilbert length~\cite{zhu-ergodicity,BZ}.
The combination of Margulis's ideas and Patterson--Sullivan theory was also used by Ricks \cite{ricks}, building on earlier work of Knieper \cite{Knieper}, to prove counting results for groups acting on rank-one ${\rm CAT}(0)$ spaces.
Note that these convex projective and rank-one ${\rm CAT}(0)$ settings are not encompassed by the present work, as the natural boundary action in these cases are not necessarily convergence actions, due to the presence of flats.

Related counting results can also be obtained using local mixing results. Chow--Sarkar \cite{chow-sarkar} establish such results in the Anosov case, which were used to obtain counting results in affine symmetric spaces by Edwards--Lee--Oh \cite{ELO}.

Turning to other dynamical approaches: Parry and Pollicott \cite[Th.\,2]{PP83} used symbolic dynamics and the thermodynamic formalism to establish counting results for Axiom A flows, which are generalizations of Anosov flows to compact spaces other than manifolds,
and Pollicott extended this to an even more general class of flows on compact spaces called metric Anosov or Smale flows~\cite[Th.\,8]{Pollicott87}.
Sambarino \cite{sambarino-quantitative,sambarino15} established counting, mixing and equidistribution results for Anosov groups isomorphic to the fundamental groups of
negatively curved manifolds $M$, by showing that the length functions coming from Anosov groups corresponds to periods of a reparametrization of the geodesic flow on $T^1M$, and this reparametrization is a metric Anosov flow, where counting results are already known.
With additional results established by Bridgeman--Canary--Labourie--Sambarino \cite{BCLS}, Sambarino's arguments generalize to all Anosov groups.  A similar idea was used earlier by Benoist \cite[Cor.\,5.7]{CD1} in the setting of Anosov representations acting cocompactly on strictly convex projective domains. Later Chow--Fromm~\cite{chow-fromm} established a general result that implies Sambarino's counting result for Anosov groups.

Lalley \cite{lalley} used symbolic dynamics and a renewal theorem to establish counting and
equidistribution results for convex cocompact Fuchsian groups. 
This approach was generalized by Dal'bo and Peign\'e \cite{dalbo-peigne}
to the setting of geometrically finite negatively curved manifolds with free fundamental group.
This approach was later implemented in Bray--Canary--Kao--Martone \cite{BCKM} for relatively Anosov representations of finitely generated torsion-free Fuchsian groups.

Finally: stronger quantitative mixing results can be used to obtain finer counting estimates. 
Recently, Delarue--Monclair--Sanders~\cite{DMS} obtained exponential mixing results in the $\{\alpha_1,\alpha_{d-1}\}$-Anosov case, yielding a counting estimate with an exponential error term.

\subsection{Outline of paper}

In Section~\ref{sec:background}, we recall some results from the theory of convergence group actions and also some results from \cite{BCZZ} about GPS systems. 

The first part of the paper, Sections \ref{sec:stable and unstable manifolds}, \ref{sec:mixing} and \ref{sec:equi},
follows a classical strategy to prove equidistribution of closed geodesics (Theorem~\ref{thm:equidistribution}), combining Margulis' ideas and Patterson--Sullivan theory.
This portion of the work relies heavily, and follows fairly quickly, from the machinery developed in \cite{BCZZ}.
In Section~\ref{sec:NALS}, we adapt several classical criteria for non-arithmeticity of length spectrum to our setting.

In Section~\ref{sec:stable and unstable manifolds}, we construct natural stable/unstable manifolds in $U_\Gamma$.
In Section~\ref{sec:mixing}, we use these stable/unstable manifolds and Coud\`ene's criterion for mixing \cite{coudene-mixing} to establish mixing (under a non-arithmeticity assumption).
Finally, in Section~\ref{sec:equi}, we combine the mixing property, the explicit product structure of the BMS measure, and a closing lemma (Lemma~\ref{lem:closing}), to establish equidistribution of closed geodesics.
In contrast to Roblin's \cite{roblin} work in the setting of ${\rm CAT}(-1)$ spaces, we do not prove nor use an equidistribution of double $\Gamma$-orbits.

One notable difference with previous works concerns our construction of stable and unstable manifolds.
Unlike previous settings such as those of ${\rm CAT}(-1)$ spaces or Hilbert geometries (see \cite[Fact\,6.11 \& Prop.\,6.14]{blayac-ps}), we do not have a natural metric on the flow space $U_\Gamma$ to use to check that our algebraic definition matches Coud\`ene's metric definition.
We address this issue by observing that in fact any metric on the one-point compactification of $U_\Gamma$ does the job.

The second part of the paper specializes to geometrically finite groups, and is more involved than the first.
To be able to apply the results of the first part to geometrically finite groups,
we first show in Section~\ref{sec:structure of the flow space} that $U_\Gamma$ decomposes into a union of a compact part and finitely many cusp-like parts, which are quotients of horoballs-like sets in $\tilde U_\Gamma$. Using this decomposition,  in Section~\ref{sec:finite BMS} we prove Theorem~\ref{thm:BMS finite criterion in intro}, which gives a criterion guaranteeing that geometrically finite groups are divergent and have finite BMS measure. In Section~\ref{sec:equi for geom finite}, we prove Theorem~\ref{GF equidistribution in intro}. As in  Roblin \cite[Th.\,5.2]{roblin}, the proof is based on estimating the measure of ``cusps'' in $U_\Gamma$ with respect to the measure $m_R$ appearing in the statement of the theorem. 

Finally, in Section~\ref{sec:semisimple Lie groups} we explain how to apply our results in  the setting of transverse and  relatively Anosov subgroups of semisimple Lie groups.


\section{Background}\label{sec:background} 


\subsection{Convergence group actions}
\label{convergence groups}

When $M$ is a compact metrizable space, a subgroup $\Gamma \subset \mathsf{Homeo}(M)$ is called a  (discrete) \emph{convergence group} if for every sequence 
$\{\gamma_n\}$ of distinct elements in $\Gamma$, there exist points $x,y \in M$ and a subsequence $\{\gamma_{n_j}\}$  such that $\gamma_{n_j}|_{M \smallsetminus \{y\}}$ converges locally uniformly to $x$. 
This notion was first introduced by Gehring and Martin \cite{GM87}, and then further developed in \cite{tukia-convergence,Bowditch99}.
Recall the following classification of elements of $\Gamma$.

\begin{lemma}[{\cite[Th.\,2B]{tukia-convergence}}]
\label{fact:classif loxparell}
If $\Gamma \subset \mathsf{Homeo}(M)$ is a convergence group, then every element $\gamma\in \Gamma$ is one of
 \begin{itemize}
  \item \emph{loxodromic}: it has two fixed points $\gamma^+$ and $\gamma^-$ in $M$ such that $\gamma^{\pm n}|_{M \smallsetminus \{\gamma^{\mp}\}}$ converges locally uniformly to $\gamma^{\pm}$, 
  \item \emph{parabolic}: it has one fixed point $p\in M$ such that $\gamma^{\pm n}|_{M \smallsetminus \{p\}}$ converges locally uniformly to $p$, or 
  \item \emph{elliptic}: it has finite order.
 \end{itemize}
\end{lemma}

Given a convergence group we define the following (cf.\ \cite{Bowditch99,tukia-convergence}):
\begin{enumerate}
\item \label{item:limit set} The \emph{limit set} $\Lambda(\Gamma)$ is the set of points $x \in M$ for which there exist $y \in M$ and a sequence $\{\gamma_n\}$ in $\Gamma$ so that $\gamma_n|_{M \smallsetminus \{y\}}$ converges locally uniformly to $x$.
(Note that fixed points of non-elliptic elements of $\Gamma$ are in the limit set.)
\item A point $x \in \Lambda(\Gamma)$ is a \emph{conical limit point} if there exist distinct points $a,b \in M$ and a sequence of elements $\{\gamma_n\}$ in $\Gamma$ where $\lim_{n \to \infty} \gamma_n(x) = a$ and $\lim_{n \to \infty} \gamma_n(y) = b$ for all $y \in M \smallsetminus \{x\}$. 
\item A point $p \in \Lambda(\Gamma)$ is a \emph{bounded parabolic point} if no element of
$$\Gamma_p:={\rm Stab}_\Gamma(p)$$
is loxodromic and  $\Gamma_p$ acts cocompactly on $\Lambda(\Gamma) \smallsetminus \{p\}$.

\end{enumerate}
We say that a convergence group $\Gamma$ is \emph{non-elementary} if $\Lambda(\Gamma)$ contains at least 3 points. 
In this case $\Lambda(\Gamma)$ is the smallest $\Gamma$-invariant closed subset of $M$ (see \cite[Th.\,2S]{tukia-convergence}). Finally, a non-elementary convergence group $\Gamma$ is \emph{geometrically finite} if every point in $\Lambda(\Gamma)$ is either a conical limit point or a bounded parabolic point. 
The stabilizers of the bounded parabolic points are called the \emph{maximal parabolic subgroups} of $\Gamma$.

{\bf For the remainder of our paper, we will assume all of our convergence groups are non-elementary.}

In the rest of this section, we recall several results about convergence groups that we will need later. We first recall a closing lemma due to Tukia.

 \begin{lemma}[{Tukia \cite[Cor.\,2E]{tukia-convergence}}]\label{lem:char of loxodromic} 
 Suppose $\Gamma \subset \mathsf{Homeo}(M)$ is a convergence group.
 If $\{\gamma_n\} \subset \Gamma$ is a sequence where 
 $\gamma_{n}|_{M \smallsetminus \{b\}}$ converges locally uniformly to $a$  and $a \neq b$, then for $n$ sufficiently large $\gamma_n$ is loxodromic,
 $\gamma_n^+\to a$ and $\gamma_n^- \to b$. 
\end{lemma}

Lemmas~\ref{fact:classif loxparell} and \ref{lem:char of loxodromic} have the following immediate  consequence.

\begin{lemma}[{\cite[Prop.\,3.2]{Bowditch99}}]\label{lem: parab conical disjoint}
If $\Gamma \subset \mathsf{Homeo}(M)$ is a convergence group, and
the stabilizer of $x\in M$ is infinite and contains no loxodromic elements, then $x$ is not a conical limit point.
In particular,  parabolic fixed points are not conical limit points.
\end{lemma}

If $\Gamma$ is geometrically finite, there are only finitely many conjugacy classes of subgroups stabilizing a parabolic fixed point.

\begin{lemma}[{\cite[Th.\,1B]{tukia-conical}}]\label{lem:finitely many parabs}
 If $\Gamma \subset \mathsf{Homeo}(M)$ is a geometrically finite convergence group, then there are finitely many $\Gamma$-orbits of parabolic fixed points.
\end{lemma}

We also will use the following well-known co-compactness result.

\begin{proposition} \label{prop:exists compact covering set in limit cross limit}
 If $\Gamma \subset \mathsf{Homeo}(M)$ is a geometrically finite convergence group, then there exists a compact subset $K \subset \Lambda(\Gamma)^{(2)}$ such that $\Gamma \cdot K =  \Lambda(\Gamma)^{(2)}$. 
 \end{proposition} 
 
 \begin{proof} By~\cite[Th.\,1C]{tukia-conical} there exists a compact set $K_0 \subset \Lambda(\Gamma)^{(2)}$ such that if $(x,y) \in \Lambda(\Gamma)^{(2)}$ and at least one of $x,y$ is conical, then $(x,y) \in \Gamma \cdot K_0$. Fix representatives $p_1,\dots,p_k$ of the $\Gamma$-orbits of bounded parabolic points (there are finitely many such orbits by Lemma~\ref{lem:finitely many parabs}). Then for each $1 \leq j \leq k$, fix a compact set $K_j \subset \Lambda(\Gamma) \smallsetminus \{p_j\}$ where $\Gamma_{p_j} \cdot K_j = \Lambda(\Gamma) \smallsetminus \{p_j\}$. Finally let 
 $$
 K : = K_0 \cup \bigcup_{j=1}^k \left( \{p_j\} \times K_j \cup K_j \times \{p_j\} \right) .
 $$
 Then $K$ is compact and $\Gamma \cdot K =  \Lambda(\Gamma)^{(2)}$.
 \end{proof}

We previously observed~\cite{BCZZ} that $M$ can be used to compactify $\Gamma$, see also~\cite{Bowditch99,CZZ3}. 

\begin{definition}\label{defn:compactification} Given a convergence group $\Gamma \subset \mathsf{Homeo}(M)$, a \emph{compactifying topology} on $\Gamma\sqcup M$ is a topology such that:
\begin{itemize} 
\item $\Gamma \sqcup M$ is a compact metrizable space.
\item The inclusions $\Gamma \hookrightarrow \Gamma \sqcup M$ and $M \hookrightarrow \Gamma \sqcup M$ are embeddings (where in the first embedding $\Gamma$ has the discrete topology). 
\item $\Gamma$ acts as a convergence group on $\Gamma \sqcup M$. 
\end{itemize} 
A metric $\dist$ on $\Gamma \sqcup M$ is called \emph{compatible} if it induces a compactifying topology.
\end{definition} 

We observed that compactifying topologies exist, are unique, and have the following properties. 

\begin{proposition}[{\cite[Prop.\,2.3]{BCZZ}}]\label{prop:compactifying} If $\Gamma \subset \mathsf{Homeo}(M)$ is a convergence group, then there exists a unique compactifying topology on $\Gamma\sqcup M$.
Moreover, with respect to this topology the following hold:
\begin{enumerate}

\item If $\{\gamma_n\} \subset \Gamma$ is a sequence where $\gamma_n\to a \in M$ and $\gamma_n^{-1}\to b\in M$, then $\gamma_{n}|_{M \smallsetminus \{b\}}$ converges locally uniformly to $a$. 
\item A sequence $\{\gamma_n\} \subset \Gamma$ converges to $a \in M$ if and only if 
for every subsequence $\{\gamma_{n_j}\}$ there exist $b \in M$ and a further subsequence $\{\gamma_{n_{j_k}}\}$ such that $\gamma_{n_{j_k}}|_{M \smallsetminus \{b\}}$ converges locally uniformly to $a$.

\item $\Gamma$ is open in $\Gamma\sqcup M$ and its closure is $\Gamma\sqcup \Lambda(\Gamma)$.
\end{enumerate}
\end{proposition}

\subsection{Cocycles and Patterson--Sullivan measures}\label{sec:background on PS}

Previously~\cite{BCZZ} we introduced the notion of an expanding cocycle and proved that the cocycles in a GPS system are expanding. In this section we recall the definition and their basic properties.

\begin{definition} Suppose $\Gamma \subset  \mathsf{Homeo}(M)$ is a convergence group and $\dist$ is a compatible distance on $\Gamma \sqcup M$. A continuous cocycle $\sigma \colon \Gamma \times M \to \Rb$ is \emph{expanding} if it is proper and for every $\gamma \in \Gamma$ there is a number $\norm{\gamma}_\sigma \in \Rb$, called the  \emph{$\sigma$-magnitude of $\gamma$}, with the following property:  
\begin{itemize}
\item for every $\epsilon > 0$ there exists $C > 0$ such that: whenever $x \in M$, $\gamma \in \Gamma$ and $\dist(x,\gamma^{-1}) > \epsilon$, then 
$$
\norm{\gamma}_\sigma - C \leq \sigma(\gamma, x) \leq \norm{\gamma}_\sigma+C.
$$ 
\end{itemize}
\end{definition} 

Given an expanding cocycle $\sigma$, the \emph{$\sigma$-critical exponent} is 
$$
\delta_\sigma(\Gamma) = \limsup_{R \to \infty} \frac{1}{R} \log \#\left\{ \gamma \in \Gamma : \norm{\gamma}_\sigma \leq R\right\}\in[0,\infty]. 
$$
Equivalently, $\delta_\sigma(\Gamma)$ is the critical exponent of the series 
$$
Q_\sigma(s) : = \sum_{\gamma \in \Gamma} e^{- s \norm{\gamma}_\sigma},
$$
that is $Q_\sigma(s)$ diverges when $0<s < \delta_\sigma(\Gamma)$ and converges when $s > \delta_\sigma(\Gamma)$. Although the magnitude function is not uniquely defined, any two choices will differ by a uniformly bounded additive amount and hence $\delta_\sigma(\Gamma)$ is independent of the choice of particular magnitude function.

 In previous work, we established the following results about expanding cocycles.

 \begin{proposition}[{\cite[Prop.\,3.2]{BCZZ}}]\label{prop:basic properties} 
Suppose $\Gamma \subset \mathsf{Homeo}(M)$ is a convergence group, $\dist$ is a compatible distance on $\Gamma \sqcup M$ and $\sigma$ is an expanding cocycle, then:
\begin{enumerate}
\item\label{item:tri inequality} For any finite subset $F \subset \Gamma$, there exists $C > 0$ such that: if $\gamma \in \Gamma$ and $f \in F$, then 
$$
\norm{\gamma}_\sigma -C \leq \norm{\gamma f}_\sigma \leq \norm{\gamma}_\sigma + C \quad \text{and} \quad \norm{\gamma}_\sigma -C \leq \norm{f\gamma}_\sigma \leq \norm{\gamma}_\sigma + C.
$$
\item\label{item:properness} $\lim_{n \to \infty} \norm{\gamma_n}_\sigma = + \infty$ for every escaping sequence $\{ \gamma_n\} \subset \Gamma$. 

\item\label{item:proper implies positive periods} $\ell_\sigma(\gamma)=\sigma(\gamma, \gamma^+) > 0$ for all loxodromic $\gamma \in \Gamma$ and $\sigma(\gamma,p)=0$ for any parabolic 
element $\gamma\in\Gamma$ with fixed point $p$. 

\item\label{item:a technical fact}   If $\{\gamma_n\}$ is a divergent sequence in $\Gamma$ and 
$$
\liminf_{n \to \infty} \sigma(\gamma_n, y_n) > -\infty,
$$
then $\dist(\gamma_n y_n, \gamma_n) \to 0$. 

\item\label{item:multiplicative estimate} For any $\epsilon > 0$ there exists $C > 0$ such that: if $\alpha, \beta \in \Gamma$ and $\dist(\alpha^{-1}, \beta) \geq \epsilon$, then 
$$
\norm{\alpha}_\sigma+\norm{\beta}_\sigma-C \leq \norm{\alpha \beta}_\sigma \leq \norm{\alpha}_\sigma + \norm{\beta}_\sigma + C. 
$$
\item\label{item:loxodromics with separated fixed points} For any compact subset $K \subset M^{(2)}$ there exists $C > 0$ such that: if $\gamma \in \Gamma$ is loxodromic and $(\gamma^-, \gamma^+) \in K$, then 
$$
 \norm{\gamma}_\sigma  - C \leq \ell_\sigma(\gamma) \leq \norm{\gamma}_\sigma  + C. 
$$
\end{enumerate}

\end{proposition}

\begin{proof} The only claim that doesn't appear in~\cite[Prop.\,3.2]{BCZZ} is~\eqref{item:loxodromics with separated fixed points}. 
Let 
$$\Gamma_K:=\{\gamma\in\Gamma : \gamma\text{ is loxodromic and }(\gamma^-,\gamma^+)\in K\}.$$
Since $K$ is a compact subset of $M^{(2)}$, there 
exists $\epsilon>0$ so that if  $\gamma\in \Gamma_K$, then $\dist(\gamma^+,\gamma^-)\ge \epsilon$.
The set $F_K$ of elements of $\Gamma_K$ such that $\dist(\gamma^{-1},\gamma^-)\ge \frac{\epsilon}{2}$ is finite (one can check  $\gamma_n^-\to z$ if and only if $\gamma_n^{-1}\to z$).
If $\gamma\in\Gamma_K\smallsetminus F_K$, then $\dist(\gamma^+,\gamma^{-1})\ge \frac{\epsilon}{2}$, so the expanding property implies that there exists $C_0>0$ such that 
$$ \norm{\gamma}_\sigma  - C_0 \leq \ell_\sigma(\gamma)=\sigma(\gamma,\gamma^+) \leq \norm{\gamma}_\sigma  + C_0.$$
Claim \eqref{item:loxodromics with separated fixed points} then follows if we take $C_1:=\max_{\gamma\in F_K}\abs{\norm{\gamma}_\sigma-\ell_\sigma(\gamma)}$ and $C:=\max(C_0,C_1)$.
\end{proof}

Proposition~\ref{prop:exists compact covering set in limit cross limit} implies the following result about the lengths associated to an expanding cocycle. 
 
\begin{corollary}
\label{length spectrum discrete}
 Suppose $\sigma$ is an expanding cocycle for a geometrically finite convergence group $\Gamma \subset \mathsf{Homeo}(M)$. If $T>0$, then
$$\{[\gamma]\in[\Gamma_{\lox}] : \ell_\sigma(\gamma)\le T\}$$
is finite. In particular, the {\em systole} defined by
$$\mathrm{sys}(\Gamma,\sigma):=\min\left\{ \ell_\sigma(\gamma) :\gamma\in \Gamma_{\lox}\right\}>0$$
is well-defined and positive.
\end{corollary} 

\begin{proof} Let $K$ be as in Proposition~\ref{prop:exists compact covering set in limit cross limit}. Then every $[\gamma]\in[\Gamma_{\lox}] $ has a representative $\gamma_0$ where $(\gamma_0^-, \gamma_0^+) \in K$. So by the properness assumption $\{[\gamma]\in[\Gamma_{\lox}] : \ell_\sigma(\gamma)\le T\}$
is finite. 
\end{proof}

Suppose $\Gamma \subset \mathsf{Homeo}(M)$ is a non-elementary convergence group and $\sigma\colon \Gamma \times M \to \Rb$ is a continuous cocycle. 
A probability measure $\mu$ on $M$ is called a \emph{$\sigma$-Patterson--Sullivan} measure of dimension $\beta$ if for every $\gamma \in \Gamma$ the measures 
$\gamma_*\mu$, $\mu$ are absolutely continuous and 
 $$
 \frac{d\gamma_*\mu}{d\mu} = e^{-\beta \sigma(\gamma^{-1}, \cdot)} \quad\mu\text{-almost everywhere}.
 $$

In previous work, we showed if $\sigma$ is an expanding cocycle for a convergence group action $\Gamma \subset \mathsf{Homeo}(M)$ with finite critical exponent 
$\delta_\sigma(\Gamma)$, then there is 
a \hbox{$\sigma$-Patterson--Sullivan measure of dimension $\delta_\sigma(\Gamma)$.} 
We also proved that this measure is unique and ergodic when the $\sigma$-Poincar\'e series diverges at its critical exponent. 

\begin{theorem}[{\cite[Th.\,4.1, Th.\,1.3 and Prop.\,6.3]{BCZZ}}]\label{thm:PS exist and unique}
\label{thm:PS measures exist} If $\sigma$ is an expanding cocycle for a convergence group $\Gamma\subset \mathsf{Homeo}(M)$ and
$\delta:=\delta_\sigma(\Gamma) < + \infty$, then there exists a $\sigma$-Patterson--Sullivan measure of dimension $\delta$, 
which is supported on the limit set $\Lambda(\Gamma)$. Moreover, if 
$$
 Q_\sigma(\delta) = \sum_{\gamma \in \Gamma} e^{- \delta \norm{\gamma}_\sigma} = +\infty,
 $$
 then:
 \begin{enumerate}
 \item there is a unique $\sigma$-Patterson--Sullivan measure $\mu$ of dimension $\delta$,
 \item $\mu$ has no atoms, and
 \item the action of $\Gamma$ on $(M,\mu)$ is ergodic. 
 \end{enumerate}
\end{theorem}

\subsection{GPS systems and flow spaces}\label{subsec:background on GPS systems}

We recall from our previous work~\cite{BCZZ} properties of the cocycles in a GPS system and the flow space associated to them. 

We proved that the cocycles in a GPS system are expanding and that the norms associated to the two cocycles are closely related.

\begin{proposition}[{\cite[Prop.\,3.3]{BCZZ}}]\label{GPS basic properties}
Suppose $(\sigma,\bar\sigma, G)$ is a continuous GPS system for a convergence group $\Gamma\subset \mathsf{Homeo}(M)$.
\begin{enumerate}
\item\label{GPS expanding} $\sigma$ and $\bar\sigma$ are expanding cocycles.
\item\label{item:norm vs dual norm} There exists $C > 0$ such that 
$$
\norm{\gamma^{-1}}_{\bar{\sigma}}-C \leq \norm{\gamma}_\sigma \leq \norm{\gamma^{-1}}_{\bar{\sigma}} + C
$$
for all $\gamma \in \Gamma$. In particular, $\delta_\sigma(\Gamma)=\delta_{\bar\sigma}(\Gamma)$.

\end{enumerate}
\end{proposition}

We also established a version of the Hopf--Tsuji--Sullivan dichotomy. 

\begin{theorem}{\cite[Th.\,1.8 and \S10.4]{BCZZ}}\label{thm:dichtomy from 1st paper} 
\label{our dichotomy}
Suppose $(\sigma, \bar\sigma, G)$ is a continuous GPS system and \hbox{$\delta_\sigma(\Gamma) < +\infty$}. 
Let $\mu$ and $\bar \mu$ be Patterson--Sullivan measures of dimension $\delta$ for $\sigma$ and  $\bar{\sigma}$. Then $\nu:=e^{\delta G}\bar\mu\otimes\mu$ 
is a locally finite $\Gamma$-invariant measure on $M^{(2)}$, and we have the following dichotomy:
\begin{enumerate} 
\item If $\sum_{\gamma \in \Gamma} e^{-\delta \norm{\gamma}_\sigma} = +\infty$, then:
\begin{enumerate}
\item $\delta = \delta_\sigma(\Gamma)$.
\item $\mu( \Lambda^{\rm con}(\Gamma)) = 1 = \bar\mu(\Lambda^{\rm con}(\Gamma))$. 
\item The $\Gamma$ action on $(M^{(2)},\nu)$ is ergodic and conservative. 
\end{enumerate} 
\item If $\sum_{\gamma \in \Gamma} e^{-\delta \norm{\gamma}_\sigma} < +\infty$, then:
\begin{enumerate}
\item $\delta \geq \delta_\sigma(\Gamma)$.
\item $\mu( \Lambda^{\rm con}(\Gamma)) = 0 = \bar\mu(\Lambda^{\rm con}(\Gamma))$. 
\item The $\Gamma$ action on $(M^{(2)}, \nu)$ is non-ergodic and dissipative. 
\end{enumerate} 
\end{enumerate} 

\end{theorem}

Next we carefully describe the flow space associated to a GPS system $(\sigma, \bar{\sigma}, G)$, which was briefly described in the introduction. 
In the case when $M=\partial_\infty X$ is the geodesic boundary of a Hadamard manifold $X$ (simply connected with pinched negative curvature), the group $\Gamma \subset \mathsf{Isom}(X)$ is discrete, and $\sigma$ is the Busemann cocycle, this flow space is topologically conjugate to the non-wandering part of the geodesic flow on the quotient $\Gamma \backslash T^1X$ of the unit tangent bundle $T^1X$ of $X$.

For the rest of the section suppose $(\sigma,\bar\sigma,G)$ is a continuous GPS system for a convergence group $\Gamma\subset\mathsf{Homeo}(M)$. 
As in the introduction, let $\Lambda(\Gamma)^{(2)} \subset \Lambda(\Gamma) \times \Lambda(\Gamma)$ denote the set of distinct pairs and let $\tilde U_\Gamma : = \Lambda(\Gamma)^{(2)} \times \Rb$.

By Proposition~\ref{GPS basic properties}, the cocycles $\sigma$ and $\bar \sigma$ are expanding. Hence, by ~\cite[Prop.\,10.2]{BCZZ},  the action of $\Gamma$ on $\tilde U_\Gamma$ given by 
$$
\gamma(x,y,t) = (\gamma x, \gamma y, t + \sigma(\gamma, y))
$$
is properly discontinuous. 
Therefore $U_\Gamma : = \Gamma \backslash \tilde U_\Gamma$ is a locally compact metrizable space.
Further the flow $\psi^t \colon \tilde U_\Gamma \to \tilde U_\Gamma$ defined by 
$$
\psi^t(x,y,s) = (x,y,s+t)
$$
descends to a flow on $U_\Gamma$, which we also denote by $\psi^t$.

Now suppose, in addition, that $\delta:=\delta_\sigma(\Gamma) < + \infty$ and  $Q_\sigma(\delta) = +\infty$. 
By Theorem~\ref{thm:PS exist and unique} and Proposition~\ref{GPS basic properties} there is a unique $\sigma$-Patterson--Sullivan measure $\mu$ of dimension $\delta$ and a unique $\bar\sigma$-Patterson--Sullivan measure $\bar \mu$ of dimension $\delta$. By Equation~\eqref{eqn:GPS property}, the measure $\tilde m$ on 
$\tilde U_\Gamma=\Lambda(\Gamma)^{(2)} \times \Rb$ defined by
\[ \tilde m := e^{\delta G(x,y)}d\bar\mu(x)\otimes d\mu(y) \otimes dt \]
is flow-invariant and $\Gamma$-invariant. 
So $\tilde m$
descends to a flow-invariant measure $m_\Gamma$ on the quotient $U_\Gamma=\Gamma \backslash \tilde U_\Gamma$ (see Section~\ref{sec:quotient measures} for the definition of quotient measures). 
We refer to $m_\Gamma$ as the \emph{Bowen--Margulis--Sullivan (BMS) measure} associated to $(\sigma,\bar\sigma, G)$.
Using Theorem~\ref{our dichotomy}, we can show that the flow is ergodic and conservative with respect to its Bowen--Margulis--Sullivan measure in this setting~\cite[Th.\,11.2]{BCZZ}.

\subsection{Quotient measures}\label{sec:quotient measures}
In this short expository section we review properties of quotient measures. 
Suppose $X$ is a proper metric space and $\tilde\nu$ is a locally finite Borel measure on $X$. 
Assume $\Gamma$ is a discrete group which acts properly on $X$ and preserves the measure $\tilde\nu$. 
Given a measurable function $f \colon X \to [0,+\infty]$, define $P(f)\colon \Gamma \backslash X \to [0,+\infty]$ by 
$$
P(f)([x]) = \sum_{\gamma \in \Gamma} f(\gamma x). 
$$
Then the quotient space $\Gamma \backslash X$ has a unique Borel measure $\nu$ such that 
\begin{equation}\label{eqn:defining function of quotient measures}
\int_{\Gamma \backslash X} P(f) d \nu = \int_X f d\tilde \nu. 
\end{equation}
for all measurable functions $f \colon X \to [0,+\infty]$. The existence of such a measure is classical, see also the discussion in our earlier paper~\cite[Appendix A]{BCZZ}. 

Next suppose that $\phi^t \colon X \to X$ is a measurable flow which commutes with the $\Gamma$ action and preserves the measure $\tilde\nu$. 
Then $\phi^t$ descends to a flow on the quotient $\Gamma \backslash X$ which we also denote by $\phi^t$ and preserves $\nu$. 
We recall that  $\phi^t \colon (\Gamma \backslash X, \nu) \to (\Gamma \backslash X, \nu)$ is {\em mixing} if $\norm{\nu}:=\nu(\Gamma \backslash X)$ is finite and whenever $A$ and $B$ are measurable subsets  of $\Gamma \backslash X$,
we have 
$$\lim_{t \to \infty} \nu(A \cap \gamma \phi^t(B))= \frac{\nu(A) \nu(B)}{\norm{\nu}}. $$
The following observation interprets this in terms of the flow on $X$.

\begin{observation}\label{obs:consequence of mixing} Suppose $\norm{\nu}=\nu(\Gamma \backslash X) <+\infty$ and $\phi^t \colon (\Gamma \backslash X, \nu) \to (\Gamma \backslash X, \nu)$ is mixing. If $A,B \subset X$ have finite $\tilde \nu$-measure, then 
$$
\lim_{t \to \infty} \sum_{\gamma \in \Gamma} \tilde \nu(A \cap \gamma \phi^t(B)) = \frac{ \tilde \nu(A)\tilde \nu(B)}{\norm{\nu}}. 
$$
\end{observation} 

\begin{proof} Notice that 
$$
P(f) \cdot P(g) =  \sum_{\gamma \in \Gamma} P\big( f \cdot (g \circ \gamma)\big). 
$$
Hence 
\begin{align*}
 \sum_{\gamma \in \Gamma}&  \tilde \nu(A \cap \gamma \phi^{-t}(B)) = \sum_{\gamma \in \Gamma} \int 1_A \cdot (1_B \circ \phi^t \circ \gamma^{-1})d \tilde \nu \\
& = \sum_{\gamma \in \Gamma} \int P\big( 1_A \cdot (1_B \circ \phi^t \circ \gamma^{-1})\big) d \nu = \int P(1_A) \cdot ( P(1_B) \circ \phi^t ) d \nu
\end{align*}
which tends to $\dfrac{\int P(1_A)d\nu \int P(1_B) d\nu}{\norm{\nu}} = \dfrac{ \tilde \nu(A)\tilde \nu(B)}{\norm{\nu}}$ as $t\to\infty$.
\end{proof}


\section{Stable and unstable manifolds}\label{sec:stable and unstable manifolds}


Given a flow $\phi^t \colon X \to X$ on a metric space, the \emph{strongly stable manifold of $v \in X$} is 
$$
W^{ss}(v) := \left\{ w \in X : \lim_{t \to \infty} \dist(\phi^t(v), \phi^t(w)) = 0\right\}
$$
and the \emph{strongly unstable manifold of $v$} is
$$
W^{su}(v) := \left\{ w \in X : \lim_{t \to -\infty} \dist(\phi^t(v), \phi^t(w)) = 0\right\}.
$$
In this section we study these sets for the flow associated to a GPS system. 
(In our general setting we do not expect these to be manifolds, but the terminology is conventional.)

Fix a continuous GPS system $(\sigma, \bar{\sigma}, G)$ for a convergence group action $\Gamma\subset\mathsf{Homeo}(M)$ and let
$
p\colon \tilde U_\Gamma \to  U_\Gamma 
$
 denote the quotient map. As observed in Section~\ref{subsec:background on GPS systems}, the quotient $U_\Gamma=\Gamma \backslash \tilde U_\Gamma$ is a locally compact metrizable space. Hence the one-point compactification $U_\Gamma \sqcup \{ \infty\}$ of $U_\Gamma$ admits a metric $\dist_\star$, see~\cite{mandelkern}. Then for $v \in U_\Gamma$, let $W^{ss}(v)$ and $W^{su}(v)$ denote the strongly stable and unstable manifolds for the metric $\dist_\star$ restricted to $U_\Gamma$.

We first show that the strongly stable manifold of $p(v^-,v ^+,t_0)$ contains quotients of all elements with the same forward endpoint  $v^+$ and time parameter $t_0$.

\begin{proposition}\label{prop:Wss computation} 
If $(v^-, v^+, t_0) \in \tilde U_\Gamma$ and $v := p(v^-, v^+, t_0) \in U_\Gamma$, then 
$$
p(y, v^+, t_0) \in W^{ss}(v)
$$
for all $y \in \Lambda(\Gamma) \smallsetminus \{ v^+\}$.
\end{proposition}

\begin{proof} Fix $y \in \Lambda(\Gamma) \smallsetminus \{ v^+\}$ and let $w:=p(y, v^+, t_0)$. Then fix a sequence $\{t_n\}$ where $t_n \to +\infty$ and 
$$
 \limsup_{t \to \infty} \dist_\star(\psi^t(v), \psi^t(w)) =  \lim_{n \to \infty} \dist_\star(\psi^{t_n}(v), \psi^{t_n}(w)).
$$
Passing to a subsequence we can assume that one of the following cases hold: 

\medskip

\noindent \emph{Case 1:}  Assume $\{ \psi^{t_n}(v)\}$ is relatively compact in $U_\Gamma$.
Passing to a further subsequence, we can find a sequence $\{\gamma_n\}$ in $\Gamma$ such that 
$$
\lim_{n \to \infty} \gamma_n \psi^{t_n}(v)=\lim_{n \to \infty} (\gamma_n v^-, \gamma_n v^+, t_n+t_0+\sigma(\gamma_n, v^+))
$$
exists in $\tilde U_\Gamma$. Passing to further subsequence we can suppose that $\gamma_n \to a \in \Lambda(\Gamma)$ and $\gamma_n^{-1} \to b \in \Lambda(\Gamma)$. Then  $\gamma_n(z)\to a$ uniformly on compact subsets of $M\smallsetminus\{b\}$.

Since $\sigma$ is a cocycle,
$$0=\sigma(\id,v^+)=\sigma(\gamma_n^{-1}\gamma_n,\gamma_n^{-1}(\gamma_n v^+))=\sigma(\gamma_n^{-1}, \gamma_n v^+) +\sigma(\gamma_n, v^+).$$
Then, since $t_n \to +\infty$ and $\{ \gamma_n \psi^{t_n}(v)\}$ converges in $\tilde U_\Gamma$, we must have 
$$\lim_{n \to \infty} \sigma(\gamma_n^{-1}, \gamma_nv^+) = \lim_{n \to \infty} -\sigma(\gamma_n, v^+)=+\infty. 
$$
So Proposition~\ref{prop:basic properties}\eqref{item:a technical fact} implies that 
$\dist(\gamma_n^{-1}(\gamma_nv^+),\gamma_n^{-1})=\dist(v^+,\gamma_n^{-1})\to 0$. Since $\gamma_n^{-1}\to b$, we see that
$b=v^+$. 
Therefore,  $\gamma_n(z) \to a$ for all $ z \in \Lambda(\Gamma) \smallsetminus \{v^+\}$. So
$$
\lim_{n \to \infty} \gamma_n(y) = a=\lim_{n \to \infty} \gamma_n(v^-).
$$
Thus
$$
\lim_{n \to \infty} \gamma_n\psi^{t_n}(y, v^+, t_0)=\lim_{n \to \infty} \gamma_n\psi^{t_n}(v^-,v^+, t_0),
$$
which implies that 
$$
 \lim_{n \to \infty} \dist_\star(\psi^{t_n}(v), \psi^{t_n}(w)) = 0
$$
and hence that $w \in W^{ss}(v)$. 
\medskip

\noindent \emph{Case 2:}  Assume $\{ \psi^{t_n}(w)\}$ is relatively compact in $U_\Gamma$. By Case 1, $v \in W^{ss}(w)$ so \hbox{$w \in W^{ss}(v)$}. 

\medskip

\noindent \emph{Case 3:} Assume $\{ \psi^{t_n}(v)\}$ and $\{ \psi^{t_n}(w)\}$ both converge to $\infty$ in $U_\Gamma \sqcup \{\infty\}$. Then 
 $$
 \lim_{n \to \infty} \dist_\star(\psi^{t_n}(v), \psi^{t_n}(w)) =  \dist_\star(\infty, \infty) = 0
 $$
and hence $w \in W^{ss}(v)$. 
\end{proof}

We establish the analogous result for the strongly unstable manifold. 

\begin{proposition}\label{prop:Wsu computation} 
If $(v^-, v^+, t_0) \in \tilde U_\Gamma$ and $v := p(v^-, v^+, t_0) \in U_\Gamma$, then 
$$
p\left( v^-, x, t_0+G(v^-,x)-G(v^-, v^+)\right) \in W^{su}(v)
$$
for all $x \in \Lambda(\Gamma) \smallsetminus \{ v^-\}$.
\end{proposition} 

\begin{proof}
Fix $x \in \Lambda(\Gamma) \smallsetminus \{ v^-\}$ and let 
$$
w:=p\left(v^-, x, t_0+G(v^-,x)-G(v^-, v^+)\right).
$$ 
Then fix a sequence $\{t_n\}$ where $t_n \to -\infty$ and 
$$
 \limsup_{t \to -\infty} \dist_\star(\psi^t(v), \psi^t(w)) =  \lim_{n\to\infty} \dist_\star(\psi^{t_n}(v), \psi^{t_n}(w)).
$$
Passing to a subsequence we can assume that one of the following cases hold: 

\medskip

\noindent \emph{Case 1:} Assume $\{ \psi^{t_n}(v)\}$ is relatively compact in $U_\Gamma$.
Then passing to a further subsequence, we can find a sequence $\{\gamma_n\}$ in $\Gamma$ such that 
$$
\lim_{n \to \infty} \gamma_n\psi^{t_n}(v)=\lim_{n \to \infty} (\gamma_n v^-, \gamma_n v^+, t_n+t_0+\sigma(\gamma_n, v^+))
$$
exists in $\tilde U_\Gamma$.  Passing to another subsequence we can suppose that $\gamma_n \to a \in \Lambda(\Gamma)$ and $\gamma_n^{-1} \to b \in \Lambda(\Gamma)$. Then $\gamma_n(z)\to a$ uniformly on compact subsets of $M\smallsetminus\{b\}$.

Since $t_n \to -\infty$, we must have 
$$
\lim_{n \to \infty} \sigma(\gamma_n, v^+) =+\infty. 
$$
So Proposition~\ref{prop:basic properties}\eqref{item:a technical fact} implies that $ \gamma_n(v^+) \to a$. Let $b^\prime : = \lim_{n \to \infty} \gamma_n(v^-)$. Since 
$$
\lim_{n \to \infty} \gamma_n(v^-, v^+) = (b^\prime, a) \in \Lambda(\Gamma)^{(2)}
$$
and $\gamma_n(z) \to a$ for all $z \in M \smallsetminus \{b\}$, we must have $v^- = b$. Then $\gamma_n(x) \to a$ since $x \in \Lambda(\Gamma) \smallsetminus\{v^-\} = \Lambda(\Gamma)\smallsetminus\{b\}$. 

Then, by Equation~\eqref{eqn:GPS property}, 
\begin{align*}
\lim_{n \to \infty}&  \sigma(\gamma_n, v^+)-\sigma(\gamma_n, x) = \lim_{n \to \infty}\Big(\bar{\sigma}(\gamma_n, v^-)+\sigma(\gamma_n, v^+)\Big) - \Big(\bar{\sigma}(\gamma_n, v^-)+\sigma(\gamma_n, x)\Big) \\
& = \lim_{n \to \infty} G(\gamma_n v^-, \gamma_n v^+) - G(v^-, v^+) - G(\gamma_n v^-, \gamma_n x) +G(v^-, x) \\
& =  G(b^\prime,a) - G(v^-,v^+) -G(b^\prime,a)+G(v^-,x) \\
&=  G(v^-,x)-G(v^-, v^+). 
\end{align*}
Thus 
$$
\lim_{n \to \infty} t_n+t_0+\sigma(\gamma_n, v^+)= \lim_{n \to \infty} t_n+t_0+G(v^-,x)-G(v^-, v^+)+\sigma(\gamma_n, x)
$$
and so
$$
\lim_{n \to \infty} \gamma_n\psi^{t_n}(v^-, v^+, t_0)=\lim_{n \to \infty} \gamma_n\psi^{t_n}(v^-,x, t_0+G(v^-,x)-G(v^-, v^+)),$$
which  implies that 
$$
 \lim_{n \to \infty} \dist_\star(\psi^{t_n}(v), \psi^{t_n}(w)) = 0
$$
and hence that $w \in W^{su}(v)$. 
\medskip

\noindent \emph{Case 2:} Assume $\{ \psi^{t_n}(w)\}$ is relatively compact in $U_\Gamma$. By Case 1, $v \in W^{su}(w)$ so \hbox{$w \in W^{su}(v)$}. 

\medskip

\noindent \emph{Case 3:}  Assume $\{ \psi^{t_n}(v)\}$ and $\{ \psi^{t_n}(w)\}$ both converge to $\infty$ in $U_\Gamma \sqcup \{\infty\}$. Then 
 $$
 \lim_{n \to \infty} \dist_\star(\psi^{t_n}(v), \psi^{t_n}(w)) =  \dist_\star(\infty, \infty) = 0
 $$
and hence $w \in W^{su}(v)$. 
\end{proof}


\section{Mixing}\label{sec:mixing}


In this section we establish mixing for the flow associated to a GPS system when the Bowen--Margulis is finite and the length spectrum is non-arithmetic. In fact, we can slightly weaken the assumption of non-arithmetic length spectrum and instead assume that the ``cross ratios'' generate a dense subgroup of $\Rb$. 

For the rest of the section suppose that  $(\sigma, \bar{\sigma}, G)$ is a continuous GPS system for a convergence group action $\Gamma\subset\mathsf{Homeo}(M)$ where $\delta:=\delta_\sigma(\Gamma) < +\infty$ and $Q_\sigma(\delta) = +\infty$. Let $m_\Gamma$ denote the BMS measure on $U_\Gamma$ constructed in Section~\ref{subsec:background on GPS systems}.

We define a cross ratio 
$$
B(x,x',y,y') := G(x,y) + G(x',y') - G(x',y) - G(x,y')
$$
for $x,x',y,y'\in M$ such that $\{x,x'\}$ and $\{y,y'\}$ are disjoint. We then  define the {\em cross ratio spectrum}
$$
\mathcal{CR}: = \left\{ B(x,x',y,y')  : y,y', x,x' \in \Lambda(\Gamma) \text{ and } \{x,x'\} \cap \{y,y'\} = \varnothing\right\}. 
$$
We say that the cross ratio spectrum is \emph{non-arithmetic} if it generates a dense subgroup of $\mathbb R$.

The next lemma shows that the cross ratio spectrum  $\mathcal{CR}$ contains the length spectrum 
$$
\Lc(\sigma, \bar \sigma, G) = \{  \ell_\sigma(\gamma) + \ell_{\bar\sigma}(\gamma) : \gamma \in \Gamma \text{ loxodromic} \}. 
$$

\begin{lemma} \label{lem:length as cross-ratio}
If $\gamma \in \Gamma$ is loxodromic and  $x \in M\smallsetminus\{\gamma^+,\gamma^-\}$, then 
$$
B(x, \gamma x,\gamma^-, \gamma^+) = \ell_\sigma(\gamma) + \ell_{\bar\sigma}(\gamma).
$$
\end{lemma}
\begin{proof} Notice that 
$$
 \bar\sigma(\gamma,\gamma^+)+\sigma(\gamma,\gamma^-) = G(\gamma^+,\gamma^-) - G(\gamma^+,\gamma^-) = 0.
$$
So, by Equation~\eqref{eqn:GPS property},
\begin{align*}
B(x, \gamma x,\gamma^-, \gamma^+)  & = G(x,\gamma^-) + G(\gamma x,\gamma^+) - G(\gamma x,\gamma^-) - G(x,\gamma^+) \\
& = G(\gamma x,\gamma^+)  - G(x,\gamma^+) - G(\gamma x,\gamma^-) +G(x,\gamma^-)  \\
& =\bar\sigma(\gamma,x)+ \sigma(\gamma,\gamma^+) - \bar\sigma(\gamma,x)-\sigma(\gamma,\gamma^-) \\
& = \sigma(\gamma,\gamma^+)  + \bar\sigma(\gamma,\gamma^+) =  \ell_\sigma(\gamma) + \ell_{\bar\sigma}(\gamma). \qedhere
\end{align*}
\end{proof}

By Lemma~\ref{lem:length as cross-ratio}, the following theorem is a (slight) extension of Theorem \ref{main mixing}.

\begin{theorem} \label{thm:mixing}
If the BMS measure  $m_\Gamma$ is finite and the cross ratio spectrum $\mathcal{CR}$ is non-arithmetic,  then the flow $\psi^t \colon(U_\Gamma, m_\Gamma) \to (U_\Gamma, m_\Gamma)$ is mixing.
 \end{theorem}

 Our proof is inspired by earlier work of Blayac \cite{blayac-ps} in the setting of rank-one convex projective manifolds with compact convex core.
 This strategy of proof goes back to work of Babillot \cite{babillot}.
 In particular, to establish mixing, we will use a criterion due to Coud\`ene. To state his result we need a preliminary definition. 
 
Suppose is a measurable flow $\phi^t \colon X \to X$ on a metric space and $\mu$ is a flow-invariant measure, then a function $f \colon X \to \Rb$ is  \emph{$W^{ss}$-invariant} 
if there exists a full $\mu$-measure subset $X' \subset X$ such that if $v,w \in X'$ and $w \in W^{ss}(v)$, then $f(v) =f(w)$. 
Likewise, $f$ is \emph{$W^{su}$-invariant} if there exists a full $\mu$-measure subset $X'' \subset X$ such that if $v,w \in X''$ and $w \in W^{su}(w)$, then $f(v) =f(w)$. 

\begin{proposition}[{Coud\`ene~\cite{coudene-mixing}}]
Let $X$ be a metric space, $\mu$ be a finite Borel measure on $X$, and $\phi^t \colon X \to X$ a measure-preserving flow on $X$. 
If any measurable function which is $W^{ss}$-invariant and $W^{su}$-invariant is constant almost everywhere,
then $\phi^t \colon (X,\mu) \to (X,\mu)$ is mixing.
\end{proposition}

 We are now ready to prove the theorem.

\begin{proof}[Proof of Theorem~\ref{thm:mixing}] As in Section~\ref{sec:stable and unstable manifolds}, we fix a metric $\dist_\star$ on the one-point compactification of $U_\Gamma$ and consider the stable/unstable manifolds relative to this metric. 

By Coud\`ene's result, it suffices to show that every measurable  function $f$ on $U_\Gamma$ which is $W^{ss}$-invariant and $W^{su}$-invariant is $m_\Gamma$-almost everywhere constant.
Let $f$ be such a function, and let $\tilde{f}$ denote the lift of $f$ to $\tilde U_\Gamma$.
Let $A\subset U_\Gamma$ be a full measure subset such that for all $v,v'\in A$, if $v'\in W^{ss}(v)$ or $W^{su}(v)$ then $f(v')=f(v)$.
Let $\tilde A\subset \tilde U_\Gamma$ denote the preimage of $A$. 

By Theorem~\ref{thm:PS exist and unique}, the measures $\mu$ and $\bar\mu$ have no atoms and so $\tilde A$ is a full measure set for the product measure $\bar \mu \otimes \mu \otimes dt$ 
on $\Lambda(\Gamma) \times \Lambda(\Gamma) \times \Rb$.

For $(x,y) \in \Lambda(\Gamma)^{(2)}$ and $y' \in \Lambda(\Gamma) \smallsetminus \{x\}$, let 
$$
\rho_{x,y}(y')= G(x,y') - G(x,y).
$$
Notice that $\rho_{x,y}(y') + \rho_{x',y'}(y)=-B(x,x',y,y')$. 

Since $f$ is $W^{ss}$- and $W^{su}$-invariant, by Propositions~\ref{prop:Wss computation} and ~\ref{prop:Wsu computation} we have 
$$
\tilde{f}(x,y,t) = \tilde{f}(x',y,t) = \tilde{f}(x,y',t+\rho_{x,y}(y'))
$$
for $\bar\mu^2\otimes \mu^2\otimes dt$-almost any $(x,x',y,y',t)$:
for $\bar\mu^2\otimes \mu^2\otimes dt$-almost any $(x,x',y,y',t)$ we have $(x,y,t)\in \tilde A$ and $(x',y,t)\in \tilde A$ and $(x,y',t+\rho_{x,y}(y'))\in\tilde A$, and the projection of $(x',y,t)$ in $U_\Gamma$ 
is in the strong stable manifold of the projection of $(x,y,t)$ (by Proposition~\ref{prop:Wss computation}), while $(x,y',t+\rho_{x,y}(y'))$ is in the strong unstable manifold of $(x,y,t)$ (by Proposition~\ref{prop:Wsu computation}).

Hence, by Fubini's Theorem, we can find $x_0, y_0 \in \Lambda(\Gamma)$ such that
$$
\tilde{f}(x_0,y_0,t) = \tilde{f}(x_0,y,t+\rho_{x_0,y_0}(y)) = \tilde{f}(x,y,t+\rho_{x_0,y_0}(y))
$$
for $\bar\mu\otimes \mu\otimes dt$-almost any $(x,y,t)$.
		
In particular it suffices to show that $g(t) := \tilde{f}(x_0,y_0,t)$ is Lebesgue-almost everywhere constant.
Consider the additive subgroup 
$$
\mathsf{H}:=\{\tau \in \Rb: g(t+\tau) = g(t) \text{ for Lebesgue-almost any } t\in\Rb\}.
$$
The following classical result says that $\mathsf H$ is a closed subgroup of $\Rb$.
For the reader's convenience we recall its proof after we finish the current proof.

\begin{lemma}\label{lem:closed subgroup}
If $g\colon\R\to\R$ be a measurable function and $\mathsf H(g):=\{\tau \in \Rb: g(t+\tau) = g(t) \text{ for Lebesgue-almost any } t\in\Rb\}$, 
 then $\mathsf H(g)$ is a closed subgroup of $\R$.
\end{lemma}

We now claim that $\mathcal{CR} \subset \mathsf H$. The assumption that  $\mathcal{CR}$ generate a dense subgroup of $\mb R$ then implies that $\mathsf H = \Rb$, and hence that $g$ is Lebesgue-almost everywhere constant, as desired.

To this end, we observe that, for $\bar\mu$-almost all $x$ and $x'$, $\mu$-almost all $y$ and $y'$ and Lebesgue-almost any $t$, the four points $x,x',y,y'$ are distinct (since $\mu$ and $\bar\mu$ do not have atoms by Theorem~\ref{thm:PS exist and unique}) and
\begin{align*}
g(t) & = \tilde{f}(x,y,t+\rho_{x_0,y_0}(y)) \\
& = \tilde{f}(x',y',t+\rho_{x_0,y_0}(y) + \rho_{x,y}(y')) \\
& = \tilde{f}(x,y,t+\rho_{x_0,y_0}(y) + \rho_{x,y}(y') + \rho_{x',y'}(y)) \\
& = \tilde{f}(x,y,t+\rho_{x_0,y_0}(y) - B(x,x',y,y')) = g(t-B(x,x',y,y')).
\end{align*}
So $B(x,x',y,y') \in \mathsf H$. 
		
		Since the Patterson--Sullivan measures $\mu$ and $\bar\mu$ have full support in $\Lambda(\Gamma)$, $B$ is continuous, and $\mathsf H$ is closed, we obtain that  $ \mathcal{CR} \subset \mathsf H$. Therefore, $g$ is Lebesgue-almost everywhere constant, so $\tilde f$ is $\tilde m$-almost everywhere constant, so $f$ is $m_\Gamma$-almost everywhere constant, which completes the proof.
\end{proof}

\begin{proof}[Proof of Lemma~\ref{lem:closed subgroup}]
 The fact that $\mathsf H(g)$ is a subgroup of $\R$ is an immediate consequence of the invariance of the Lebesgue measure under translation.
 It remains to check $\mathsf H(g)$ is closed.

 Note that 
 \[ \mathsf H(g)=\bigcap_{R>0}\mathsf H(g \cdot 1_{|g|\leq R}), \]
 so we can assume $|g|$ is bounded by some $R>0$.
 
 For any compactly supported continuous function $\alpha\colon \R\to\R$, let 
 \[ \mathsf H_\alpha(g):=\left\{\tau: \int \alpha(t)g(t+\tau)dt=\int \alpha(t)g(t)dt\right\}. \]
 It is an easy exercise in measure theory that
 \[ \mathsf H(g) = \bigcap_{\alpha\in C_c(\R)} \mathsf H_\alpha(g).\]
 So it suffices to prove $\mathsf H_\alpha(g)$ is closed.
 This a consequence of the fact that 
 $$\int \alpha(t) g(t+\tau)dt=\int \alpha(t-\tau) g(t)dt$$
 is continuous in $\tau$, which follows from the fact that $\alpha$ is uniformly continuous.
\end{proof}

\begin{remark} (1) 
It follows from Theorem~\ref{thm:equidistribution}, that if $m_\Gamma$ is mixing then $\{ \ell_\sigma(\gamma) : \gamma \in \Gamma \text{ loxodromic}\}$ 
generates a dense subgroup of $\mb R$. However this does not imply that the length spectrum $\Lc(\sigma, \bar \sigma,G)$  is non-arithmetic,
unless $\sigma$ is symmetric (i.e.\ if $\sigma=\bar\sigma$).
If $\sigma$ is symmetric, we have an equivalence between $m_\Gamma$ being mixing, the length spectrum being non-arithmetic, and the set of cross ratios generating a dense subgroup of $\mb R$.

(2) There exist cases when $m_\Gamma$ is finite but not mixing (and the length spectrum generates a discrete subgroup of $\Rb$).
For instance, take $M$ to be the Gromov boundary of an infinite $4$-regular tree with all edges of length $1$, take $\Gamma$ to be the nonabelian free group with two generators acting on $M$,
and define $\sigma$ and $\bar\sigma$ to be the Busemann cocycles on this ${\rm CAT}(-1)$ space.
Then the free group action is a uniform convergence action, hence $m_\Gamma$ is finite, but the length spectrum generates a discrete additive subgroup 
which is contained in $\Zb$, the measure $m_\Gamma$ is not mixing, and $Re^{-\delta_\Gamma R}\#\{[\gamma]\in[\Gamma]:\ell_\sigma(\gamma)\leq R\}$ does not converge as $R\to\infty$. 
\end{remark}


\section{Non-arithmeticity of the cross ratio spectrum}\label{sec:NALS}


In this section, we investigate when the length spectrum or  cross ratio spectrum is non-arithmetic. Our results and arguments are very similar to earlier work in the context of Riemannian manifolds, see Dal'bo \cite[\S II]{Dalbo}. 

When $\Gamma$ contains a parabolic element, the length spectrum itself is always non-arithmetic.

\begin{proposition}
\label{prop:parabolic NALS} 
Suppose $(\sigma, \bar{\sigma}, G)$ is a continuous GPS system for a convergence group $\Gamma\subset\mathsf{Homeo}(M)$. If $\Gamma$ contains a parabolic element, then the length spectrum $\Lc(\sigma, \bar \sigma,G)$ is non-arithmetic.
\end{proposition}

\begin{proof} Let $\alpha\in\Gamma$ be a parabolic element with fixed point $p\in M$ and choose $\beta\in \Gamma$ that does not fix $p$.
 Set $\gamma_n:=\beta\alpha^n$ for each $n\geq 0$.  Note that $\gamma_n\to \beta(p)$ while $\gamma_n^{-1}=\alpha^{-n}\beta^{-1}\to p\neq \beta(p)$.
 So, by Lemma~\ref{lem:char of loxodromic}, for $n$ large enough $\gamma_n$ is loxodromic with $\gamma_n^+\to \beta(p)$ and $\gamma_n^-\to p$.
 
 To prove the proposition it suffices to show the sequence
 \[ \{\ell_\sigma(\gamma_{n+1})+\ell_{\bar\sigma}(\gamma_{n+1}) - \ell_\sigma(\gamma_n) - \ell_{\bar\sigma}(\gamma_n)\}\]
takes arbitrarily small nonzero values.
By Proposition~\ref{prop:basic properties}\eqref{item:properness} and \eqref{item:loxodromics with separated fixed points} the sequence $\{\ell_\sigma(\gamma_n) + \ell_{\bar\sigma}(\gamma_n)\}$ tends to infinity. So the above sequence of differences must be nonzero infinitely often.
 Hence it suffices to prove that 
 $$
 \lim_{n \to \infty} \ell_\sigma(\gamma_{n+1})-\ell_\sigma(\gamma_n) =  0 =  \lim_{n \to \infty} \ell_{\bar\sigma}(\gamma_{n+1})-\ell_{\bar\sigma}(\gamma_n).
 $$

 By definition,
 $$\ell_\sigma(\gamma_n)=\sigma(\beta\alpha^n,\gamma_n^+)=\sigma(\beta,\alpha^n\gamma_n^+)+\sigma(\alpha^n,\gamma_n^+).$$
Notice that $\sigma(\beta,\alpha^n\gamma_n^+) \to \sigma(\beta,p)$. Further $\sigma(\alpha^{n+1},\gamma_{n+1}^+)=\sigma(\alpha,\alpha^n(\gamma_{n+1}^+))+\sigma(\alpha^n,\gamma_{n+1}^+)$ and $\sigma(\alpha,\alpha^n(\gamma_{n+1}^+))\to \sigma(\alpha,p)=0$. So,
$$
\limsup_{n \to \infty} \abs{ \ell_\sigma(\gamma_{n+1})-\ell_\sigma(\gamma_n)} = \limsup_{n \to \infty} \abs{\sigma(\alpha^n,\gamma_{n+1}^+)-\sigma(\alpha^n,\gamma_n^+)}.
$$
Since $\alpha^{-n}\beta(p)\to p$, we see that
 \begin{align*}
  \sigma(\alpha^{n},& \gamma_{n+1}^+) -\sigma(\alpha^n,\gamma_n^+) =\bar\sigma(\alpha^n,\alpha^{-n}\beta(p))+ \sigma(\alpha^{n},\gamma_{n+1}^+)  - \bar\sigma(\alpha^n,\alpha^{-n}\beta(p))-  \sigma(\alpha^n,\gamma_n^+)\\
  &= G(\beta(p),\alpha^n(\gamma_{n+1}^+))-G(\alpha^{-n}\beta(p),\gamma_{n+1}^+) -G(\beta(p),\alpha^n\gamma_n^+)+G(\alpha^{-n}\beta(p),\gamma_n^+)\\
  &\to G(\beta(p),p)-G(p,\beta(p))-G(\beta(p),p)+G(p,\beta(p))=0.
 \end{align*}
 Hence $\ell_\sigma(\gamma_{n+1})-\ell_\sigma(\gamma_n) \rightarrow 0$. 
 
 The proof that  $\ell_{\bar\sigma}(\gamma_{n+1})-\ell_{\bar\sigma}(\gamma_n)\to 0$ is completely analogous.
\end{proof}

We verify that the cross ratio spectrum is non-arthimetic in the following cases. 

\begin{proposition}\label{prop:CR spectrum nonarthimetic}
Suppose $(\sigma, \bar{\sigma}, G)$ is a continuous GPS system for a convergence group $\Gamma\subset\mathsf{Homeo}(M)$. If any one of the following hold, then the cross ratio spectrum 
$\mathcal{CR}$ is non-arithmetic:
\begin{enumerate}
\item\label{item:NALS1} there exists $x\in \Lambda(\Gamma)$ so that the path component of $\Lambda(\Gamma)$ containing $x$ is infinite and  $\lim_{y\in\Lambda(\Gamma) \to x} G(x,y)=+\infty$, 
\item\label{item:NALS2} there exists a conical limit point $x\in \Lambda^{\rm con}(\Gamma)$ so that the path component of $\Lambda(\Gamma)$ containing $x$ is infinite,  or
\item\label{item:NALS3} $\Gamma$ is uniform convergence group (i.e.\ every limit point is conical) and is not virtually free.
\end{enumerate}
\end{proposition}

The rest of the section is devoted to the proof of Proposition~\ref{prop:CR spectrum nonarthimetic}. 
We will see that \eqref{item:NALS3} is a particular case of \eqref{item:NALS2}.

\begin{proof}[Proof of non-arithmeticity given \eqref{item:NALS1}]
Fix a continuous path $c\colon [0,1]\to \Lambda(\Gamma)$ such that $c(0)=x$ and $c(0)\neq c(t)$ for any $t\in (0,1]$ and $c(\frac12)\neq c(1)$.
 The function 
 \begin{align*}t\in(0,\epsilon)\mapsto & B(c(0),c(\tfrac12),c(t),c(1))\\&=G(c(0),c(t))+G(c(\tfrac12),c(1))-G(c(\tfrac12),c(t))-G(c(0),c(1))\end{align*}
 is well-defined  and continuous for $\epsilon>0$ small enough,  and goes to infinity as $t\to 0$ because $G(c(0),c(t))\to\infty$ by assumption,
 while $G(c(\frac12),c(t))\to G(c(\frac12),c(0))$. By the intermediate value theorem, the image of that function contains a nontrivial interval of $\mb R$ and 
 hence $\mathcal{CR}$ generates $\mb R$ as a group.
\end{proof}

\begin{proof}[Proof of non-arithmeticity given \eqref{item:NALS2}]
Fix a continuous path $c\colon [0,1]\to \Lambda(\Gamma)$ such that $c(0)=x$ and $c(0)\neq c(t)$ for any $t\in (0,1]$ and $c(\frac12)\neq c(1)$.
 The function 
 \begin{align*}t\in(0,\epsilon)\mapsto  f(t) : & = B(c(0),c(\tfrac12),c(t),c(1))\\&=G(c(0),c(t))+G(c(\tfrac12),c(1))-G(c(\tfrac12),c(t))-G(c(0),c(1))\end{align*}
 is well-defined and continuous for $\epsilon>0$ small enough.
 
  Since $x$ is conical there exist $\{\gamma_k\}\subset\Gamma$ and $a\neq b\in \Lambda(\Gamma)$ such that $\gamma_kx\to a$ and $\gamma_ky\to b$ for any $y\neq x$, which implies $\gamma_k\to b$.
 Since $\gamma_k c(0) = \gamma_k x \to a$ and $\gamma_k c(1) \to b$, after passing to a subsequence we can find $\{t_k\} \subset \left(0, \frac\epsilon2 \right)$ such that $\gamma_k c(t_k) \to z \in M \smallsetminus \{a,b\}$. 
 
Since $z \neq b$ and $a \neq b$, by the expanding property there exists a constant $C>0$ such that 
$$\sigma(\gamma_k^{-1},\gamma_kc(t_k))\geq \norm{\gamma_k^{-1}}_\sigma-C \quad\text{and}\quad \bar\sigma(\gamma_k^{-1},\gamma_kx)\geq \norm{\gamma_k^{-1}}_{\bar\sigma}-C$$ 
for all $k \geq 1$. 
Thus by Proposition~\ref{prop:basic properties}\eqref{item:properness} we have 
 $$
 \lim_{k \to \infty} \sigma(\gamma_k^{-1},\gamma_kc(t_k))= \lim_{k \to\infty} \bar\sigma(\gamma_k^{-1},\gamma_kx)=+\infty. 
 $$
Further, since $z \neq a$ and $G: M^{(2)} \rightarrow \Rb$ is continuous, the sequence $\{  G(\gamma_k x, \gamma_k c(t_k)) \}$ is bounded. 
 
 Then by Equation~\eqref{eqn:GPS property}, we have 
 \begin{align*}
 \lim_{k \to \infty} G(x, c(t_k)) & =  \lim_{k \to \infty}  G(\gamma_k x, \gamma_k c(t_k)) - \bar\sigma(\gamma_k, x) - \sigma(\gamma_k, c(t_k)) \\
 & =  \lim_{k \to \infty} G(\gamma_k x, \gamma_k c(t_k))+\bar\sigma(\gamma_k^{-1}, \gamma_kx) + \sigma(\gamma_k^{-1}, \gamma_kc(t_k)) =+\infty. 
 \end{align*}
 Thus $f(t_k) \to +\infty$ as $k\to\infty$. So by the intermediate value theorem, the image of $f$ contains a nontrivial interval of $\mb R$ and 
 hence $\mathcal{CR}$ generates $\mb R$ as a group.
\end{proof}

\begin{proof}[Proof of non-arithmetic length spectrum given \eqref{item:NALS3}]
Since $\Gamma$ is a uniform convergence group, $\Gamma$ is word hyperbolic and there exists a equivariant homeomorphism between the Gromov boundary and the limit set $\Lambda(\Gamma)$. Then, since $\Gamma$ is not virtually free, a result of Bonk--Kleiner~\cite{BK2002} implies that the limit set contains an embedded circle. It then follows from \eqref{item:NALS2} that the group generated by $\mathcal{CR}$ is dense in $\mb R$.
\end{proof}

As an aside, we observe that the condition $\lim_{y\in\Lambda(\Gamma) \to x} G(x,y)=+\infty$ is automatically satisfied at bounded parabolic points.

\begin{lemma}\label{lem:G to infinity at parabs}
If $x\in\Lambda(\Gamma)$ is a bounded parabolic point, then 
$$\lim_{y\in\Lambda(\Gamma)\to x} G(x,y)=+\infty=\lim_{y\in\Lambda(\Gamma)\to x} G(y,x).$$
\end{lemma}

\begin{proof}
Since $x$ is a bounded parabolic point there exists a compact subset $K$ of 
$\Lambda(\Gamma)\smallsetminus\{x\}$  such that $\Gamma_x \cdot K=\Lambda(\Gamma)\smallsetminus\{x\}$.

Fix a sequence $\{y_n\}\subset \Lambda(\Gamma)\smallsetminus\{x\}$ converging to $x$.
For each $n$, choose $\gamma_n\in\Gamma_x$ so that $\gamma_ny_n\in K$.
Since $\sigma$ is expanding and $\gamma_n \to x$, there exists $C>0$ such that 
$$\sigma(\gamma_n^{-1},\gamma_ny_n)\geq ||\gamma_n||_\sigma-C$$ 
for all $n \geq 1$. Since $\bar\sigma(\gamma_n^{-1},x)=0$, see Proposition \ref{prop:basic properties}\eqref{item:proper implies positive periods},
Equation~\eqref{eqn:GPS property} implies that
\begin{align*}
G(x,y_n)&=G(x,\gamma_n(y_n))+\sigma(\gamma_n^{-1},\gamma_n(y_n)) \geq \min_{y \in K} G(x,y) + \norm{\gamma_n^{-1}}_\sigma-C.
\end{align*}
Since $\gamma_n\to x$, Proposition~\ref{prop:basic properties}\eqref{item:properness} implies that $\norm{\gamma_n^{-1}}_\sigma\to+\infty$, so $G(x,y_n)\to +\infty$. A similar argument shows that $G(y_n,x)\to+\infty$, where we use $\bar\sigma$ in place of $\sigma$.
Since $\{y_n\}\subset \Lambda(\Gamma)\smallsetminus\{x\}$ was an arbitrary sequence converging to $x$, the lemma follows. 
\end{proof}


\section{Equidistribution}\label{sec:equi}


We are now ready to prove Theorem~\ref{thm:equidistribution}, which we restate here. The proof follows a classical strategy that goes back to Margulis \cite{margulis}.
Our particular implementation of this strategy is  influenced by Roblin \cite{roblin}, although, unlike Roblin,  we directly establish
equidistribution of closed geodesics, without first establishing double equidistribution of orbit points in $\tilde U_\Gamma$.

\begin{theorem}\label{thm:equidistribution in paper}
Suppose $(\sigma, \bar{\sigma}, G)$ is a continuous GPS system for a convergence group $\Gamma\subset\mathsf{Homeo}(M)$ where $\delta:=\delta_\sigma(\Gamma) < +\infty$ and $Q_\sigma(\delta) = +\infty$. 
Let $\mu$ be the unique $\sigma$-Patterson--Sullivan measure of dimension $\delta$ and let $\bar\mu$ be the unique $\bar\sigma$-Patterson--Sullivan measure of dimension $\delta$.

If the BMS measure $m_\Gamma$ is finite and mixing, then
 \[
 \lim_{T \to \infty} \delta e^{-\delta T} \sum_{\substack{\gamma \in \Gamma_{\lox}\\ \ell_\sigma(\gamma)\leq T}}\mc D_{\gamma^-}\otimes \mc D_{\gamma^+} = \frac{1}{\norm{m_\Gamma}} e^{\delta G(x,y)} \bar\mu(x)\otimes \mu(y)
 \]
  in the dual of compactly supported continuous functions.
\end{theorem}

The rest of the section is devoted to the proof of the theorem.  Reformulating in terms of measures on the flow space,  the conclusion of the theorem is equivalent to 
 \[
  {\norm{m_\Gamma}} \delta e^{-\delta T} \sum_{\substack{\gamma \in \Gamma_{\lox}\\ \ell_\sigma(\gamma)\leq T}}\mc D_{\gamma^-}\otimes \mc D_{\gamma^+}\otimes \Leb \xrightarrow[T\to\infty]{} \tilde m=e^{\delta G(x,y)} \bar\mu(x)\otimes \mu(y) \otimes \Leb.
 \]
 For ease of notation, let
 \[
  \tilde{m}_T := \sum_{\substack{\gamma \in \Gamma_{\lox}\\ \ell_\sigma(\gamma)\leq T}}\mc D_{\gamma^-}\otimes \mc D_{\gamma^+}\otimes \Leb.
 \]
 
 \begin{lemma} For any compact set $K \subset M^{(2)}$ and bounded interval $I \subset \Rb$, 
 $$
\sup_{T \geq 0}  e^{-\delta T}\tilde{m}_T(K \times I) < +\infty. 
 $$
 \end{lemma} 
 
 \begin{proof} By Proposition~\ref{prop:basic properties}\eqref{item:loxodromics with separated fixed points}, there exists $C_1 > 0$ such that: if $\gamma \in \Gamma_{\lox}$ and $(\gamma^-, \gamma^+) \in K$, then 
 $$
 \ell_\sigma(\gamma) \geq \norm{\gamma}_\sigma - C_1. 
 $$
 Further, by~\cite[Prop.\,6.3]{BCZZ} there exists $C_2 > 0$ such that 
 $$
 \#\{ \gamma \in \Gamma : \norm{\gamma}_\sigma \leq R \} \leq C_2 e^{\delta R}. 
 $$
 Hence  
\begin{equation*}
\sup_{T \geq 0}  e^{-\delta T}\tilde{m}_T(K \times I) \leq C_2 e^{\delta C_1} {\rm Leb}(I). \qedhere
\end{equation*}
 
 \end{proof} 
 
 The above lemma implies that the family of measures $\{ e^{-\delta T} \tilde{m}_T\}$ is relatively compact in the dual of compactly supported continuous functions. As a consequence, it is enough to fix an accumulation point
 $$\tilde m':=\lim_{n \to \infty} \norm{m_\Gamma}\delta e^{-\delta T_n}\tilde{m}_{T_n}$$
 (where $T_n\to +\infty$) and prove that $\tilde m=\tilde m'$.
 
 The heart of the proof consists of the following lemmas that give measure estimates for rectangles  $A\times B\times I\subset M^{(2)}\times\mb R$ which satisfy 
  \begin{equation}\label{eq:assumption on G}
  |G(a,b)-G(a',b')|\leq \epsilon\quad \text{and}\quad   |G(b,a)-G(b',a')|\leq \epsilon\quad \forall(a,b),(a',b')\in A\times B
 \end{equation}
 for some $\epsilon > 0$.  We will postpone the proof of these lemmas until after we conclude the proof of Theorem~\ref{thm:equidistribution}.

 \begin{lemma}\label{lem:equid m<m'}
  For any $\epsilon>0$ and any relatively compact subset $A\times B\times I\subset M^{(2)}\times\mb R$ with $\tilde m,\tilde m'$-null boundary satisfying
  \eqref{eq:assumption on G} and $\diam I\leq \epsilon$,
  $$\tilde m(A\times B\times I)\leq e^{7\epsilon\delta} \tilde m'(A\times B\times I).$$
 \end{lemma}
 
 \begin{lemma}\label{lem:equid m'<m}
For any $\epsilon>0$ and any relatively compact subset $A\times B\times I\subset M^{(2)}\times\mb R$ with $\tilde m,\tilde m'$-null boundary satisfying
  \eqref{eq:assumption on G} and $\diam I\leq \epsilon$,
  $$\tilde m(A\times B\times I)\geq e^{-7\epsilon\delta} \tilde m'(A\times B\times I).$$
 \end{lemma}

\subsection{Proof of Theorem~\ref{thm:equidistribution in paper}} Assuming Lemmas~\ref{lem:equid m<m'} and \ref{lem:equid m'<m} we prove Theorem~\ref{thm:equidistribution in paper}. We start with a general observation that shows that null boundary sets are abundant. 
 
 \begin{observation}\label{obs:null boundary abundant} If $X$ is a metric space and $\lambda$ is a locally finite Borel measure on $X$, then for any $x_0 \in X$ there is $r_0>0$ such that the set 
 $$
 \{ r<r_0 : \lambda(\partial B_r(x_0)) > 0 \}
 $$ 
 is countable. 
 \end{observation} 
 
 \begin{proof} Fix $r_0>0$ such that $\lambda(B_{r_0}(x_0))<\infty$. Then the function 
 $$r \in [0,r_0] \mapsto F(r) = \lambda(B_r(x_0)) \in \Rb$$ is monotone increasing. Since a monotone increasing function can have only countably many points of discontinuity, we see that 
  $$
\{r < r_0 : F \text{ discontinuous at $r$}\} \supset  \{ r<r_0 : \lambda(\partial B_r(x_0)) > 0 \}
 $$ 
 is countable. 
 \end{proof}

 Next we use  Lemmas~\ref{lem:equid m<m'} and \ref{lem:equid m'<m} to prove the following. 
 
 \begin{lemma}\label{lem:equid m<m'<m}
  If $A\times B\times I\subset M^{(2)}\times\mb R$ is relatively compact and has $\tilde m,\tilde m'$-null boundary, then 
  $$\tilde m'(A\times B\times I) = \tilde m(A\times B\times I).$$
 \end{lemma}
 
 \begin{proof}  Fix $\epsilon>0$. By the relative compactness of $A\times B\times I$, the continuity of $G$, and Observation~\ref{obs:null boundary abundant}, we can find finite covers $\bigcup_i A_i \supset A$, $\bigcup_j B_j \supset B$ and $\bigcup_k I_k \supset I$ such that for all $i,j,k$, we have
  \begin{itemize}
   \item $\bar\mu(\partial A_i)=0$, $\mu(\partial B_j) =0$, and ${\rm Leb}(\partial I_k)=0$;
   \item $A_i\times B_j\times I_k$ is relatively compact in $M^{(2)}\times\mb R$ and satisfies \eqref{eq:assumption on G};
   \item $\diam I_k\leq \epsilon$.
  \end{itemize}
 
  Set $A'_i:=A\cap A_i-(A_{i-1}\cup\dots\cup A_1)$, $B'_j:=B\cap B_j-(B_{j-1}\cup\dots\cup B_1)$ and $I'_k:=I\cap I_k-(I_{k-1}\cup\dots\cup I_1)$ for all $i,j,k$. Then 
    \begin{equation*}
   A\times B\times I = \bigsqcup_{i,j,k} A'_i\times B'_j\times I'_k.
  \end{equation*}
  Further, each $A'_i\times B'_j\times I'_k$ has $\tilde m,\tilde m'$-null boundary, hence by Lemmas~\ref{lem:equid m<m'} and \ref{lem:equid m'<m} we have
  $$e^{-7\epsilon\delta}\tilde m'(A'_i\times B'_j\times I'_k)\leq \tilde m(A'_i\times B'_j\times I'_k)\leq e^{7\epsilon\delta} \tilde m'(A'_i\times B'_j\times I'_k).$$
  So  
  $$e^{-7\epsilon\delta}\tilde m'(A\times B\times I) \leq \tilde m(A\times B\times I) \leq e^{7\epsilon\delta} \tilde m'(A\times B\times I).$$
Since $\epsilon > 0$ was arbitrary, $\tilde m(A\times B\times I)=\tilde m'(A\times B\times I)$. 
 \end{proof}

\begin{proof}[Proof of Theorem~\ref{thm:equidistribution in paper}]
The collection of relatively compact subsets $A\times B\times I\subset M^{(2)}\times\mb R$ with $\tilde m,\tilde m'$-null boundary is a $\pi$-system 
(i.e.\ it is closed under finite intersection) and by Observation~\ref{obs:null boundary abundant} it generates the Borel sigma-algebra.
This completes the proof, since two measures are equivalent if they agree on a $\pi$-system which generates the Borel sigma-algebra.
\end{proof}

\subsection{Proofs of Lemma \ref{lem:equid m<m'} and \ref{lem:equid m'<m}}
 
In the proofs of Lemma \ref{lem:equid m<m'} and \ref{lem:equid m'<m} we will need the following uniform closing lemma for GPS systems.
Recall that $\psi^t$ denotes the flow on $\psi^t(x,y,s) = (x,y,s+t)$ on  $\tilde U_\Gamma$.

\begin{lemma}\label{lem:closing}
Suppose that $A\times B\times I\subset M^{(2)}\times\mb R$ is open and relatively compact, and that  $A'\times B'\times I'$ is a compact subset of $ A\times B\times I$.
Then there exists $T$ such that: if $\gamma\in\Gamma$, $t\geq T$ and 
$$A'\times B'\times I' \cap \psi^{-t}\gamma(A'\times B'\times I') \neq \varnothing,$$
 then 
\begin{enumerate} 
\item $\gamma^{-1}(A)\times\gamma( B)\subset A\times B$, 
\item $\gamma$ is loxodromic with $(\gamma^-,\gamma^+)\in A\times B$.
\end{enumerate} 
\end{lemma}

\begin{proof}
If not, then there exist sequences $\{\gamma_n\}\subset\Gamma$, $\{t_n\}\subset\mb R$ and 
$$(x_n,y_n,s_n)\in A'\times B'\times I'\cap \psi^{-t_n}\gamma_n(A'\times B'\times I'),$$ 
such that $t_n\to+\infty$ and each $\gamma_n$ fails the conclusion of the lemma.

Since $\psi^{t_n}\gamma_n^{-1}(x_n,y_n,s_n)\in A'\times B'\times I'$, the sequence
$\{s_n+t_n+\sigma(\gamma_n^{-1},y_n)\}$ is bounded, 
so $\sigma(\gamma_n^{-1},y_n)\to-\infty$ and hence $\{\gamma_n\}$ is an escaping sequence.
Then, passing to a subsequence, we may assume that the sequences $\gamma_n^{\pm1} \to p^{\pm}\in M$.

 We claim that $(p^-, p^+) \in A' \times B'$. Since $\sigma(\gamma_n,\gamma_n^{-1}y_n)=-\sigma(\gamma_n^{-1},y_n)\to+\infty$, Proposition~\ref{prop:basic properties}\eqref{item:a technical fact}
 implies that $y_n\to p^+$.  Therefore, $p^+\in B'$. 
 If $K$ is any compact subset of $M\smallsetminus\{p^+\}$, then $\gamma^{-1}_n(K)\to p^-$.
 Since $A'$ is a compact subset of $M\smallsetminus B\subset M\smallsetminus \{p^+\}$ and $x_n\in A'$ for all $n$, we see that 
 $\gamma_n^{-1}(x_n)\to p^-$ and  so $p^-\in A'$.  

Now since $\bar A$ is a compact subset of $M\smallsetminus \overline{B} \subset M \smallsetminus \{p^+\}$, $p^- \in A$, and $A$ is open,  we see that
$\gamma_n^{-1}(A)\subset A$ for all large enough $n$. Similarly, $\gamma_n(B)\subset B$ for $n$ large. Further, since $p^- \neq p^+$, Lemma \ref{lem:char of loxodromic} implies that for $n$ large $\gamma_n$ is loxodromic and $(\gamma_n^-,\gamma_n^+) \to (p^-,p^+) \in A \times B$. Since $A \times B$ is open, then $(\gamma_n^-,\gamma_n^+) \in A \times B$ for $n$ large. We have achieved a contradiction. 
\end{proof}

 \begin{proof}[Proof of Lemma~\ref{lem:equid m<m'}] 
It suffices to fix $\epsilon'>0$ and show that
 $$\tilde m(A\times B\times I)\leq e^{4\epsilon'} e^{7\epsilon\delta} \tilde m'(A\times B\times I).$$

We start by reducing to the setting of Lemma~\ref{lem:closing}. By considering its interior (which has full $\tilde m,\tilde m'$-measure since the boundary has zero $\tilde m,\tilde m'$-measure), we can assume that $A\times B\times I$ is open. By inner regularity, we find a compact subset $A'\times B'\times I'\subset A\times B\times I$ such that 
 \begin{equation}\label{eqn:inner regularity in Lemma 6.2}
 \tilde m(A\times B\times I)\leq e^{\epsilon'} \tilde m(A'\times B'\times I').
 \end{equation} 
Then by Lemma~\ref{lem:closing} there exists $T_1$ such that: if $\gamma\in\Gamma$, $t\geq T_1$ and 
 $$
 A'\times B'\times I' \cap \psi^{-t}\gamma(A'\times B'\times I') \neq \varnothing,$$
 then $\gamma^{-1}A\times\gamma B\subset A\times B$, 
 and $\gamma$ is loxodromic with $(\gamma^-,\gamma^+)\in A\times B$. 
 We will show that in addition
 \begin{enumerate}[label=(\alph*),series=equid]
  \item \label{item:closing1} $|\ell_\sigma(\gamma)-\sigma(\gamma,b)|\leq \epsilon$ for every $b\in B$,
  \item \label{item:closing2} $|\ell_\sigma(\gamma)-t|\leq 2\epsilon$.
 \end{enumerate}
 In particular, 
\begin{enumerate}[resume*=equid]  
\item \label{item:closing3} $\gamma$ is in $\Theta^{t+2\epsilon} :=\{\gamma\in\Gamma_{\lox}:\ \gamma^{-1}A\times\gamma B\subset A\times B,\ (\gamma^-,\gamma^+) \in A \times B, \ \ell_\sigma(\gamma)\leq t+2\epsilon\}.$
\end{enumerate}
 
 Indeed \ref{item:closing1} is due to assumption \eqref{eq:assumption on G} and the fact that
 \begin{align}\label{eq:l=sigma}
  \ell_\sigma(\gamma)-\sigma(\gamma,b) & = \bar\sigma(\gamma,\gamma^-)+\sigma(\gamma,\gamma^+)  - (\bar\sigma(\gamma,\gamma^-)+\sigma(\gamma,b)) \\
  \nonumber & = G(\gamma\gamma^-,\gamma\gamma^+) - G(\gamma^-,\gamma^+) - G(\gamma\gamma^-,\gamma b) + G(\gamma^-,b) \\
  \nonumber & = G(\gamma^-,b) - G(\gamma^-,\gamma b).
 \end{align}
 
Moreover, if $(x,y,s)\in A'\times B'\times I'\cap \psi^{-t}\gamma(A'\times B'\times I')$ then both $s$ and $s+t-\sigma(\gamma,y)$ are in $I'$ so 
 $$|t-\sigma(\gamma,y)|\le \diam I\leq\epsilon,$$
so \ref{item:closing2} and \ref{item:closing3} hold.

By Observation~\ref{obs:consequence of mixing} (mixing) we may choose $T_2>T_1$ such that  for any $t\geq T_2$ 
 \begin{equation}\label{eqn:mixing applied to A'B'I'}
 \frac{\tilde m(A'\times B'\times I')^2}{\norm{m_\Gamma}} \leq e^{\epsilon'} \sum_{\gamma\in\Gamma} \tilde m(A'\times B'\times I'\cap \psi^{-t}\gamma (A'\times B'\times I')).
 \end{equation}
 Finally,  we may choose $T_3>T_2$,
 so that if $T>T_3$, then
 \begin{equation}\label{eqn:defn of T3}
1\leq e^{\epsilon'}\delta e^{-\delta T}\int_{T_2}^Te^{\delta t}dt.
\end{equation} 
 
 We now fix $(a_0,b_0)\in A\times B$ and $T>T_3$. By Equations~\eqref{eqn:inner regularity in Lemma 6.2} and~\eqref{eqn:defn of T3},
 \begin{align*}
 \frac{\tilde m(A\times B\times I)^2}{\norm{m_\Gamma}} \leq e^{2\epsilon'} \frac{\tilde m(A'\times B'\times I')^2}{\norm{m_\Gamma}} \leq e^{3\epsilon'}  \delta e^{-\delta T} \int_{t=T_2}^T \frac{\tilde m(A'\times B'\times I')^2}{\norm{m_\Gamma}}  e^{\delta t}dt.
 \end{align*} 
 So by Equation~\eqref{eqn:mixing applied to A'B'I'} and ~\ref{item:closing3},
  \begin{align*}
&   \frac{\tilde m(A\times B\times I)^2}{\norm{m_\Gamma}} \leq e^{4\epsilon'} \delta e^{-\delta T} \sum_{\gamma\in\Gamma}\int_{t=T_2}^T \tilde m(A'\times B'\times I'\cap \psi^{-t}\gamma (A'\times B'\times I')) e^{\delta t}dt \\
& \leq e^{4\epsilon'} \delta e^{-\delta T} \sum_{\gamma\in\Theta^{T+2\epsilon} }\int_{t=T_2}^T \tilde m(A'\times B'\times I'\cap \psi^{-t}\gamma (A'\times B'\times I')) e^{\delta t}dt.
\end{align*}
Notice that assumption~\eqref{eq:assumption on G} implies that the integrand satisfies 
 \begin{align*}
\tilde m & (A'\times B'\times I'\cap \psi^{-t}\gamma (A'\times B'\times I')) \leq \tilde m(A\times \gamma B\times I) \\ 
& \leq e^{\delta \epsilon}e^{\delta G(a_0, b_0)} \bar\mu(A) \mu(\gamma B) \mr{Leb}(I). 
 \end{align*}
 Further, by  \ref{item:closing1},
 \begin{align*}
  \mu(\gamma B) =  \int_B e^{-\delta \sigma(\gamma, b)} d\mu \leq e^{\epsilon\delta}e^{-\delta\ell_\sigma(\gamma)}\mu(B).
  \end{align*}
  By \ref{item:closing2}, 
  \begin{align*}
  \{t : A'\times B'\times I'\cap \psi^{-t}\gamma (A'\times B'\times I') \neq \varnothing\} \subset [\ell_\sigma(\gamma)-2\epsilon, \ell_\sigma(\gamma)+2\epsilon]
  \end{align*}
  and has Lebesgue measure at most $\mr{Leb}(I)$. Hence 
  \begin{align*}
  \frac{\tilde m(A\times B\times I)^2}{\norm{m_\Gamma}}  \leq e^{4\epsilon'} e^{4\epsilon\delta}  
      \cdot e^{\delta G(a_0,b_0)} \bar\mu(A) \mu(B) \mr{Leb}(I) \cdot \delta e^{-\delta T}\mr{Leb}(I)\# \big(\Theta^{T+2\epsilon}\big).
  \end{align*}
So by assumption \eqref{eq:assumption on G} and the definition of $\tilde{m}_{T+2\epsilon}$, 
  \begin{align*}
  \frac{\tilde m(A\times B\times I)^2}{\norm{m_\Gamma}}  \leq e^{4\epsilon'} e^{5\epsilon\delta}  \tilde m(A \times B \times I) \delta e^{-\delta T}\tilde{m}_{T+2\epsilon}(A\times B\times I).
  \end{align*}

By applying the above estimate to the sequence $\{T_n-2\epsilon\}$, we see that,
for all sufficiently large $n$,
$$\tilde m(A\times B\times I)^2\leq e^{4\epsilon'} e^{5\epsilon\delta}\tilde m(A\times B\times I)\cdot e^{2\delta\epsilon}
\left(\norm{m_\Gamma}\delta e^{-\delta T_n}\tilde{m}_{T_n}(A\times B\times I)\right).$$
Letting $n\to\infty$, we see that
\begin{equation*}
  \tilde m(A\times B\times I)\leq e^{4\epsilon'} e^{7\epsilon\delta} \tilde m'(\bar A\times \bar B\times\bar  I)=e^{4\epsilon'} e^{7\epsilon\delta} \tilde m'( A\times  B\times I).
\end{equation*}
Notice that the last equality used the fact  that $A\times B\times I$ has $\tilde m'$-null boundary.
 \end{proof}

 \begin{proof}[Proof of Lemma~\ref{lem:equid m'<m}] The proof is very similar to the proof of Lemma~\ref{lem:equid m<m'}.

 It suffices to fix $\epsilon'>0$ and show that
 $$\tilde m'(A\times B\times I)\leq e^{2\epsilon'} e^{7\epsilon\delta} \tilde m(A\times B\times I).$$
As in the proof of Lemma~\ref{lem:equid m<m'}, we can assume that $A\times B\times I$ is open and find a compact subset $A'\times B'\times I'\subset A\times B\times I$ with $\tilde m'$-null boundary such that 
 \begin{equation}\label{eqn:inner regularity in Lemma 6.3}
 \tilde m'(A\times B\times I)\leq e^{\epsilon'} \tilde m'(A'\times B'\times I').
 \end{equation}  
 
For $S<T$, define
 $$\Gamma_S^T :=\{\gamma\in\Gamma_{\lox}:\ (\gamma^-,\gamma^+)\in A'\times B',\ S\leq \ell_\sigma(\gamma)\leq T\}.$$
 Note that if $\gamma\in\Gamma_{\lox}$ and $(\gamma^-,\gamma^+)\in A'\times B'$, then 
  $$
 A'\times B'\times I' \cap \psi^{-\ell_\sigma(\gamma)}\gamma(A'\times B'\times I') \neq \varnothing.
 $$
So by Lemma~\ref{lem:closing} there exists $T_1$ such that if $\gamma\in\Gamma^\infty_{T_1}$, then
 \begin{enumerate}[resume*=equid]
  \item \label{item:closing0bis} $\gamma^{-1}A\times\gamma B\subset A\times B$.
   \end{enumerate}
   Moreover: 
    \begin{enumerate}[resume*=equid]
  \item \label{item:closing1bis} $|\ell_\sigma(\gamma)-\sigma(\gamma,b)|\leq \epsilon$ for every $b\in B$.
  \item \label{item:closing2bis} If $A\times B\times I \cap \psi^{-t}\gamma(A\times B\times I) \neq \varnothing$ for some $t$, then $|t-\ell_\sigma(\gamma)|\leq 2\epsilon$.
 \end{enumerate}
 Notice that \ref{item:closing1bis} follows from the computation in \eqref{eq:l=sigma} and  \ref{item:closing2bis} is an immediate consequence of \ref{item:closing1bis} and the fact that $\diam I\leq \epsilon$.

 By Observation~\ref{obs:consequence of mixing} (mixing) we may choose $T_2>T_1$ such that  for any $t\geq T_2$ 
 $$
 \frac{\tilde m(A\times B\times I)^2}{\norm{m_\Gamma}} \geq e^{-\epsilon'} \sum_{\gamma\in\Gamma} \tilde m(A\times B\times I\cap \psi^{-t}\gamma (A\times B\times I)).
 $$
 
  We now fix $(a_0,b_0)\in A\times B$ and $T>T_2$. Using the fact that $\delta e^{-\delta T} \int_{T_2}^T e^{\delta t}dt < 1$ and ~\ref{item:closing2bis}, 
    \begin{align*}
  &\frac{\tilde m(A\times B\times I)^2}{\norm{m_\Gamma}}  
  \geq \delta e^{-\delta T} \int_{t=T_2}^T \frac{\tilde m(A\times B\times I)^2}{\norm{m_\Gamma}}   e^{\delta t}\ dt \\ \noalign{\vskip0pt}
  & \geq e^{-\epsilon'} \delta e^{-\delta T} e^{-2\delta \epsilon} \int_{t=T_2}^T\sum_{\gamma\in\Gamma^{T-2\epsilon}_{T_2+2\epsilon}} {\tilde m(A\times B\times I\cap \psi^{-t}\gamma (A\times B\times I))}e^{\delta \ell_\sigma(\gamma)} \ dt.
      \end{align*}
            By ~\ref{item:closing0bis}, in the integrand we have $A \cap \gamma A = A$ and $B \cap \gamma B = \gamma B$, so by assumption \eqref{eq:assumption on G} we have
      $$
      \tilde m (A\times B\times I\cap \psi^{-t}\gamma (A\times B\times I)) \geq e^{-\delta \epsilon} e^{\delta G(a_0, b_0)} \bar\mu(A) \int_{\gamma B} {\rm Leb}(I \cap  (I-t+\sigma(\gamma,b)))d\mu(b). 
      $$
      Also, since $\gamma \in \Gamma^{T-2\epsilon}_{T_2+2\epsilon}$, property~\ref{item:closing1bis} implies that
      \begin{align*}
       \int_{t=T_2}^T  & {\rm Leb}(I \cap  (I-t+\sigma(\gamma,b)))dt =\int_{s \in I} {\rm Leb}([T_2,T] \cap (I-s+\sigma(\gamma, b)) )ds \\
       & =\int_{s \in I} {\rm Leb}(I-s+\sigma(\gamma, b))ds= {\rm Leb}(I)^2. 
       \end{align*}
       Combining the above estimates we have, 
       \begin{align*}
       &\frac{\tilde m(A\times B\times I)^2}{\norm{m_\Gamma}}  \geq  e^{-\epsilon'} \delta e^{-\delta T} e^{-3\delta \epsilon}e^{\delta G(a_0, b_0)} \bar\mu(A)  {\rm Leb}(I)^2\sum_{\gamma\in\Gamma^{T-2\epsilon}_{T_2+2\epsilon}} e^{\delta \ell_\sigma(\gamma)}\mu(\gamma B).
       \end{align*} 
       By~\ref{item:closing1bis}, we have 
       $$
       \mu(\gamma B) = \int_{B} e^{-\delta \sigma(\gamma, b)} d \mu(b) \geq e^{-\delta \epsilon} e^{-\delta \ell_\sigma(\gamma)} \mu(B). 
       $$
       Hence 
       $$
       \frac{\tilde m(A\times B\times I)^2}{\norm{m_\Gamma}} \geq e^{-\epsilon'} e^{-4\delta \epsilon} e^{\delta G(a_0, b_0)} \bar\mu(A) \mu(B){\rm Leb}(I) \cdot \delta e^{-\delta T} \#\left(\Gamma^{T-2\epsilon}_{T_2+2\epsilon}\right) {\rm Leb}(I).
       $$
       By assumption \eqref{eq:assumption on G}, 
       $$
       e^{\delta G(a_0, b_0)} \bar\mu(A) \mu(B){\rm Leb}(I) \geq e^{-\delta \epsilon}\tilde m(A \times B \times I)
       $$
       and by definition 
       $$
       \delta e^{-\delta T} \#\left(\Gamma^{T-2\epsilon}_{T_2+2\epsilon}\right){\rm Leb}(I) = \delta e^{-\delta T} \left(\tilde{m}_{T-2\epsilon} - \tilde{m}_{T_2+2\epsilon}\right)(A'\times B' \times I).
       $$
       So by applying the above estimates to the sequence $\{T_n+2\epsilon\}$, we see that,
for all sufficiently large $n$,
$$\tilde m(A\times B\times I)^2\geq e^{-\epsilon'} e^{-5\epsilon\delta}\tilde m(A\times B\times I)\cdot e^{-2\delta\epsilon}
\left(\norm{m_\Gamma}\delta e^{-\delta T_n}(\tilde{m}_{T_n}-\tilde{m}_{T_2+2\epsilon})(A'\times B'\times I)\right).$$
Sending $n \to \infty$ and using Equation~\eqref{eqn:inner regularity in Lemma 6.3} and that $A'\times B'\times I'$ has $\tilde m'$-null boundary, we obtain 
   $$
   \tilde m(A\times B\times I) \geq     e^{-\epsilon'} e^{-7\epsilon\delta} \tilde m'(A' \times B' \times I') \geq e^{-2\epsilon'}e^{-7\epsilon\delta} \tilde m'(A \times B \times I), 
   $$
which completes the proof. 
\end{proof}


\section{Structure of the flow space}\label{sec:structure of the flow space}


For the rest of the section fix a continuous GPS system $(\sigma,\bar\sigma,G)$ for a convergence group $\Gamma\subset\mathsf{Homeo}(M)$. In this section we study the flow space $M^{(2)} \times \Rb$ and the flow invariant subset $\tilde U_\Gamma = \Lambda(\Gamma)^{(2)} \times \Rb$. Our main aim is to show that, in the geometrically finite case, the quotient $U_\Gamma = \Gamma \backslash \tilde U_\Gamma$ decomposes into a compact part and a finite number of ``cusps.''

To accomplish this we will first explain how to compactify $M^{(2)} \times \Rb$ using $M$ and then define a notion of horoballs based at bounded parabolic points. 

\subsection{Compactifying the flow space} 

In the arguments that follow it will be helpful to observe that the space $M^{(2)} \times \Rb$ can be compactified using $M$. 
 
\begin{proposition}\label{prop:compactifying flowspace} 
The set $\overline{M^{(2)}\times \Rb } : = \left( M^{(2)}\times \Rb\right) \sqcup M$ has a topology with the following properties: 
\begin{enumerate}
\item $\overline{M^{(2)}\times \Rb }$ is a compact metrizable space.
\item The inclusions $M^{(2)}\times \Rb  \hookrightarrow \overline{M^{(2)}\times \Rb } $ and $M \hookrightarrow \overline{M^{(2)}\times \Rb } $ are embeddings, and $M^{(2)}\times \Rb$ is open.
\item The action of $\Gamma$ on $\overline{M^{(2)}\times \Rb }$ is by homeomorphisms. 
\item\label{item:sequences in compactification of flow space} A sequence $\{v_n=(x_n,y_n,t_n)\}$ in $\overline{M^{(2)}\times \Rb }$ converges to  $z \in M$ if and only if every subsequence $\{ v_{n_j}\}$ has a further subsequence $\{ v_{n_{j_k}}\}$where  either
\begin{enumerate}
\item
$x_{n_{j_k}} \to z$ and $t_{n_{j_k}} \to -\infty$,
\item
$y_{n_{j_k}} \to z$ and  $t_{n_{j_k}} \to +\infty$, or
\item
$x_{n_{j_k}} \to z$  and $y_{n_{j_k}} \to z$.
\end{enumerate}
\end{enumerate}
\end{proposition}

The proof is an exercise in point set topology (namely showing that the convergence in ~\eqref{item:sequences in compactification of flow space} is compatible with a topology) and so we only sketch the argument. 
\begin{proof}[Proof sketch]For an open set $U \subset M$, let 
$$
I_{U,n} : = \left\{ (x,y,t): (x,t) \in U \times (-\infty,-n) \text{ or }   (y,t) \in U \times (n,\infty) \text{ or }  x,y\in U \right\}.
$$
Fix a countable basis $\Bc_0$ of the topology on $M^{(2)} \times \Rb$ and a countable basis $\Bc_1$ of the topology on $M$. Let 
$$
\Bc : = \Bc_0 \cup \{ I_{U,n} \cup U : U \in \Bc_1, n \in \Nb\}
$$
and endow $\overline{M^{(2)} \times \Rb}$ with the topology generated by $\Bc$, which is a countable basis.
One can check the topology is regular,
so by Urysohn's metrization theorem it is metrizable.
By the definition of the $I_{U,n}$, this topology satisfies~\eqref{item:sequences in compactification of flow space}, which implies it is compact.
Properties (2) and (3)  follow in a straightforward way from the definition of $\Bc$. 
\end{proof} 

We also observe that $\Gamma$ acts on $\overline{M^{(2)} \times \Rb}$ as a convergence group. 

\begin{proposition}\label{prop:convergence action on compactification of flow space} If $a,b \in M$, $\{\gamma_n\} \subset \Gamma$ and $\gamma_n|_{M \smallsetminus \{b\}}$ converges locally uniformly to $a$, then $\gamma_n|_{\overline{M^{(2)} \times \Rb}\smallsetminus \{b\}}$ converges locally uniformly to $a$. Hence, $\Gamma$ acts on $\overline{M^{(2)} \times \Rb}$ as a convergence group. 
\end{proposition} 

\begin{proof} Suppose the first assertion is false. Then there exists a relatively compact sequence $\{v_n\} \subset \overline{M^{(2)} \times \Rb}\smallsetminus \{b\}$ such that $\gamma_n v_n \rightarrow c \neq a$. 

Passing to a subsequence we can assume that either $\{v_n\} \subset M$ or $\{v_n\} \subset M^{(2)} \times \Rb$. If $\{v_n\} \subset M$, then $\gamma_nv_n \rightarrow a$ since $\gamma_n|_{M \smallsetminus \{b\}}$ converges locally uniformly to $a$. So we must have $\{v_n\} \subset M^{(2)} \times \Rb$. 

Let $v_n = (x_n, y_n, t_n)$. Passing to a subsequence we can suppose that $x_n \rightarrow x$ and $y_n \rightarrow y$. Since $b$ is not a limit point of $\{v_n\}$, Proposition~\ref{prop:compactifying flowspace}\eqref{item:sequences in compactification of flow space} implies that $(x,y) \neq (b,b)$. 

\medskip

\noindent \emph{Case 1:} Assume $x \neq b$ and $y \neq b$. Then $\gamma_n x_n \rightarrow a$ and $\gamma_n y_n \rightarrow a$. So Proposition~\ref{prop:compactifying flowspace}\eqref{item:sequences in compactification of flow space} implies that $\gamma_n v_n \rightarrow a$, which is a contradiction. 

\medskip

\noindent \emph{Case 2:} Assume $x = b$ and $y \neq b$. Since $b$ is not a limit point of $\{v_n\}$, Proposition~\ref{prop:compactifying flowspace}\eqref{item:sequences in compactification of flow space} implies that $\{t_n\}$ is bounded below. Since $\sigma$ is expanding and $\gamma_n^{-1} \rightarrow b$, there exists $C > 0$ such that 
$$
\sigma(\gamma_n, y_n) \geq \norm{\gamma_n}_\sigma - C
$$
for all $n \geq 1$. Then Proposition~\ref{prop:basic properties}\eqref{item:properness} implies that $\sigma(\gamma_n, y_n) \rightarrow +\infty$. Since we also have $\gamma_n y_n \rightarrow a$, Proposition~\ref{prop:compactifying flowspace}\eqref{item:sequences in compactification of flow space} implies that 
$$
\gamma_n v_n = (\gamma_n x_n, \gamma_n y_n, t_n + \sigma(\gamma_n, y_n))
$$
converges to $a$, which is a contradiction. 

\medskip

\noindent \emph{Case 3:} Assume $x \neq b$ and $y = b$. Then $\gamma_n x_n \rightarrow a$. Since $\gamma_n v_n \rightarrow c \neq a$, Proposition~\ref{prop:compactifying flowspace}\eqref{item:sequences in compactification of flow space} implies that $a$ is not a limit point of $\{\gamma_n y_n\}$. Then, since $\sigma$ is expanding and $\gamma_n \rightarrow a$, there exists $C > 0$ such that 
$$
\sigma(\gamma_n^{-1}, \gamma_ny_n) \geq \norm{\gamma_n^{-1}}_\sigma - C
$$
for all $n \geq 1$.

Also, since $b$ is not a limit point of $\{v_n\}$, Proposition~\ref{prop:compactifying flowspace}\eqref{item:sequences in compactification of flow space} implies that $\{t_n\}$ is bounded above. Then Proposition~\ref{prop:basic properties}\eqref{item:properness} implies that
$$
t_n + \sigma(\gamma_n, y_n) = t_n - \sigma(\gamma_n^{-1}, \gamma_ny_n) \leq t_n -  \norm{\gamma_n^{-1}}_\sigma + C
$$
converges to $-\infty$. Since we also have $\gamma_n x_n \rightarrow a$, Proposition~\ref{prop:compactifying flowspace}\eqref{item:sequences in compactification of flow space} implies that 
$$
\gamma_n v_n = (\gamma_n x_n, \gamma_n y_n, t_n + \sigma(\gamma_n, y_n))
$$
converges to $a$, which is a contradiction.
\end{proof} 

\subsection{Horoballs}\label{sec:horoballs} In this subsection we define horoballs based at  bounded parabolic point, in the spirit of \cite{tukia-conical}. 
To that end fix a bounded parabolic point $p \in M$.
Then let $K \subset \Lambda(\Gamma) \smallsetminus \{p\}$ be a compact subset such that $\Gamma_p \cdot K = \Lambda(\Gamma) \smallsetminus \{p\}$ (recall that $\Gamma_p = {\rm Stab}_\Gamma(p)$). 
Then given a compact neighborhood $N$ of $K$ in $\overline{M^{(2)}\times\R}\smallsetminus\{p\}$, the \emph{horoball associated to $N$ based at $p$} is the subset 
\[H(p,N) = \tilde U_\Gamma\smallsetminus (\Gamma_p\cdot N).\]
 
 Directly from the definition, horoballs have the following properties. 
 
 \begin{observation} If $\gamma \in \Gamma$, then 
 $$
 \gamma H(p,N)= H(\gamma p, \gamma N). 
 $$
 Moreover, if $\gamma \in \Gamma_p$, then 
  $$
 \gamma H(p, N)= H(p, N).
 $$
 Finally, if $N \subset N'$, then 
 $$
 H(p, N') \subset H(p, N).
 $$
 \end{observation} 

We also verify that the horoballs are open. 

\begin{proposition}\label{prop:horoballs are open} Any horoball $H(p,N)$ is open in $M^{(2)} \times \Rb$. 
\end{proposition} 

\begin{proof} Proposition~\ref{prop:convergence action on compactification of flow space} implies that $\Gamma_p$ acts properly on $\overline{M^{(2)} \times \Rb} \smallsetminus \{p\}$. Thus $\Gamma_p \cdot N$ is closed in $\overline{M^{(2)} \times \Rb} \smallsetminus \{p\}$, and hence $H(p,N)$ is open in $M^{(2)} \times \Rb$. 

\end{proof}

\subsection{Decomposing the flow space}\label{sec:decomposing the flow space}

We now make the additional assumption that $\Gamma\subset\mathsf{Homeo}(M)$ is geometrically finite. Our aim in this subsection is to show that the flow space $U_\Gamma = \Gamma \backslash \tilde U_\Gamma = \Gamma \backslash \Lambda(\Gamma)^{(2)} \times \Rb$ decomposes into a compact part and a finite union of ``cusps'' corresponding to the bounded parabolic points.

By Lemma~\ref{lem:finitely many parabs} there are finitely many $\Gamma$-orbits of bounded parabolic points.
Let $\{p_1,\dots,p_k\}$ contain one representative of each orbit.
For each $1 \leq j \leq k$, fix a compact subset $K_j \subset \Lambda(\Gamma) \smallsetminus \{p_j\}$ such that $\Gamma_{p_j} \cdot K_j =  \Lambda(\Gamma) \smallsetminus \{p_j\}$.
Finally, for each $1 \leq j \leq k$ fix a compact neighborhood $N_j$  of $K_j$ in $\overline{M^{2}\times\R}\smallsetminus\{p_j\}$.
Then let 
 $$
 \Hc := \bigcup_{\gamma \in \Gamma} \bigcup_{j=1}^k \gamma H(p_j, N_j). 
 $$
 
 We will show that $\Gamma$ acts cocompactly on the complement of $\Hc$.
 
  \begin{theorem}\label{thm:compact quotient} The quotient $\Gamma \backslash\left(\tilde U_\Gamma -\Hc\right)$ is compact.
 \end{theorem} 
 
  Notice that Theorem~\ref{thm:compact quotient} shows that $U_\Gamma = \Gamma \backslash \tilde U_\Gamma$ decomposes into a compact part, namely  $ \Gamma \backslash\left(\tilde U_\Gamma -\Hc\right)$, and $k$ ``cusps'', namely the image of each horoball $ H(p_j, N_j)$. 
Also notice that the theorem is true for any choice of $\{N_j\}$. 

For later use, we will also observe that one can assume that $\Hc$ consists of disjoint horoballs. 

\begin{proposition}\label{prop:disjoint horoballs}
The $N_j$'s can be chosen so that if $\alpha p_j\neq\beta p_i$, then 
$$
\alpha H(p_j,N_j) \cap \beta H(p_i,N_i)= \varnothing.
$$
\end{proposition}

We will first prove Proposition~\ref{prop:disjoint horoballs} and then prove Theorem~\ref{thm:compact quotient}.

\subsection{Proof of Proposition~\ref{prop:disjoint horoballs}} 

It suffices to prove the following:
given two distinct parabolic points $p,q$, there are sets $N$ and $N'$ such that if $\gamma q\neq p$, then $H(p,N)$ is disjoint from $\gamma H(q,N')$.

Fix a distance on $\overline{M^{2}\times\R}$ which generates the topology and let $B(x,r)$ denote the associated open balls. 

Fix $\epsilon_n\searrow 0$. For $n \geq 1$, let 
$$
N_n : = \overline{M^{(2)}\times\R} \smallsetminus B(p,\epsilon_n) \quad \text{and} \quad N_n':=\overline{M^{(2)}\times\R} \smallsetminus B(q,\epsilon_n). 
$$
We claim that $N_n$ and $N_n'$ have the desired property when $n$ is large. Suppose not. Then after passing to a subsequence there exists $\{\gamma_n\}\subset \Gamma$ such that $\gamma_n q\neq p$ and the horoballs $H(p,N_n)$ and $\gamma_n H(q,N'_n)$ intersect nontrivially for all $n \geq 1$. 

Since $\Gamma_p$ acts cocompactly on $\Lambda(\Gamma) \smallsetminus \{p\}$ and $\Gamma_q$ acts cocompactly on $\Lambda(\Gamma) \smallsetminus \{q\}$, we can pass to a subsequence and replace each $\gamma_n$ by an element in $\Gamma_q \gamma_n \Gamma_p$ so that $\gamma_n p \rightarrow x \neq q$ and $\gamma_n^{-1} q \rightarrow y \neq p$.

Notice that 
$$
H(p,N_n)  \subset \tilde U_\Gamma \smallsetminus N_n \subset B(p,\epsilon_n) \quad \text{and} \quad H(q,N'_n)\subset \tilde U_\Gamma \smallsetminus N_n' \subset B(q,\epsilon_n).
$$
So $B(p,\epsilon_n) \cap \gamma_n B(q,\epsilon_n) \neq \varnothing$ for all $n \geq 1$, which implies that $\{\gamma_n\}$ must be escaping. 

Since $p$ is not a conical limit point (see Lemma~\ref{lem: parab conical disjoint}) and $\gamma_n p \rightarrow x$ we must have $\gamma_n \rightarrow x$. Likewise, using the fact that $q$ is not a conical limit point, we must have $\gamma_n^{-1} \rightarrow y$. Then, Proposition~\ref{prop:convergence action on compactification of flow space} implies that $\gamma_n|_{\overline{M^{(2)} \times \Rb}\smallsetminus\{x\}}$ converges locally uniformly to $y$. Since $p \neq y$ and $q \neq x$,  for $n$ sufficiently large 
$$
H(p,N_n) \cap \gamma_n H(q,N'_n) \subset B(p,\epsilon_n) \cap \gamma_n B(q,\epsilon_n)  = \varnothing,
$$
which is a contradiction.

\subsection{Proof of Theorem~\ref{thm:compact quotient}} 
By Proposition~\ref{prop:exists compact covering set in limit cross limit} there exists a compact subset $L \subset \Lambda(\Gamma)^{(2)}$ such that $\Gamma \cdot {\rm int}(L) = \Lambda(\Gamma)^{(2)}$. Then let 
 $$
 \mc D_L := \left\{v\in L\times\mb R: \tau(v)\leq \tau(v') \text{ for all } v' \in \Gamma v \cap (L \times \Rb) \right\}
 $$
 where 
$$
\tau(x,y,t) =\abs{t}. 
$$
That is, $ \mc D_L$ consists of the points in each $\Gamma$-orbit that are in $L \times \Rb$ and minimize the ``$t$-value'' amongst all points in the $\Gamma$-orbit in $L \times \Rb$. 

\begin{observation}\label{obs:DK tiles} $\Gamma ( \mc D_L) = \Lambda(\Gamma)^{(2)} \times \Rb$. \end{observation} 

\begin{proof} Fix $v \in \Lambda(\Gamma)^{(2)} \times \Rb$. By definition, $ \Gamma v \cap (L \times \Rb) \neq \varnothing$. Since $\Gamma$ acts properly on $\Lambda(\Gamma)^{(2)} \times \Rb$, for any $T > 0$ the set $\Gamma v \cap (L \times [-T, T])$ is finite. Hence  there exists $v' \in \Gamma v \cap (L \times \Rb)$ such that 
  $$
  \tau(v') = \min\{ \tau(\gamma v) : \gamma \in \Gamma \text{ and } \gamma v \in L \times \Rb\}.
  $$
  Then $v' \in \mc D_L$ and $v \in \Gamma v' \subset \Gamma( \mc D_L) $. 
  \end{proof} 
  
  We will prove the following facts about the closure of  $\mathcal{D}_L$ in $\overline{M^{(2)} \times \Rb}$. 

\begin{lemma}\label{lem:closure of DK has no conical limit points}
The closure of $\mathcal{D}_L$ in $\overline{M^{(2)} \times \Rb}$ does not contain any conical limit points.
\end{lemma}

\begin{lemma}\label{lem:closure of DK has no bounded parabolic points}
The closure of $\mathcal{D}_L \smallsetminus \Hc$ in $\overline{M^{(2)} \times \Rb}$ does not contain any bounded parabolic points. 

\end{lemma} 

Assuming the lemmas for a moment, we prove Theorem~\ref{thm:compact quotient}.

\begin{proof}[Proof of Theorem~\ref{thm:compact quotient}] By Observation~\ref{obs:DK tiles}, the projection of $\mathcal{D}_L \smallsetminus \Hc$ to $U_\Gamma$ is 
 $$
 \Gamma \backslash\left(\tilde U_\Gamma -\Hc\right).
 $$
 By definition $\mc D_L$ is closed and Proposition~\ref{prop:horoballs are open} implies that $\Hc$ is open. So  $\mathcal{D}_L \smallsetminus \Hc$ is closed. So it suffices to show that $\mathcal{D}_L \smallsetminus \Hc$ does not accumulate on any points in $M = \overline{M^{(2)} \times \Rb} \smallsetminus M^{(2)} \times \Rb$. 
 
Since $\mc D_L \subset \Lambda(\Gamma)^{(2)} \times \Rb$,  Proposition~\ref{prop:compactifying flowspace}\eqref{item:sequences in compactification of flow space} implies that the closure of  $\mathcal{D}_L \smallsetminus \Hc$ in $\overline{M^{(2)} \times \Rb}$ is contained in $\Lambda(\Gamma)^{(2)} \times \Rb \sqcup \Lambda(\Gamma)$. Then, since every point in $\Lambda(\Gamma)$ is either conical or bounded parabolic, Lemmas~\ref{lem:closure of DK has no conical limit points} and~\ref{lem:closure of DK has no bounded parabolic points} imply that $\mathcal{D}_L \smallsetminus \Hc$ is compact. 
\end{proof} 

\subsection{Proof of Lemma~\ref{lem:closure of DK has no conical limit points}} Suppose $z\in M$ is a conical limit point and $\{v_n\} \subset L \times \Rb$ converges to $z$. We will show that $v_n \notin \mathcal{D}_L$ when $n$ is large. 

Since $z$ is a conical limit point, there exist $\{\gamma_m\}$ in $\Gamma$ and distinct $a, b\in M$ such that $\gamma_mz\to a$ and $\gamma_m|_{M \smallsetminus \{z\}} \to b$  locally uniformly. 
Thus $$
\gamma_m \to b \quad \text{and} \quad \gamma_m^{-1} \to z
$$ 
in the topology on $\Gamma \sqcup M$.

Let $v_n = (x_n, y_n, t_n)$. Since $(x_n,y_n) \in L$ and $v_n \to z$, by Proposition~\ref{prop:compactifying flowspace}\eqref{item:sequences in compactification of flow space} we must have
 $$
 (y_n \to z \text{ and } t_n \to +\infty) \text{ or }  (x_n \to z \text{ and } t_n \to -\infty).
 $$
 
 \medskip
 
 \noindent \emph{Case 1:} Assume $y_n \to z$ and $t_n \to +\infty$. Fix a distance $\dist$ on $\overline{M^{(2)} \times \Rb}$ which generates the topology. Since $v_n \to z$ and $y_n \to z$, we can find $n_1 < n_2 < \dots$ such that 
 $$
 \dist(\gamma_m(z), \gamma_m(v_{n_m})) < \frac{1}{m} \quad \text{and}  \quad \dist(\gamma_m(z), \gamma_m(y_{n_m})) < \frac{1}{m}
 $$
 for all $m \geq 1$.  Then 
 $$
 \lim_{m \to \infty}  \gamma_m(v_{n_m}) =  \lim_{m \to \infty}  \gamma_m(z) = a \quad \text{and} \quad  \lim_{m \to \infty}  \gamma_m(y_{n_m}) =  \lim_{m \to \infty}  \gamma_m(z) = a.
 $$
 So by replacing $\{v_n\}$ with $\{v_{n_m}\}$, we can assume that $\gamma_n v_n \to a$ and $\gamma_n y_n \to a$. Also, since $y_n \to z$ and $(x_n,y_n) \in L$, any limit point of $\{ x_n\}$ is contained in $M \smallsetminus \{ z\}$. So $\gamma_n x_n \to b \neq a$.  Then, since 
$$
\gamma_n v_n=(\gamma_n x_n,\gamma_n y_n, t_n + \sigma(\gamma_n, y_n))\to a,
$$
Proposition~\ref{prop:compactifying flowspace}\eqref{item:sequences in compactification of flow space} implies that $t_n + \sigma(\gamma_n, y_n) \to +\infty$. 
 
Then for $n$ sufficiently large,
\begin{align*}
\tau(\gamma_n v_n)-\tau(v_n) & = \abs{t_n + \sigma(\gamma_n, y_n)} -\abs{t_n} = t_n + \sigma(\gamma_n, y_n)-t_n \\
& = -\sigma(\gamma_n^{-1}, \gamma_n y_n).
\end{align*}
Since $\gamma_n \to b$ and $\gamma_n y_n \to a \neq b$, by the expanding property there exists $C > 0$ such that 
$$
\sigma(\gamma_n^{-1}, \gamma_n y_n) \geq \norm{\gamma_n^{-1}}_\sigma-C
$$
for all $n \geq 1$. Then Proposition~\ref{prop:basic properties}\eqref{item:properness} implies that 
$$
\lim_{n \to \infty} \tau(\gamma_n v_n)-\tau(v_n) = -\infty. 
$$

Finally, fix $\alpha \in \Gamma$ such that $\alpha (b,a) \in {\rm int}(L)$. Since $\gamma_n(x_n,y_n) \to (b,a)$, then $\alpha\gamma_n v_n \in L \times \Rb$ for $n$ large. Further, 
$$
\lim_{n \to \infty} \tau(\alpha\gamma_n v_n)-\tau(v_n) \leq \max_{q \in \Lambda(\Gamma)} \abs{ \sigma(\alpha, q)} + \lim_{n \to \infty} \tau(\gamma_n v_n)-\tau(v_n)= -\infty. 
$$
So $v_n \notin \mc D_L$ when $n$ is large.

\medskip 

 \noindent \emph{Case 2:} Assume $x_n \to z$ and $t_n \to -\infty$.  The proof is very similar to the proof of Case 1. Since $v_n \to z$ and $x_n \to z$, arguing as in Case 1 we can replace $\{v_n\}$ with a subsequence so that $\gamma_n v_n \to a$ and $\gamma_n x_n \to a$. Also, since $x_n \to z$ and $(x_n,y_n) \in L$, any limit point of $\{ y_n\}$ is contained in $M \smallsetminus \{ z\}$. So $\gamma_n y_n \to b \neq a$.  Then, since 
$$
\gamma_n v_n=(\gamma_n x_n,\gamma_n y_n, t_n + \sigma(\gamma_n, y_n))\to a,
$$
 Proposition~\ref{prop:compactifying flowspace}\eqref{item:sequences in compactification of flow space} implies that $t_n + \sigma(\gamma_n, y_n) \to -\infty$. 
 
 Then for $n$ sufficiently large,
\begin{align*}
\tau(\gamma_n v_n)-\tau(v_n) & = \abs{t_n + \sigma(\gamma_n, y_n)} -\abs{t_n} = -t_n - \sigma(\gamma_n, y_n)+t_n \\
& = -\sigma(\gamma_n, y_n).
\end{align*}
Since $\gamma_n^{-1} \to z$ and any limit point of $\{ y_n\}$ is contained in $M \smallsetminus \{ z\}$, by the expanding property there exists $C > 0$ such that 
$$
\sigma(\gamma_n,  y_n) \geq \norm{\gamma_n}_\sigma-C
$$
for all $n \geq 1$. Then Proposition~\ref{prop:basic properties}\eqref{item:properness} implies that 
$$
\lim_{n \to \infty} \tau(\gamma_n v_n)-\tau(v_n) = -\infty. 
$$

Finally, fix $\alpha \in \Gamma$ such that $\alpha(a,b) \in {\rm int}(L)$. Since $\gamma_n(x_n,y_n) \to (a,b)$, then $\alpha\gamma_n v_n \in K \times \Rb$ for $n$ large. Further, 
$$
\lim_{n \to \infty} \tau(\alpha\gamma_n v_n)-\tau(v_n) \leq \max_{q \in \Lambda(\Gamma)} \abs{ \sigma(\alpha, q)} + \lim_{n \to \infty} \tau(\gamma_n v_n)-\tau(v_n)= -\infty. 
$$
So $v_n \notin \mc D_L$ when $n$ is sufficiently large.

\subsection{Proof of Lemma~\ref{lem:closure of DK has no bounded parabolic points}} 
Suppose $p$ is a bounded parabolic point and $\{v_n\} \subset (L \times \Rb) \smallsetminus \Hc$ converges to $p$. 
We will show that $v_n \notin \mathcal{D}_L$ when $n$ is large.

Let $v_n = (x_n, y_n, t_n)$.
Passing to a subsequence we can assume $x_n \to x$ and $y_n \to y$. 
Since $(x_n,y_n) \in L$ and $L$ is compact, we must have $x \neq y$. 
Then, since $v_n \to p$, Proposition~\ref{prop:compactifying flowspace}\eqref{item:sequences in compactification of flow space} implies that 
 \begin{equation*}
 (y=p, \ x \neq p \text{ and } t_n \to +\infty) \text{ or }  (x = p, \ y \neq p \text{ and } t_n \to -\infty).
 \end{equation*}

Let $H(p,N)$ be the horoball of $\Hc$ at $p$. Since $v_n\not\in H(p,N)$, we have $v_n\in\gamma_n N$ for some $\gamma_n\in\Gamma_p$. Passing to a subsequence, we can assume $\gamma_n^{-1}x_n\to x'$ and $\gamma_n^{-1}y_n\to y'$.

 \medskip

 \noindent \emph{Case 1:} Assume $x=p$, $y \neq p$ and $t_n \to -\infty$.
Since $\{\gamma_n\}\subset\Gamma_p$ is escaping, $\gamma_n|_{M \smallsetminus \{p\}}$ converges locally uniformly to $p$. So we have $\gamma_n^{-1}y_n\to p=y'$. Since $p \notin \overline{N} = N$ is not a limit point of 
$$
\{\gamma_n^{-1}v_n\}=\{(\gamma_n^{-1}x_n,\gamma_n^{-1}y_n,t_n+\sigma(\gamma_n^{-1},y_n))\} \subset N,
$$
Proposition~\ref{prop:compactifying flowspace}\eqref{item:sequences in compactification of flow space} implies that $x'\neq p$ and $t_n+\sigma(\gamma_n^{-1},y_n)$ is bounded above by some positive constant $C$.

Then for $n$ sufficiently large,
\begin{align*}
\tau(\gamma_n^{-1}v_n)-\tau(v_n) & = \abs{ t_n+\sigma(\gamma_n^{-1},y_n)}-\abs{t_n} \\
& \leq \abs{t_n+\sigma(\gamma_n^{-1},y_n)-C}+\abs{C}+\abs{t_n} \\
& = - t_n -\sigma(\gamma_n^{-1},y_n) +2 C +t_n =  -\sigma(\gamma_n^{-1},y_n)+2 C.
\end{align*}
Since $y_n\to y\neq p$, by the expanding property there exists $C' > 0$ such that 
$$
\sigma(\gamma_n^{-1},  y_n) \geq \norm{\gamma_n^{-1}}_\sigma-C'
$$
for all $n \geq 1$. So 
$$
\lim_{n \rightarrow \infty} \tau(\gamma_n^{-1}v_n)-\tau(v_n)=-\infty.
$$
Finally, fix $\alpha\in \Gamma$ such that $\alpha(x', p) \in {\rm int}(L)$.
Then $\alpha \gamma_n^{-1} v_n \in L \times \Rb$ for $n$ sufficiently large and 
\begin{align*}
\lim_{n \to \infty} & \tau(\alpha\gamma_n^{-1} v_n) - \tau(v_n)    \leq \max_{q \in \Lambda(\Gamma)} \abs{\sigma(\alpha,q)} +\lim_{n \to \infty} \tau(\gamma_n^{-1}  v_n) - \tau(v_n) = -\infty. 
\end{align*}
Hence $v_n \notin \mc D_L$ for $n$ sufficiently large.

  \medskip

 \noindent \emph{Case 2:} Assume $y=p$, $x \neq p$ and $t_n \to +\infty$.
Arguing as in Case 1, we have $x'=p\neq y'$ and $t_n+\sigma(\gamma_n^{-1},y_n)$ is bounded below by some $-C \leq 0$. Hence for $n$ sufficiently large,
\begin{align*}
\tau(\gamma_n^{-1}v_n)-\tau(v_n)  \leq t_n +\sigma(\gamma_n^{-1},y_n) +2 C -t_n =  \sigma(\gamma_n^{-1},y_n)+2 C.
\end{align*}
Since $\gamma_n^{-1}y_n\to y'\neq p$, by the expanding property there exists $C' > 0$ such that 
\[\sigma(\gamma_n^{-1},y_n)=-\sigma(\gamma_n,\gamma_n^{-1}y_n)\leq -\norm{\gamma_n}_\sigma+C'\]
for all $n \geq 1$.  So 
$$
\lim_{n \rightarrow \infty} \tau(\gamma_n^{-1}v_n)-\tau(v_n)=-\infty.
$$
Finally, fix $\alpha \in \Gamma$ such that $\alpha(p, y') \in {\rm int}(L)$. Then $\alpha\gamma_n^{-1} v_n \in L \times \Rb$ for $n$ sufficiently large and 
$$
\lim_{n \to \infty} \tau(\gamma' \gamma_n^{-1} v_n) - \tau(v_n)= -\infty. 
$$
Hence $v_n \notin \mc D_L$ for $n$ sufficiently large.


\section{Finiteness criteria for BMS measures}\label{sec:finite BMS}


We obtain a finiteness criterion for BMS measures in the geometrically finite case, which is the natural
analogue of the criterion of  Dal'bo--Otal--Peign\'e \cite{DOP}. 
This criterion was generalized to geometrically infinite groups in the context of negatively curved Riemannian manifolds by Pit--Schapira \cite[Th.\,1.4, 1.6, 1.8]{PS18}.

\begin{theorem}\label{BMS finite criterion}
Suppose $(\sigma, \bar{\sigma}, G)$ is a continuous GPS system for a geometrically finite convergence group $\Gamma\subset\mathrm{Homeo}(M)$ where $\delta=\delta_\sigma(\Gamma) < +\infty$.
If $\delta_\sigma(P)<\delta$ for any maximal parabolic subgroup $P$ of $\Gamma$, then $Q_\sigma(\delta) = +\infty$ and the BMS measure is finite.
\end{theorem}

Previously~\cite[Th.\,4.2]{BCZZ} we showed that if the $\sigma$-Poincar\'e series diverges for a maximal parabolic subgroup $\Gamma_p$ of $\Gamma$, then
$\delta_\sigma(\Gamma_p)<\delta_\sigma(\Gamma)$. 
So Theorem \ref{BMS finite criterion} implies the following criterion.

\begin{corollary} \label{cor:divergent at parabolics implies finite BMS}
Suppose $(\sigma,\bar\sigma,G)$ is a continuous GPS system for a geometrically finite convergence group $\Gamma\subset\mathsf{Homeo}(M)$ with 
$\delta:=\delta_\sigma(\Gamma)<+\infty$. If  
$$
\sum_{\gamma \in \Gamma_p} e^{-\delta_\sigma(\Gamma_p)\norm{\gamma}_\sigma}=+\infty
$$
whenever $p$ is a bounded parabolic point of $\Gamma$,  then $Q_\sigma(\delta)=+\infty$ and
the BMS measure is finite.
\end{corollary}

The rest of the section is devoted to the proof of Theorem~\ref{BMS finite criterion}. Fix a continuous GPS system $(\sigma,\bar\sigma,G)$  for a geometrically finite convergence group $\Gamma\subset\mathsf{Homeo}(M)$ with $\delta:=\delta_\sigma(\Gamma)<+\infty$.

The theorem will follow from the next two propositions and Theorem~\ref{thm:compact quotient} above. The first provides a condition for the ``cusps'' to have finite measure and the second verifies this condition when the hypotheses of Theorem~\ref{BMS finite criterion} are satisfied. 

\begin{proposition}\label{prop:geomfin implies BMS finite} Suppose $p$ is a bounded parabolic point of $\Gamma$, $\mu$ is a $\sigma$-Patterson--Sullivan measure of dimension $\delta$, and $\bar\mu$ is a $\bar\sigma$-Patterson--Sullivan measure of dimension $\delta$.
Let $m_\Gamma$ denote the quotient measure on $U_\Gamma$ associated to the measure $\tilde m = e^{\delta G(x,y)}d\bar\mu(x) \otimes d\mu(y)  \otimes dt$ on $\tilde U_\Gamma$. 

If $\mu(\{p\})=\bar \mu(\{p\})=0$ and 
$$\sum_{\gamma\in \Gamma_p} ||\gamma||_\sigma e^{-\delta||\gamma||_\sigma}<+\infty,$$
then for every $\epsilon > 0$ there is a horoball $H$ based at $p$ whose image $\wh{H}$ in $U_\Gamma = \Gamma \backslash \tilde U_\Gamma$ satisfies $m_\Gamma(\wh{H}) < \epsilon$. 
\end{proposition}

\begin{proposition}\label{no atoms}
If $\delta_\sigma(\Gamma_p)<\delta$ whenever $p$ is a bounded parabolic point of $\Gamma$, then there exist  a $\sigma$-Patterson--Sullivan-measure $\mu$ of dimension $\delta$  and a $\bar\sigma$-Patterson--Sullivan-measure $\bar\mu$ of dimension $\delta$ such that 
$$
\mu(\{p\})=\bar\mu(\{p\})=0
$$
whenever $p$ is a bounded parabolic point.
\end{proposition}

Assuming the propositions we prove Theorem~\ref{BMS finite criterion}. 

\begin{proof}[Proof of Theorem \ref{BMS finite criterion}]

Let $\Fc \subset \Lambda(\Gamma)$ denote the set of bounded parabolic points. Proposition \ref{no atoms} implies that there exist  Patterson--Sullivan measures $\mu$ and $\bar \mu$ of dimension $\delta$ with the property that
$\mu(\{p\})=\bar\mu(\{p\})=0$ for any $p\in\mc F$. Since $\Gamma$ is geometrically finite, there are countably many bounded parabolic points (see Lemma~\ref{lem:finitely many parabs}), and the rest are conical, so the
conical limit set has full measure for both $\mu$ and $\bar\mu$. 
Theorem~\ref{thm:dichtomy from 1st paper} then implies that
$Q_\sigma(\delta)=Q_{\bar\sigma}(\delta)=+\infty.$ Let $m_\Gamma$ denote the BMS measure constructed in Section~\ref{subsec:background on GPS systems}.

By Lemma~\ref{lem:finitely many parabs} there are finitely many $\Gamma$-orbits of bounded parabolic points. Let $\{p_1,\dots,p_k\}$ contain one representative of each orbit. Since $\delta_\sigma(\Gamma_{p_j})<\delta$, we see that
$$
\sum_{\gamma\in \Gamma_{p_j}} \norm{\gamma}_\sigma e^{-\delta\norm{\gamma}_\sigma}<+\infty.
$$
So by Proposition \ref{prop:geomfin implies BMS finite} for each $1 \leq j \leq k$ there is a horoball $H(p_j, N_j)$ based at $p_j$ whose image $\wh{H}_j$ in $U_\Gamma$ has finite $m_\Gamma$-measure. 

Let 
$$
\Hc := \bigcup_{\gamma \in \Gamma} \bigcup_{j=1}^k \gamma H(p_j, N_j). 
$$
By Theorem~\ref{thm:compact quotient} the quotient $\Gamma \backslash \left(\tilde U_\Gamma - \Hc\right)$ is compact. By definition $m_\Gamma$ is locally finite and so the $m_\Gamma$-measure of 
$\Gamma \backslash \left(\tilde U_\Gamma - \Hc\right)$ is finite. Hence 
$$
m_\Gamma \left( U_\Gamma\right) \leq m_\Gamma \left( \Gamma \backslash \left(\tilde U_\Gamma - \Hc\right) \right) + \sum_{j=1}^k m_\Gamma\left(\wh{H}_j\right)
$$
is finite. 
\end{proof} 

\subsection{Estimate of excursions in the horoballs}

To establish Proposition \ref{prop:geomfin implies BMS finite}, as well as the later  Proposition \ref{prop:bound on mR horoball}, we will need an estimate on lengths of excursions in the horoballs.

Let $p\in\Lambda(\Gamma)$ be bounded parabolic, let $K\subset\Lambda(\Gamma)\smallsetminus\{p\}$ be such that $\Gamma_p\cdot K=\Lambda(\Gamma)\smallsetminus\{p\}$ and let $N$ be a compact neighborhood of $K$ in $\overline{M^{(2)}\times\R}\smallsetminus\{p\}$.

Let $H(p,N)$ denote the horoball introduced in Section~\ref{sec:horoballs}. Then given $x,y\in \Lambda(\Gamma) \smallsetminus\{p\}$ distinct, let $I_{x,y} \subset \Rb$ be the subset satisfying 
$$
H(p,N)\cap\big( (x,y) \times \Rb \big)= (x,y) \times I_{x,y}.
$$

Since $\Gamma_p$ only accumulates on $p$, the expanding property implies that there exists a constant $D>0$ such that 
\begin{equation}\label{eqn:expanding property in cusp excursion} 
\norm{\gamma}_\sigma-D \leq \sigma(\gamma, z)\leq \norm{\gamma}_\sigma + D
\end{equation}
for all $\gamma\in \Gamma_p$ and $z\in K$. 

\begin{lemma}
\label{horoball length bound}
There exists $C>0$ such that: if $\gamma\in\Gamma_p$, $x\in K$, $y\in\gamma K$ and $I_{x,y}\ne\varnothing$,
then
\[ \abs{G(x,y)}\leq C \]
and
\[\left[\tfrac C2,\norm\gamma_\sigma-\tfrac C2\right]\subset I_{x,y} \subset \left[-\tfrac C2,\norm\gamma_\sigma+\tfrac C2\right].\]
In particular
$$ \norm{\gamma}_\sigma -C\le{\rm Leb}(I_{x,y})\le \norm{\gamma}_\sigma +C.$$
\end{lemma}

\begin{proof} Fix a distance $\dist$ which generates the topology on $\overline{M^{(2)} \times \Rb}$. Since $N$ is a neighborhood of the compact set $K$, by Proposition~\ref{prop:compactifying flowspace}\eqref{item:sequences in compactification of flow space} there exist $T,\epsilon>0$  such that:
\begin{itemize}
\item $K \times M \times (-\infty, -T) \subset N$, 
\item $M \times K \times (T,+\infty) \subset N$, 
\item  if $a\in K$ and $\dist(a,b)<\epsilon$, then $(a,b) \times \Rb \subset N$.
\end{itemize} 

By assumption $I_{x,y} \neq \varnothing$ and so there exists $t_0 \in \Rb$ such that $(x,y,t_0)\in H(p,N)$.
In particular, $(x,y,t_0)\notin N$ and so $\dist(x,y)\geq\epsilon$ by our choice of $\epsilon$. Thus 
$$
\abs{G(x,y)} \leq \max_{\dist(a,b)\geq \epsilon} \abs{G(a,b)},
$$
which is finite by the continuity of $G$.

Fix $t\in I_{x,y}$. Then $(x,y,t)\in H(p,N)$. So $(x,y,t)\notin N$ and so $t\geq -T$ since $x \in K$. Further $(x,y,t)\notin \gamma N$, i.e.\ 
$$\gamma^{-1}(x,y,t) = (\gamma^{-1}x,\gamma^{-1}y,t+\sigma(\gamma^{-1},y))\notin N.$$
Since $\gamma^{-1}y\in K$,  this implies that $t+\sigma(\gamma^{-1},y)\leq T$ or equivalently that
$$t\leq T+\sigma(\gamma,\gamma^{-1}y).$$
Then Equation~\eqref{eqn:expanding property in cusp excursion}  implies that $t \leq \norm{\gamma}_\sigma + T + D$. 

Since $t \in I_{x,y}$ was arbitrary, we have 
\[ I_{x,y} \subset [-T,\norm\gamma_\sigma+T+D].\]

For the other inclusion, we again apply Proposition~\ref{prop:compactifying flowspace}\eqref{item:sequences in compactification of flow space}, but this time using the fact that the complement of $N$ is a neighborhood of $p$ in $\overline{M^{(2)}\times\R}$. Then there exist $T',\epsilon' > 0$ such that $(a,b,t) \notin N$ whenever 
\begin{itemize}
\item $\dist(a,p)<\epsilon'$ and $t<-T'$, 
\item $\dist(b,p)<\epsilon'$ and $t>T'$, or  
\item $\dist(a,p)<\epsilon'$ and $\dist(b,p)<\epsilon'$.
\end{itemize}

Note that since $\Gamma_p$ acts properly discontinuously on $\Lambda(\Gamma)\smallsetminus\{p\}$, the set $F$ of $\gamma\in\Gamma_p$ such that 
$$
\max_{z \in K} \dist(\gamma z, p) \geq \epsilon'
$$ 
is finite. Let
$$D':=\max\left(\max_{\alpha\in F\cup F^{-1},\ q\in M}\abs{\sigma(\alpha,q)},\ \max_{\alpha,\beta\in F}\norm{\alpha^{-1}\beta}_\sigma,\ 1\right).$$

Now fix $t \in \Rb \smallsetminus I_{x,y}$. We will show that $t> C$ or $t< \norm\gamma_\sigma-C$ for some constant $C$.
Since $(x,y,t)\notin H(p,N)$ we have $(x,y,t)\in\alpha N$ for some $\alpha\in\Gamma_p$, or equivalently
\[(\alpha^{-1}x,\alpha^{-1}y,t+\sigma(\alpha^{-1},y))\in N.\]
In particular, at least one of $\alpha^{-1}x,\alpha^{-1}y$ must be a distance at least $\epsilon'$ from $p$.

 \medskip
 
 \noindent \emph{Case 1:} Assume  $\epsilon'\leq \dist(\alpha^{-1}x,p)$ and $\epsilon' \leq \dist(\alpha^{-1}y,p)$.
Since $x\in K$ and $y\in\gamma K$, we see that $\alpha^{-1}\in F$ and $\alpha^{-1}\gamma\in F$. 
Hence $\norm\gamma_\sigma=\norm{\alpha\alpha^{-1}\gamma}_\sigma\leq D'$.
In particular $\norm\gamma_\sigma-D'<D'$ so either  $t < D'$ or $t > \norm\gamma_\sigma-D'$. 

 \medskip
\noindent \emph{Case 2:}  Assume $ \dist(\alpha^{-1}y,p)<\epsilon'\leq\dist(\alpha^{-1}x,p)$.
Then $t+\sigma(\alpha^{-1},y)\leq T'$ since $\alpha^{-1}(x,y,t)\in N$.
Moreover, since $x\in K$, we have $\alpha^{-1}\in F$ and so $\sigma(\alpha^{-1},y)\geq - D'$.
Thus
\[t\leq T'+D'.\]

 \medskip
\noindent \emph{Case 3:}  Assume $ \dist(\alpha^{-1}x,p)<\epsilon'\leq\dist(\alpha^{-1}y,p)$.
Then $t+\sigma(\alpha^{-1},y)\geq -T'$ since $\alpha^{-1}(x,y,t)\in N$.
Moreover, since $y\in\gamma K$, we have $\alpha^{-1}\gamma\in F$ and so 
$$
\sigma(\gamma^{-1},y) =\sigma(\alpha^{-1},y)- \sigma(\alpha^{-1}\gamma,\gamma^{-1}y) \geq \sigma(\alpha^{-1},y)-D'.
$$
Thus
\[t\geq -\sigma(\gamma^{-1},y)-T'-D'= \sigma(\gamma,\gamma^{-1}y)-T'-D'.\]
Then Equation~\eqref{eqn:expanding property in cusp excursion}  implies that 
\[t\geq \norm{\gamma}_\sigma-D-T'-D'.\]

So in all cases,
$$
(t \leq T'+D') \quad \text{or} \quad (t \geq \norm{\gamma}_\sigma-D-T'-D').
$$
Since $t \in \Rb \smallsetminus I_{x,y}$ was arbitrary,  we then have 
\begin{equation*}
(T'+D', \norm{\gamma}_\sigma-D-T'-D') \subset I_{x,y}. \qedhere
\end{equation*}
\end{proof}

\subsection{Proof of Proposition \ref{prop:geomfin implies BMS finite}}\label{subsec:proof of BMS meas finite} 
Our proof closely follows a classical argument (see \cite[Th.\,B]{DOP} in the case of negatively curved Riemannian manifolds).
Fix a bounded parabolic point $p \in \Lambda(\Gamma)$ and Patterson--Sullivan measures $\mu$, $\bar \mu$ as in the statement of the proposition. 
We use the notation from the previous section:
let $K\subset\Lambda(\Gamma)\smallsetminus\{p\}$ be such that $\Gamma_p\cdot K=\Lambda(\Gamma)\smallsetminus\{p\}$ and $N$ be a compact neighborhood of $K$ in $\overline{M^{(2)}\times\R}\smallsetminus\{p\}$.

Let $H(p,N)$ denote the horoball introduced in Section~\ref{sec:horoballs}. Then given $x,y\in \Lambda(\Gamma) \smallsetminus\{p\}$ distinct, let $I_{x,y} \subset \Rb$ be the subset satisfying 
$$
H(p,N)\cap\big( (x,y) \times \Rb \big)= (x,y) \times I_{x,y}.
$$

Since $\Gamma_p$ only accumulates on $p$, the expanding property implies that there exists a constant $D>0$ such that 
\begin{equation}\label{eqn:Gammap acting on K}
\norm{\gamma}_\sigma-D \leq \sigma(\gamma, z)\leq \norm{\gamma}_\sigma + D
\end{equation}
for all $\gamma\in \Gamma_p$ and $z\in K$.

As in the statement of Proposition \ref{prop:geomfin implies BMS finite}, let 
$$
d\tilde m  := e^{\delta G(x,y)} d\bar\mu(x) \otimes d\mu(y) \otimes dt
$$
and let $m_\Gamma$ be the quotient measure on $U_\Gamma$. Also let $\pi \colon \tilde U_\Gamma \to U_\Gamma$ be the projection map.
It follows from Equation~\eqref{eqn:defining function of quotient measures} that 
$$
m_\Gamma(\pi(A)) \leq \tilde m (A)
$$
for all measurable subsets $A \subset \tilde U_\Gamma$, since $P(1_{A}) \geq 1_{\pi(A)}$.
(Observe that $\pi(A)$ is measurable since $\pi^{-1}(\pi(A))=\Gamma\cdot A$ is a countable union of measurable sets.)

For $\gamma\in\Gamma_p$, let
$$H_\gamma:=\bigcup_{x\in K,y\in\gamma K}(x,y) \times I_{x,y}=H(p,N)\cap K\times\gamma K\times\R.$$
Let $C > 0$ satisfy Lemma~\ref{horoball length bound}. Then by Equation~\eqref{eqn:Gammap acting on K} we have 
\begin{align*}
\tilde m (H_\gamma) & = \int_{\substack{x\in K\\y\in\gamma K}} {\rm Leb}(I_{x,y})e^{\delta G(x,y)}d\bar\mu(x)d\mu(y) \le \left(\norm{\gamma}_\sigma+C\right) e^{\delta C}\bar\mu(K)\mu(\gamma K)\\
& = \left(\norm{\gamma}_\sigma+C\right) e^{\delta C}\bar\mu(K)\int_K e^{-\delta\sigma(\gamma,y)} d\mu(y) \\
& \le   (\norm{\gamma}_\sigma+C) e^{\delta C}\bar\mu(K)e^{\delta D}e^{-\delta \norm{\gamma}_\sigma}\mu(K).
\end{align*}

Let $\wh{H}$ denote the image of $H(p,N)$ in $U_\Gamma$. 
Notice that $\wh{H}$ is contained in the image of 
$$
\Big(\{p\} \times (\Lambda(\Gamma) \smallsetminus\{p\}) \times \Rb\Big) \cup \Big((\Lambda(\Gamma) \smallsetminus\{p\}) \times \{p\} \times \Rb \Big) \cup \bigcup_{\gamma \in \Gamma_p } H_\gamma .
$$
Since $\mu$ and $\bar\mu$ do not have atoms at $p$, the measure of the set of geodesic segments in $H_F$ with one endpoint at $p$ has $\tilde m$-measure zero. 
Therefore,
\begin{align*}
m_\Gamma(\wh{H}) & \le \sum_{\gamma\in \Gamma_p } \tilde m(H_\gamma)  \le e^{\delta(D+C)} \bar\mu(K)\mu(K) \sum_{\gamma\in \Gamma_p}
\left(\norm{\gamma}_\sigma+C\right)e^{-\delta\norm{\gamma}_\sigma} .
\end{align*}
By Proposition~\ref{prop:basic properties}\eqref{item:properness} and the assumption that 
$$\sum_{\gamma\in \Gamma_p} \norm{\gamma}_\sigma e^{-\delta\norm{\gamma}_\sigma}<+\infty,$$
the sum 
$$ \sum_{\gamma\in \Gamma_p} \left(\norm{\gamma}_\sigma+C\right)
e^{-\delta\norm{\gamma}_\sigma}
$$
is finite. Hence $m_\Gamma(\wh{H})< \epsilon$.

\subsection{Proof of Proposition~\ref{no atoms}} By symmetry, it suffices to construct a $\sigma$-Patterson--Sullivan measure $\mu$ of dimension $\delta$ such that 
$$
\mu(\{p\})=0
$$
whenever $p$ is a bounded parabolic point. 
(See \cite[Th.\,IV.2]{dalbo-peigne} for a similar argument in the case of negatively curved Riemannian manifolds).

We will recall the construction of Patterson--Sullivan measures given in~\cite[Th.\,4.1]{BCZZ} (which is a modification of the standard construction).  Endow $\Gamma \sqcup M$ with the topology described in Proposition~\ref{prop:compactifying}.

The first step in the construction is to choose a well behaved magnitude function. Using~\cite[Lem.\,4.2]{BCZZ} we can assume that our magnitude function satisfies 
$$
\limsup_{\gamma \rightarrow x} \abs{\sigma(\alpha, x)- (\norm{\alpha \gamma}_\sigma -\norm{\gamma}_\sigma)} = 0
$$
(in our case  the $\kappa$ in~\cite[Lem.\,4.2]{BCZZ} is zero). Then using~\cite[Lem.\,3.1]{Patterson}, there exists a non-decreasing function $\chi:\mb R \to\mb R_{\geq1}$ such that 
\begin{enumerate}[label=(\alph*)]
 \item \label{item:PS meas chi subexp bis} for every $\epsilon>0$ there exists $R>0$ such that $\chi(r+t)\leq e^{\epsilon t}\chi(r)$ for any $r\geq R$ and $t\geq 0$,
 \item \label{item PS meas chi div bis} $\sum_{\gamma \in \Gamma} \chi(\norm{\gamma}_\sigma) e^{-\delta \norm{\gamma}_\sigma} =+\infty$.
\end{enumerate}
(When $\sum_{\gamma \in \Gamma} e^{-\delta \norm{\gamma}_\sigma} =+\infty$, we can take $\chi \equiv 1$.) Finally, for $s >\delta$, consider the probability measure
$$
\mu_s := \frac{1}{Q(s)} \sum_{\gamma \in \Gamma} \chi(\norm{\gamma}_\sigma) e^{-s \norm{\gamma}_\sigma} \mathcal D_\gamma,
$$
where $Q(s):=\sum_{\gamma \in \Gamma} \chi(\norm{\gamma}_\sigma) e^{-s \norm{\gamma}_\sigma} $
and $D_\gamma$ denotes the Dirac measure centered on $\gamma\in\Gamma\sqcup M$.

Fix $s_m \searrow \delta$ such that $\mu_{s_m} \to \mu$ in the weak-$*$ topology. Then one can show that $\mu$ is a $\sigma$-Patterson--Sullivan measure of dimension $\delta$ which is supported on $\Lambda(\Gamma)$, see the proof of~\cite[Th.\,4.1]{BCZZ} for details.  

Now we fix a bounded parabolic point $p \in M$ and show that $\mu(\{p\})$ equals zero. We will need the following cocompactness result. 

\begin{lemma}\label{lem:parabolic cocompact} $\Gamma_p$ acts cocompactly on $(\Gamma \sqcup \Lambda(\Gamma)) \smallsetminus \{p\}$. 
\end{lemma} 

\begin{proof} Fix a compact set $K_0 \subset \Lambda(\Gamma)\smallsetminus\{p\}$ such that $\Gamma_p \cdot K_0 = \Lambda(\Gamma)\smallsetminus\{p\}$. 
Then fix an open set $U \supset K_0$ in $\Gamma \sqcup \Lambda(\Gamma)$ such that $p \notin \overline{U}$. Fix an enumeration $\Gamma = \{ \gamma_k\}$ so that $\gamma_1=\id$ and let 
$F_n := \{ \gamma_1, \dots, \gamma_n\}$. Then $K_n := \overline{U} \cup F_n$ is compact in $(\Gamma \sqcup \Lambda(\Gamma)) \smallsetminus\{p\}$.
 We claim that for $n$ sufficiently large,  $\Gamma_p \cdot K_n = (\Gamma \sqcup \Lambda(\Gamma))\smallsetminus\{p\}$. 

Suppose not. Since 
$$
\Lambda(\Gamma)\smallsetminus\{p\} = \Gamma_p \cdot K_0 \subset \Gamma_p \cdot K_n,
$$
then for each $n$ we can find $\gamma_{k_n} \in \Gamma -  \Gamma_p \cdot K_n$.  
Since $\gamma_{k_n} \notin \Gamma_p$, we must have $\gamma_{k_n}(p) \neq p$. Then for each $n \geq 1$, there exists $\alpha_n \in \Gamma_p$ such that $\alpha_n \gamma_{k_n}(p) \in K_0$. 
Since $\gamma_{k_n} \notin \Gamma_p \cdot F_n$, we see that $\{\alpha_n \gamma_{k_n}\}$ is an escaping sequence in $\Gamma$. 
So passing to a subsequence we can suppose that there exists $b^+, b^- \in M$ such that $\alpha_n \gamma_{k_n}(x) \to b^+$ for all $x \in M \smallsetminus \{ b^-\}$. 
We can further suppose that $\alpha_n \gamma_{k_n}(p) \to a \in K_0$. If $a \neq b^+$, then $p = b^-$ and $p$ is a conical limit point, which is impossible by Lemma~\ref{lem: parab conical disjoint}.
So we must have $b^+ = a \in K_0$. So for $n$ large, $\alpha_n \gamma_{k_n} \in U$, which implies that $\gamma_{k_n} \in \Gamma_p \cdot K_n$. So we have a contradiction. 
\end{proof}

By Lemma~\ref{lem:parabolic cocompact}, there exists a compact subset $K$ of $(\Gamma \sqcup \Lambda(\Gamma)) \smallsetminus\{p\}$ such that 
$\Gamma_p \cdot K=(\Gamma \sqcup \Lambda(\Gamma)) \smallsetminus \{p\}.$
Then let $\Gamma_0 := K \cap \Gamma$ and notice that $\Gamma_p(\Gamma_0)=\Gamma$. 

Next enumerate $\Gamma_p=\{\alpha_n\}$. Then consider the decreasing sequence of open neighborhoods of $p$  in $\Gamma \sqcup \Lambda(\Gamma)$ given by
$$V_n:=(\Gamma\sqcup \Lambda(\Gamma))\smallsetminus\left(\alpha_1(K)\cup\dots\cup \alpha_n(K)\right).$$
We will show that $\mu(V_n)\to 0$ as $n\to\infty$.

Since $\alpha_k^{-1}\to p$ as $k\to \infty$, Proposition~\ref{prop:basic properties}\eqref{item:multiplicative estimate} implies that there exists $D>0$ such that 
\begin{equation}\label{eqn:h is close to additive}
\norm{\alpha_k}_\sigma+ \norm{\gamma}_\sigma-D\leq \norm{\alpha_k\gamma}_\sigma \leq \norm{\alpha_k}_\sigma+ \norm{\gamma}_\sigma+D
\end{equation}
for all $k\geq 1$ and $\gamma\in \Gamma_0$. 

Set $\epsilon:=(\delta_\sigma(\Gamma)-\delta_\sigma(P))/2$.

\begin{lemma} There exists $C > 0$ such that 
$$\chi(\norm{\alpha_k\gamma}_\sigma)\leq C  e^{\epsilon \norm{\alpha_k}_\sigma} \chi(\norm{\gamma}_\sigma)$$
for all  $k\geq 1$ and $\gamma\in \Gamma_0$. 
\end{lemma} 

\begin{proof} By property \ref{item:PS meas chi subexp bis} of the function $\chi$ there exists $R>0$ such that $$\chi(r+t)\leq e^{\epsilon t}\chi(r)$$ for any $r\geq R$ and $t\geq 0$. Proposition~\ref{prop:basic properties}\eqref{item:properness} implies that the set 
$$
\Gamma_0' : = \{ \gamma \in \Gamma_0 : \norm{\gamma}_\sigma < R\}
$$
is finite and that the number 
$$
c_0 : = \max\left\{ 0,-\inf_{\gamma \in \Gamma} \norm{\gamma}_\sigma\right\}\ge 0
$$
is finite. 

If $\gamma \not\in \Gamma_0'$, then Equation~\eqref{eqn:h is close to additive} and the monotonicity of $\chi$ imply that 
\begin{align*}
\chi(\norm{\alpha_k\gamma}_\sigma) & \leq \chi(\norm{\alpha_k}_\sigma+\norm{\gamma}_\sigma+D) \leq \chi(\norm{\alpha_k}_\sigma+\norm{\gamma}_\sigma+D+c_0) \\
& \leq  e^{\epsilon (D+c_0)}e^{\epsilon \norm{\alpha_k}_\sigma} \chi(\norm{\gamma}_\sigma).
\end{align*}
If $\gamma \in \Gamma_0'$, then Equation~\eqref{eqn:h is close to additive} and the monotonicity of $\chi$ imply that
\begin{align*}
\chi(\norm{\alpha_k\gamma}_\sigma) & \leq \chi(\norm{\alpha_k}_\sigma+\norm{\gamma}_\sigma+D) \leq \chi(\norm{\alpha_k}_\sigma+\norm{\gamma}_\sigma+D+2c_0+R) \\
& \leq  e^{\epsilon (D+c_0)}e^{\epsilon \norm{\alpha_k}_\sigma} \chi(\norm{\gamma}_\sigma+R+c_0) \\
& \leq e^{\epsilon (D+c_0)}e^{\epsilon \norm{\alpha_k}_\sigma} \max_{\beta \in \Gamma_0'} \frac{\chi(\norm{\beta}_\sigma+R+c_0)}{\chi(\norm{\beta}_\sigma)} \chi(\norm{\gamma}_\sigma).  
\end{align*}
So 
$$
C: = e^{\epsilon(D+c_0)}\max_{\beta \in \Gamma_0'} \frac{\chi(\norm{\beta}_\sigma+R+c_0)}{\chi(\norm{\beta}_\sigma)} 
$$
satisfies the estimate for all $\gamma \in \Gamma_0$. 
\end{proof}

Then for $n$ sufficiently large we obtain for any $s>\delta$
\begin{align*}
\mu_s(V_n) & \leq \frac{1}{Q(s)} \sum_{k>n}\sum_{\gamma\in\Gamma_0} \chi(\norm{\alpha_k\gamma}_\sigma) e^{-s \norm{\alpha_k\gamma}_\sigma}\\
& \le  \frac{Ce^{sD}}{Q(s)}\sum_{k>n}\sum_{\gamma\in\Gamma_0}  e^{\epsilon \norm{\alpha_k}_\sigma}\chi(\norm{\gamma}_\sigma) e^{-s(\norm{\alpha_k}_\sigma+\norm{\gamma}_\sigma)}\\
  &\le  \frac{Ce^{sD}}{Q(s)}\sum_{k>n} e^{-(s-\epsilon)\norm{\alpha_k}_\sigma}\sum_{\gamma\in\Gamma_0} \chi(\norm{\gamma}_\sigma)e^{-s\norm{\gamma}_\sigma}\\
    &\le  C e^{sD}\sum_{k>n} e^{-(\delta-\epsilon)\norm{\alpha_k}_\sigma}.
\end{align*}
Since $V_n$ is open, we can take the limit $s_m \searrow \delta$ to get
\begin{align*}
\mu(\{p\})& \le  \lim_{n \to \infty}\liminf_{m \to \infty} \mu_{s_m}(V_n)  \le \lim_{n \to \infty} Ce^{\delta D}\sum_{k>n} e^{-(\delta-\epsilon)\norm{\alpha_k}_\sigma}=0, 
\end{align*}
where in the last equality we use the fact that $\sum_{k=1}^\infty e^{-(\delta-\epsilon) \norm{\alpha_k}_\sigma} < +\infty$.


\section{Equidistribution for geometrically finite actions}\label{sec:equi for geom finite}


In this section we establish the equidistribution result Theorem~\ref{GF equidistribution in intro} and the counting result Theorem~\ref{thm:counting for geomfin actions} stated in the introduction, which we restate below. 
For the rest of the section suppose that:
\begin{itemize}
\item $(\sigma,\bar\sigma,G)$ is a GPS system for a geometrically finite convergence group $\Gamma \subset \mathsf{Homeo}(M)$ with $\delta : = \delta_\sigma(\Gamma) < +\infty$;
\item $\delta_\sigma(\Gamma_p)<\delta$ whenever $p$ is a bounded parabolic point of $\Gamma$, so by Theorem~\ref{BMS finite criterion} the Poincar\'e series of $\Gamma$ diverges at its critical exponent and the BMS measure $m_\Gamma$ is finite; and
\item $\mathcal L(\sigma,\bar\sigma,G)$ is non-arithmetic, so the flow $\psi^t \colon (U_\Gamma,m_\Gamma) \to (U_\Gamma,m_\Gamma)$ is mixing.
\end{itemize}

Recall also the definition of the measure $\tilde m_R$ from Section~\ref{sec:equi}.
For every $R > 0$, let $m_R$ be the quotient measure on $U_\Gamma = \Gamma \backslash \tilde U_\Gamma$ associated to the measure 
$$\tilde m_R:=\sum_{\substack{\gamma \in \Gamma_{\lox}\\ \ell_\sigma(\gamma)\leq R}}\mc D_{\gamma^-}\otimes \mc D_{\gamma^+}\otimes dt$$
on $\tilde U_\Gamma$. Notice that, by Corollary~\ref{length spectrum discrete}, $m_R$ is a finite measure. 

This section's equidistribution result is an extension to bounded functions of Theorem~\ref{thm:equidistribution}, which only applied to compactly supported functions:

\begin{theorem}\label{GF equidistribution}
For any bounded continuous function $f$ on $\Gamma \backslash (M^{(2)}\times\mb R)$ we have
$$\lim_{R\to \infty}\delta e^{-\delta R}\int f \,dm_R =\int f \,\frac{dm_\Gamma}{\norm{m_\Gamma}}.$$
\end{theorem}

As a corollary, we will obtain Theorem~\ref{thm:counting for geomfin actions} from the introduction.

\begin{corollary}\label{cor:counting for geomfin actions}
$$
\#\left\{[\gamma]\in [\Gamma_{\lox}]^w : 0<\ell_\sigma(\gamma) \leq R\right\} \sim \frac{e^{\delta R}}{\delta R},
$$
i.e.\ the ratio of the two sides goes to 1 as $R \to +\infty$.
\end{corollary}

Here $[\Gamma_{\lox}]^w$ refers to the set of weak conjugacy classes of loxodromic elements, which we introduce now.
The method we will use to obtain counting result naturally provides counting for closed $\psi^t$-orbits.
In general (when one allows torsion), these orbits do not correspond to conjugacy classes, instead they correspond weak conjugacy classes, which can be slightly bigger.

We say two loxodromic elements $\gamma_1,\gamma_2\in\Gamma$ are \emph{weakly conjugate} if they have the same period and there exists $g\in \Gamma$ such that $\gamma_1$ and $g\gamma_2g^{-1}$ have the same attracting fixed point and the same repelling fixed point.
We denote by $[\Gamma_{\lox}]^w$ the set of weak conjugacy classes. Notice that if $\Gamma$ is torsion-free, then the set of weak conjugacy classes is simply
the set of conjugacy classes.
Moreover, the number of conjugacy classes in a given weak conjugacy class $[\gamma]^w$ is bounded above by the size of the stabilizer of the axis of $\gamma$ (see \cite[Eq.\,(7.1)]{BZ}).

To get a counting result from an equidistribution result as Theorem~\ref{GF equidistribution} above, one usually apply it to the constant function equal to one.
Unfortunately, here this does not yield exactly Corollary~\ref{cor:counting for geomfin actions}.
Let us describe what the quotient measure $m_R$ computes,
i.e.\ let us compute explicitly its total mass.
 Given $\gamma \in \Gamma_{\lox}$, let 
$$
S_\gamma : = \{ \alpha \in \Gamma : \alpha(\gamma^\pm) = \gamma^\pm \}.
$$
Recall that $\Gamma$ acts properly discontinuously on $M^{(2)} \times \Rb$ and notice that if $\alpha \in S_\gamma$, then 
$$
\alpha \cdot (\gamma^-, \gamma^+, t) = (\gamma^-, \gamma^+, t + \sigma(\alpha, \gamma^+)).
$$
Hence
$$\ell_\sigma^{\rm prim}(\gamma):=\min\{\abs{\sigma(\alpha,\gamma^+)} : \alpha\in\Gamma_{\lox},\ \alpha^\pm=\gamma^\pm\}$$
exists and is finite. Further, by Corollary~\ref{length spectrum discrete},  $\ell_\sigma^{\rm prim}(\gamma)$ is positive. 

\begin{observation}\label{obs:computing mass} Given $\gamma \in \Gamma_{\lox}$, let $[\gamma]\in[\Gamma_{\lox}]^w$ denote the weak conjugacy class of $\gamma$ and let $\nu_{[\gamma]}$ be the quotient measure on $U_\Gamma = \Gamma \backslash \tilde U_\Gamma$ associated to the measure
$$
\tilde\nu_{[\gamma]}:= \sum_{\hat\gamma\in[\gamma]}\mc D_{\hat\gamma^-}\otimes \mc D_{\hat\gamma^+}\otimes dt.
$$
Then 
$$
\norm{\nu_{[\gamma]}}= \ell_\sigma^{\rm prim}(\gamma).
$$
In particular, 
$$
\norm{m_R}=\sum_{\substack{[\gamma] \in [\Gamma_{\lox}]^w\\ \ell_\sigma(\gamma)\leq R}} \ell_\sigma^{\rm prim}(\gamma).
$$
\end{observation}

\begin{proof} See Section~\ref{sec:proof of obs:computing mass}  for the proof. \end{proof}

\subsection{Proof of Theorem~\ref{GF equidistribution}} 

The main idea in the proof is to obtain uniform estimates for the measure $m_R$ on ``cusps'' based at a bounded parabolic point. 

\begin{proposition}\label{prop:bound on mR horoball}
If $p\in M$ is a bounded parabolic point and $\epsilon > 0$, then there exists a horoball $H$ based at $p$ whose image $\wh{H}$ in $U_\Gamma$ satisfies 
 $$ m_R(\wh H) \leq e^{\delta R} \epsilon
 $$
 for all $R > 0$. 
\end{proposition}

Delaying the proof of the proposition, we prove Theorem~\ref{GF equidistribution}. 

\begin{proof}[Proof of Theorem~\ref{GF equidistribution}]  Fix $\epsilon > 0$ and fix representatives $p_1,\dots,p_k$ of the $\Gamma$-orbits of bounded parabolic points (recall there are finitely many such orbits by Lemma~\ref{lem:finitely many parabs}). Using Proposition \ref{prop:geomfin implies BMS finite} and Proposition~\ref{prop:bound on mR horoball}, for each $1 \leq j \leq k$ there exists a horoball $H_j$ based at $p_j$ whose image $\wh{H}_j$ in $U_\Gamma$ satisfies  
$$
m_\Gamma(\wh{H}_j) < \frac{\epsilon\norm{m_\Gamma}}{k} \quad \text{and} \quad m_R(\wh{H}_j) \leq \frac{e^{\delta R}}{\delta} \frac{\epsilon}{k}. 
$$
Let 
 $$
 \wh{C}:=\wh H_1\cup\dots\cup \wh H_k.
 $$
 Recall $U_\Gamma\smallsetminus \wh C$ is compact by Theorem~\ref{thm:compact quotient},
so we can find a compactly supported continuous function $\chi \colon U_\Gamma\to[0,1]$ which is equal to $1$ on $U_\Gamma\smallsetminus \wh C$.
By Theorem~\ref{thm:equidistribution},  
$$\lim_{R \to \infty} \left|\delta e^{-\delta R}\int \chi f\,dm_R-\int \chi f \,\frac{dm_\Gamma}{\norm{m_\Gamma}} \right|=0.
$$
So,
 \begin{align*}
  \limsup_{R \to \infty} &  \left|\delta e^{-\delta R}\int f \,dm_R-\int f \,\frac{dm_\Gamma}{\norm{m_\Gamma}}\right|\\
  &  \leq \limsup_{R \to \infty}   \left|\int (1-\chi)f\,\frac{dm_\Gamma}{\norm{m_\Gamma}}\right| + \left|\delta e^{-\delta R}\int (1-\chi)f\,dm_R\right|\\
  &\leq \norm{f}_\infty \frac{m_\Gamma(\wh C)}{\norm{m_\Gamma}} + \norm{f}_\infty \delta e^{-\delta R}m_R(\wh C)\leq 2\norm{f}_\infty\epsilon.
 \end{align*}
 Since $\epsilon > 0$ was arbitrary, 
  \begin{equation*}
  \lim_{R \to \infty} \delta e^{-\delta R}\int f \,dm_R= \int f \,\frac{dm_\Gamma}{\norm{m_\Gamma}}. \qedhere
  \end{equation*}
  
  \end{proof}
 
\subsection{Proof of Proposition~\ref{prop:bound on mR horoball}} The proof is motivated by the argument in \cite[Th.\,5.2]{roblin}.

Fix a compact set $K\subset \Lambda(\Gamma)\smallsetminus\{p\}$ such that $\Gamma_p(K)=\Lambda(\Gamma)\smallsetminus\{p\}$.
By Propostion~\ref{prop:disjoint horoballs}, we can find a compact neighborhood $N$ of $K$ in $\overline{M^{(2)}\times\R}\smallsetminus\{p\}$ such that $H(p,N)$ is disjoint from $\gamma H(p,N)$ whenever $\gamma p\neq p$.

 Let 
$$
F_0 : = \{ \gamma \in \Gamma_p : \gamma K \cap K \neq \varnothing\} \cup \{ \gamma \in \Gamma_p : \norm{\gamma}_\sigma \leq 1\}.
$$
Notice that $F_0$ is finite by Proposition~\ref{prop:basic properties}\eqref{item:properness} and the fact that $\Gamma_p$ only accumulates on $p$. 

Given a finite subset $F \supset F_0$, let $N_F$ be a compact neighborhood of $K$ in $\overline{M^{(2)}\times\R}\smallsetminus\{p\}$ containing $N$ and all $(a,b,t)$ with $a,b\in F\cdot K$. Then set
$$
H_F : = H(p, N_F).
$$
Notice that $H(p,N_F) \subset H(p,N)$, hence $H(p,N_F)$ is disjoint from $\gamma H(p,N_F)$ whenever $\gamma p\neq p$.

Given $x,y\in \Lambda(\Gamma) \smallsetminus\{p\}$ distinct, let $I_{x,y} \subset \Rb$ be the subset satisfying 
$$
H_{F_0} \cap\big( (x,y) \times \Rb \big)= (x,y) \times I_{x,y}.
$$
Given $\alpha\in\Gamma_p$, let
$$ A_{\alpha} := \{ \gamma\in\Gamma_{\lox} : \gamma^-\in K,\ \gamma^+\in \alpha K\}. $$

Let $\wh{H}_F$ be the image of $H_F$ in $U_\Gamma$. Then $\wh{H}_F$ is contained in the image of 
$$
\Big(\{p\} \times (\Lambda(\Gamma) \smallsetminus\{p\}) \times \Rb\Big) \cup \Big((\Lambda(\Gamma) \smallsetminus\{p\}) \times \{p\} \times \Rb \Big) \cup \bigcup_{\alpha \in \Gamma_p \smallsetminus F } K \times \alpha K \times \Rb.
$$
Then, since $p$ is not the fixed point of a loxodromic element, 
\begin{align}\label{eqn:second estimate on cusp of level F} 
 m_R\left( \wh H_{F}\right)  \leq \sum_{\alpha \in \Gamma_p \smallsetminus F} \ \sum_{\substack{\gamma \in A_\alpha \\ \ell_\sigma(\gamma) \leq R }} \mr{Leb}\left(I_{\gamma^-,\gamma^+}\right).
 \end{align}
So we need to control $\mr{Leb}\left(I_{\gamma^-,\gamma^+}\right)$ and the number of loxodromics in $A_\alpha$ with length at most $R$.

By construction 
$$
K \times \bigcup_{\alpha \in \Gamma_p \smallsetminus F_0} \alpha K
$$
is relatively compact in $M^{(2)}$. So by Proposition~\ref{prop:basic properties}\eqref{item:loxodromics with separated fixed points} there exists $C_1 > 0$ such that:  if $\alpha \in \Gamma_p \smallsetminus F_0$ and $\gamma \in A_\alpha$, then 
$$
\ell_\sigma(\gamma) \geq \norm{\gamma}_\sigma - C_1. 
$$
Also, by  Lemma \ref{horoball length bound} there exists $C_2 > 0$ such that: if $\alpha \in \Gamma_p$,  $\gamma \in A_\alpha$ and $I_{\gamma^-,\gamma^+} \ne\varnothing$, then
\begin{equation*}
\norm{\alpha}_\sigma - C_2 \leq \mr{Leb}\left(I_{\gamma^-,\gamma^+}\right) \leq \norm{\alpha}_\sigma  +C_2
\end{equation*}
and 
\begin{equation}\label{eqn:length of I gamma^- gamma^+}
\left[\frac{C_2}{2}, \norm{\gamma}_\sigma - \frac{C_2}{2} \right] \subset I_{\gamma^-,\gamma^+}.
\end{equation}

Then Equation~\eqref{eqn:second estimate on cusp of level F} becomes 
\begin{align}
 m_R\left( \wh H_{F}\right) & \leq \sum_{\alpha \in \Gamma_p \smallsetminus F} (\norm{\alpha}_\sigma  +C_2)\ \# \{ \gamma \in A_\alpha : \norm{\gamma}_\sigma \leq R + C_1\} \nonumber \\
& \leq (C_2+1)\sum_{\alpha \in \Gamma_p \smallsetminus F} \norm{\alpha}_\sigma \# \{ \gamma \in A_\alpha : \norm{\gamma}_\sigma \leq R + C_1\}  \label{eqn:third estimate on cusp of level F}.
  \end{align}
  
  To bound the size of $\# \{ \gamma \in A_\alpha : \norm{\gamma}_\sigma \leq R + C_1\}$, we first establish the following relationship between the magnitudes of $\alpha$ and $\gamma$. 

\begin{lemma}\label{lem:additive estimate for Aalpha} There exists $C_3 > 0$ such that: if $\alpha \in \Gamma_p \smallsetminus F_0$ and $\gamma \in A_\alpha$, then 
$$
\norm{\gamma}_\sigma \geq \norm{\alpha}_\sigma +\norm{\alpha^{-1}\gamma}_\sigma- C_3. 
$$
\end{lemma} 
\begin{proof}
By Proposition~\ref{prop:basic properties}\eqref{item:multiplicative estimate}, it suffices to show that there exists $\epsilon > 0$ such that:   if $\alpha \in \Gamma_p \smallsetminus F_0$ and $\gamma \in A_\alpha$, then $\dist(\alpha,\alpha^{-1}\gamma) > \epsilon.$

Suppose not. Then there exist sequences $\{\alpha_n\},\{\gamma_n\}$ such that $\alpha_n \in \Gamma_p \smallsetminus F_0$, $\gamma_n \in A_{\alpha_n}$ and $\dist(\alpha_n,\alpha_n^{-1}\gamma_n)\to 0$.
Then $\{\alpha_n\}$ must be escaping, and hence converges to $p$ in $\Gamma\sqcup M$, and hence $\alpha_n^{-1}\gamma_n\to p$.

Recall that $\gamma H_{F_0}$ is disjoint from $H_{F_0}$, which means $I_{\gamma_n^-,\gamma_n^+}$ is disjoint from its translate by $\ell_\sigma(\gamma_n)$.
Since $I_{\gamma_n^-,\gamma_n^+}$ contains an interval of length $\norm{\alpha_n}_\sigma-C_2$, see Equation~\eqref{eqn:length of I gamma^- gamma^+}, this means that
\[
\ell_\sigma(\gamma_n)=\sigma(\gamma_n,\gamma_n^+)> \norm{\alpha_n}_\sigma-C_2.
\]

Since $\Gamma_p$ only accumulates on $p$, the expanding property implies that there exists a constant $D>0$ such that 
\begin{equation*}
\norm{\alpha}_\sigma-D \leq \sigma(\alpha, z)\leq \norm{\alpha}_\sigma + D
\end{equation*}
for all $\alpha\in \Gamma_p$ and $z\in K$. 

Then by the cocycle property
\begin{align*}
\sigma(\alpha_n^{-1}\gamma_n,\gamma_n^+)& =\sigma(\alpha_n^{-1},\gamma_n^+)+\sigma(\gamma_n,\gamma_n^+)=-\sigma(\alpha_n,\alpha_n^{-1}\gamma_n^+)+\sigma(\gamma_n,\gamma_n^+) \\
& > -\norm{\alpha_n}_\sigma-D+\norm{\alpha_n}_\sigma-C_2 = -D-C_2.
\end{align*}
Then by Proposition~\ref{prop:basic properties}\eqref{item:a technical fact}  the distance between $\alpha_n^{-1}\gamma_n$ and $\alpha_n^{-1}\gamma_n\gamma_n^+=\alpha_n^{-1}\gamma_n^+$ tends to zero. So $\alpha_n^{-1}\gamma_n^+\to p$. But this contradicts the assumption $\alpha_n^{-1}\gamma_n^+\in K$.
\end{proof}

\begin{lemma}\label{lem:size of A_alpha} There exists $C_4 > 0$ such that: if $\alpha \in \Gamma_p \smallsetminus F_0$, then 
$$
\# \{ \gamma \in A_\alpha : \norm{\gamma}_\sigma \leq R + C_1\} \leq C_4 e^{\delta(R - \norm{\alpha}_\sigma)}
$$
for all $R > 0$. 
\end{lemma} 

\begin{proof} By~\cite[Prop.\,6.3]{BCZZ} there exists $c > 0$ such that: 
$$
\#\{ \gamma \in \Gamma : \norm{\gamma}_\sigma \leq R \} \leq c e^{\delta R}
$$
for all $R \geq 0$. Then by Lemma~\ref{lem:additive estimate for Aalpha}, 
\begin{align*}
\# \{ \gamma \in A_\alpha : \norm{\gamma}_\sigma \leq R + C_1\} & \leq \# \{ \gamma \in A_\alpha : \norm{\alpha^{-1}\gamma}_\sigma \leq R + C_1+C_3 -\norm{\alpha}_\sigma\} \\
& \leq  \# \{ \gamma \in \Gamma : \norm{\gamma}_\sigma \leq R + C_1+C_3 -\norm{\alpha}_\sigma\}\\
&  \leq c e^{\delta(C_1+C_3)} e^{\delta(R - \norm{\alpha}_\sigma)}. \qedhere
\end{align*}
\end{proof} 

By Lemma~\ref{lem:size of A_alpha}  and Equation~\eqref{eqn:third estimate on cusp of level F}, there exists $C > 0$ such that 
\begin{align*}
 m_R\left( \wh H_{F}\right) &\leq  Ce^{\delta R} \sum_{\alpha \in \Gamma_p \smallsetminus F} \norm{\alpha}_\sigma e^{-\delta \norm{\alpha}_\sigma} 
\end{align*}
for all $R > 0$ and $F \supset F_0$. Since $ \sum_{\alpha \in \Gamma_p} \norm{\alpha}_\sigma e^{-\delta \norm{\alpha}_\sigma}$ converges, we can pick $F$ sufficiently large so that $ m_R\left( \wh H_{F}\right) \leq e^{\delta R}\epsilon$ for all $R > 0$.

\subsection{Proof of Observation~\ref{obs:computing mass}}\label{sec:proof of obs:computing mass} By construction 
$$
\norm{m_R}=\sum_{\substack{[\gamma] \in [\Gamma_{\lox}]^w\\ \ell_\sigma(\gamma)\leq R}} \norm{\nu_{[\gamma]}}
$$
and so it suffices to fix $\gamma \in \Gamma_{\lox}$ and show that $\norm{\nu_{[\gamma]}}= \ell_\sigma^{\rm prim}(\gamma)$.

Let 
$$
I : = \{(\gamma^-, \gamma^+)\} \times [0, \ell_\sigma^{\rm prim}(\gamma)). 
$$
Then, by definition, $\norm{\nu_{[\gamma]}} = \nu_{[\gamma]}(\pi(I))$. We will use Equation~\eqref{eqn:defining function of quotient measures} to compute this measure. 

Notice that when $t \in (0,\ell_\sigma^{\rm prim}(\gamma))$ we have 
\begin{align*}
P(1_I)([\gamma^-, \gamma^+, t]) & = \sum_{\alpha \in \Gamma} 1_I(\alpha(\gamma^-, \gamma^+, t))=\sum_{\alpha \in \Gamma} 1_I( \alpha \gamma^-, \alpha \gamma^+, t + \sigma(\alpha, \gamma^+)) \\
& =\# \{ \alpha \in \Gamma : \alpha \gamma^\pm = \gamma^\pm \text{ and } 0 \leq t + \sigma(\alpha, \gamma^+) < \ell_\sigma^{\rm prim}(\gamma) \} \\
& = \#\{ \alpha \in S_\gamma : 0 \leq t + \sigma(\alpha, \gamma^+) < \ell_\sigma^{\rm prim}(\gamma)\}.
\end{align*}
By definition
$$
\{ \abs{\sigma(\alpha, \gamma^+)} : \alpha \in S_\gamma \} \subset \{0\} \cup [\ell_\sigma^{\rm prim}(\gamma),+\infty).
$$
So 
\begin{align*}
P(1_I)([\gamma^-, \gamma^+, t]) = \# S_\gamma^0
\end{align*}
where $S_\gamma^0:= \{ \alpha \in S_\gamma : \sigma(\alpha, \gamma^+) = 0\}$. 

Thus
$$
P(1_I)= (\# S_\gamma^0) 1_{\pi(I)} .
$$
Then by Equation~\eqref{eqn:defining function of quotient measures},
\begin{align*}
\nu_{[\gamma]}(\pi(I)) & = \frac{1}{\# S_\gamma^0 } \int_{U_\Gamma} P(1_I)d\nu_{[\gamma]}  =  \frac{1}{\# S_\gamma^0 } \int_{\tilde U_\Gamma} 1_I d\tilde \nu_{[\gamma]} \\
& =  \frac{\ell_\sigma^{\rm prim}(\gamma)}{\# S_\gamma^0 } \#\{ \alpha \in [\gamma] : \alpha^{\pm} = \gamma^{\pm}\} \\
& =  \frac{\ell_\sigma^{\rm prim}(\gamma)}{\# S_\gamma^0 } \#\{ \alpha \in S_\gamma : \sigma(\alpha, \gamma^+) = \sigma(\gamma, \gamma^+)\}.
\end{align*}
Since 
$$
\alpha \in S_\gamma \mapsto \sigma(\alpha, \gamma^+) \in (\Rb, +)
$$
is a homomorphism, we see that $\#\{ \alpha \in S_\gamma : \sigma(\alpha, \gamma^+) = \sigma(\gamma, \gamma^+)\} = \# S_\gamma^0$ and the proof is complete.

\subsection{Proof of Corollary~\ref{cor:counting for geomfin actions}}
In our proof of our counting corollary, it will  be convenient to use a slightly modified equidistribution statement which concerns a measure $m_R'$ which is
closely related to $m_R$.

For every $R > 0$, let $m_R'$ be the quotient measure associated to the measure
$$\tilde m_R'=\sum_{\substack{\gamma \in \Gamma_{\lox}\\ \ell_\sigma(\gamma)\leq R}}\frac{1}{\ell_\sigma^{\rm prim}(\gamma)} \mc D_{\gamma^-}\otimes \mc D_{\gamma^+}\otimes dt$$
on $\tilde U_\Gamma$.  Then by Observation~\ref{obs:computing mass},
$$
\norm{m_R'}= \#\left\{[\gamma]\in [\Gamma_{\lox}]^w : \ell_\sigma(\gamma) \leq R\right\}.
$$

One expects that $m_R$ is close to $Rm'_R$ since we expect most geodesics of length at most $R$ have length close to $R$. The next result
makes this intuition precise.

\begin{proposition}\label{prop:GF equidistribution'}
If $f \colon \Gamma \backslash (M^{(2)} \times \Rb) \to \Rb$ is a bounded continuous function, then
$$\lim_{R \to \infty} \delta R e^{- \delta R} \int f dm'_R = \int f \frac{dm_\Gamma}{\norm{m_\Gamma}}.$$
\end{proposition}

\begin{proof}
For any real number $R > 0$, let 
$$\mc{G}_R:=\{ [\gamma] \in  [\Gamma_{\lox}]^w :  \ell_\sigma(\gamma) \leq R\}$$
and, as before, for $c \in [\Gamma_{\lox}]^w$ 
let $\nu_c$ be the quotient measure associated to $\sum_{\gamma\in c} \mc D_{\gamma^-}\otimes \mc D_{\gamma^+}\otimes dt$. 
Note that 
\begin{align*}
R\,m'_R - m_R  & = \sum_{c \in \mc{G}_R} \left( \frac{R}{\ell_\sigma^{\rm prim}(c)} - 1\right) \nu_c\geq 0.
\end{align*}
In particular,
\[ \liminf_{R\to\infty}\delta Re^{-\delta R}\int fdm_R' \geq \int f\frac{m_\Gamma}{\norm{m_\Gamma}}.\]

Let 
$\mc{G}_R^p:=\{ c \in  \mc{G}_R :  \ell_\sigma(c)=\ell_\sigma^{\rm prim}(c)\}$ and $\mc{G}_R^{np}:= \mc{G}_R\setminus \mc{G}_R^p$.
Then, it is clear that
 $$\#\left( \mc{G}_R^{np}\right) \le\#\left( \mc{G}_{\frac{R}{2}}\right)+\#\left( \mc{G}_{\frac{R}{3}}\right)+\cdots+\#\left( \mc{G}_{\frac{R}{\lfloor R\rfloor}}\right).$$
Let 
$
s:=\mr{sys}(\Gamma,\sigma) = \min\{ \ell_\sigma(\gamma) : \gamma \in \Gamma_{\lox}\} 
$
 be the systole of $(\Gamma,\sigma)$, which is positive by Corollary~\ref{length spectrum discrete}. Notice that $\#\left( \mc{G}_S\right) \le \frac{1}{s}\norm{m_S}$ for all $S > 0$. So Theorem \ref{GF equidistribution}, applied to $f\equiv 1$  implies that there exists $D>0$ so
that
$$\#\left( \mc{G}_S\right) \le De^{\delta S}$$
for all $S>0$. Therefore, $\#\left( \mc{G}_R^{np}\right)\le DRe^{\frac{\delta R}{2}}$, so 
$$\lim_{R\to\infty} e^{-\delta R}R \#\left( \mc{G}_R^{np}\right)=0.$$

For any $r > 0$ and $R \geq 2e^r$, we have
\begin{align*}
\int f & \,dm_R  \geq \sum_{c \in \mc{G}_R^p-\mc{G}_{e^{-r}R}} \frac{e^{-r}R}{\ell_\sigma(c)} \int f \,d\nu_c  \\
 & \ge  e^{-r} \int f R \,dm'_R - e^{-r} R \sum_{c \in \mc{G}_{e^{-r}R}} \frac{1}{\ell_\sigma^{\rm prim}(c)} \int f \,d\nu_c -
\sum_{c \in \mc{G}_R^{np}} \frac{e^{-r}R}{\ell_\sigma^{\rm prim}(c)} \int f \,d\nu_c \\
  & \ge e^{-r} \int f R \,dm'_R - e^{-r} R \sum_{c \in \mc{G}_{e^{-r}R}} \frac{1}{s} \int f \,d\nu_c
  - \frac{e^{-r}R}{\ell_\sigma^{\rm prim}(c)}  \#\left(\mc{G}_R^{np}\right)\norm{f}_\infty \cdot \ell_\sigma^{\rm prim}(c) \\
 & =  e^{-r} \int f R \,dm'_R - \frac{R}{s} \int f \,dm_{e^{-r} R} - e^{-r}R  \#\left(\mc{G}_R^{np}\right)\norm{f}_\infty.	
\end{align*}
So
$$
 e^{-\delta R}  \int f R \,dm'_R  \leq e^r e^{-\delta R}  \int f \,dm_R + e^{-\delta R} \frac{R}{s}\int f \,dm_{e^{-r} R}+
e^{-\delta R} R  \#\left(\mc{G}_R^{np}\right)\norm{f}_\infty.$$
We previously observed that $\lim_{R\to\infty} e^{-\delta R}R \#\left( \mc{G}_R^{np}\right)=0.$
Theorem \ref{GF equidistribution} implies that $e^{-\delta e^{-r}R}\int f \,dm_{e^-r R}$ remains bounded as $R\to\infty$,
so $e^{-\delta R} \frac{R}{s}\int f \,dm_{e^{-r} R}\to0$ as $R\to\infty$.
Hence, again by Theorem \ref{GF equidistribution}, we have
$$
\limsup_{R\to\infty}\delta R e^{-\delta R} \int f \,dm'_R \leq e^r\int f\frac{dm_\Gamma}{\norm{m_\Gamma}}
$$
for any $r > 0$. We may then conclude by taking $r \to 0$.
\end{proof}

By integrating the bounded constant 1 function against both sides of the equation in Proposition~\ref{prop:GF equidistribution'}, we obtain the desired counting statement (Corollary~\ref{cor:counting for geomfin actions}).\qed

\section{Transverse and relatively Anosov subgroups of semisimple Lie groups}\label{sec:semisimple Lie groups}

In this section we will apply our results to the particular case of relatively Anosov subgroups of semisimple Lie groups. To state these results in full generality requires a number of definitions, for a more detailed discussion using the same notation see~\cite{CZZ3}.  

Let $\mathsf G$ be a connected semisimple Lie group without compact factors and with finite center. Fix a Cartan decomposition $\mathfrak{g} = \mathfrak{k} + \mathfrak{p}$ of the Lie algebra of $\mathsf G$, a Cartan subspace $\mathfrak a \subset \mathfrak p$ and a Weyl chamber $\mathfrak a^+ \subset \mathfrak a$. Let $\mathsf{K} \subset \mathsf G$ denote the maximal compact subgroup with Lie algebra $\mathfrak k$. Then let $\kappa \colon \mathsf G \to \mathfrak a^+$ denote the \emph{Cartan projection}, that is $\kappa(g) \in \mathfrak a^+$ is the unique element where 
$$
g = k_1 e^{\kappa(g)} k_2
$$
for some $k_1, k_2 \in \mathsf K$. Let $\iota \colon \mathfrak a \to \mathfrak a$ be the opposite involution, which has the defining property that 
\begin{equation}\label{eqn:involution on kappa}
\iota(\kappa(g)) = \kappa(g^{-1})
\end{equation}
for all $g \in \mathsf G$. 

Let $\Delta \subset \mathfrak{a}^*$ denote the system of simple restricted roots corresponding to the choice of $\mathfrak a^+$. 
Given a subset $\theta \subset \Delta$, we let $\mathsf P_\theta \subset \mathsf G$ denote the associated parabolic subgroup and let $\mathcal F_\theta:=\mathsf G/\mathsf P_\theta$ denote the associated flag manifold. 

A discrete subgroup  $\Gamma \subset \mathsf{G}$
is {\em $\mathsf P_\theta$-divergent} if whenever $\{\gamma_n\}$ is a sequence of distinct elements of $\Gamma$ and $\alpha\in\theta$, then 
$$
\lim_{n \to \infty} \alpha(\kappa(\gamma_n))=+\infty.
$$

Next we describe the limit set of a $\mathsf{P}_\theta$-divergent group. For every $g \in \mathsf{G}$, fix a Cartan decomposition 
$$
g=m_g e^{\kappa(g)} \ell_g
$$
Then following the notation in~\cite{GGKW}, define a map 
\[U_\theta\colon \mathsf{G} \to \Fc_\theta\] 
by letting $U_\theta(g) := m_g\mathsf{P}_\theta$.
One can show that if $\alpha(\kappa(g)) > 0$ for all $\alpha \in \theta$, then $U_\theta(g)$ is independent of the choice of Cartan decomposition, see \cite[Chap.\,IX, Th.\,1.1]{Helgason}.
In particular,  if $\Gamma$ is $\mathsf{P}_\theta$-divergent, then $U_\theta(\gamma)$ is uniquely defined for all but finitely many $\gamma\in\Gamma$.
The limit set of a $\mathsf{P}_\theta$-divergent group $\Gamma \subset \mathsf{G}$ is given by 
$$
\Lambda_\theta(\Gamma) : = \{ F \in \Fc_\theta : \exists \{ \gamma_n\} \subset \Gamma \text{ distinct such that } U_\theta(\gamma_n) \to F\}. 
$$
One motivation for this definition comes from the dynamical behavior described in Proposition~\ref{prop:characterizing convergence in general symmetric case} below.

\textbf{For the rest of the section, we assume that  $\theta$ is symmetric} (i.e.\ $\iota^*(\theta) = \theta$ where $\iota\colon \mathfrak a \to \mathfrak a$ is the involution associated to $\mathfrak a^+$). 
In this case there is a natural notion of transversality for pairs in $\mathcal F_\theta$, and a $\mathsf{P}_\theta$-divergent subgroup $\Gamma$ is called {\em $\mathsf{P}_\theta$-transverse} if any two flags in $\Lambda_\theta(\Gamma)$ are transverse.  We say $\Gamma$ is \emph{non-elementary} if $\#\Lambda_\theta(\Gamma)\geq 3$. Every non-elementary $\mathsf{P}_\theta$-transverse group acts on
its limit set as a convergence group action, see \cite[\S5.1]{KLP1} or Observation \ref{obs:transverse implies convergent} below. For a more in-depth discussion of transverse groups
and their dynamical properties see \cite{CZZ3,DKO,KOW,KOW23}.

Associated to a subset $\theta \subset \Delta$ is a natural subspace of $\mathfrak a$ defined by 
$$\mathfrak{a}_\theta:=\{a\in\mathfrak{a} : \beta(a)=0\text{ for all } \beta\in\Delta\smallsetminus\theta\}.$$
One can show that $\mathfrak{a}_\theta^*$ is generated by $\{\omega_\alpha|_{\mathfrak{a}_\theta} :\alpha\in \theta\}$ where $\omega_\alpha$ is the fundamental weight associated to $\alpha$. 
Hence we can identify $\mathfrak{a}_\theta^*$ as a subspace of $\mathfrak a^*$. 
Then given $\phi\in \mathfrak a_{\theta}^*$  and a $\mathsf{P}_\theta$-divergent subgroup $\Gamma$, we define a Poincar\'e series
$$Q^\phi_\Gamma(s)=\sum_{\gamma\in\Gamma} e^{-s\phi(\kappa(\gamma))},$$
which has  a critical exponent $\delta^\phi(\Gamma) :=\inf\{ s > 0 : Q^\phi_\Gamma(s) < +\infty\}\in[0,\infty]$. 
Moreover, there exists a well-defined $\theta$-Cartan projection
$$\kappa_\theta:\mathsf{G}\to \mathfrak{a}_\theta$$
so that $\kappa_\theta=p_\theta\circ\kappa$ where $p_\theta:\mathfrak{a}\to\mathfrak{a}_\theta$ is the unique projection map so that $\omega_\alpha|_{\mathfrak{a}_\theta}=\omega_\alpha\circ p_\theta$ for all $\alpha\in\theta$. Hence 
$$
\phi \circ \kappa = \phi\circ \kappa_\theta
$$
for all $\phi \in \mathfrak{a}_\theta$.

The action of $\mathsf G$ on $\mathcal F_\theta$ preserves a smooth vector-valued cocycle 
$$B_\theta \colon \mathsf{G}\times \mathcal F_\theta\to \mathfrak{a}_\theta,$$
so that $B_\theta(g,U_\theta(g))=\kappa_\theta(g)$ for all $g$,
see Quint \cite{quint-ps} or Benoist--Quint \cite[\S 6.7.5]{BQ}. 
Let $\mathcal F_\theta^{(2)} \subset \mathcal F_\theta \times \mathcal F_\theta$ denote the subset of transverse flags. There is a smooth map $G_\theta \colon \mathcal F_\theta^{(2)} \to \mathfrak a_\theta$ which satisfies 
\begin{equation}\label{eqn:GPS property for Iwasawa}
G_\theta( g F, g F') - G_\theta(F,F') = \iota \circ B_\theta(g, F) + B_\theta(g,F'), 
\end{equation}
for any $g\in G$, see~\cite[Lem.\ 4.12]{sambarino15} or ~\cite[pg.\,11]{KOW23}.

Given a $\mathsf{P}_\theta$-transverse group $\Gamma$ and $\phi \in \mathfrak{a}_\theta^*$, a probability measure $\mu$ on $\Lambda_\theta(\Gamma)$ is called a \emph{$\phi$-Patterson--Sullivan} measure of dimension $\beta$ if for every $\gamma \in \Gamma$ the measures 
$\gamma_*\mu$, $\mu$ are absolutely continuous and 
 $$
 \frac{d\gamma_*\mu}{d\mu} = e^{-\beta \phi(B_\theta(\gamma^{-1}, \cdot))} \quad\mu\text{-almost everywhere}.
 $$
 We emphasize that our measures are assumed to be supported on the limit set.
 
 We will show that existence of a Patterson--Sullivan measure implies that the associated critical exponent is finite. 
 
 \begin{proposition}\label{prop:transverse groups when a PS measure exists} Suppose $\Gamma$ is non-elementary $\mathsf{P}_\theta$-transverse, $\phi\in \mathfrak{a}_\theta^*$  and $\mu$ is a $\phi$-Patterson--Sullivan measure of dimension $\beta$. Then 
 $$
\delta^\phi(\Gamma) \leq \beta.
 $$
 In particular, since $\delta^\phi(\Gamma) < +\infty$,  
 $$
\lim_{n \rightarrow \infty} \phi(\kappa(\gamma_n)) = +\infty
$$
whenever $\{\gamma_n\}$ is a sequence of distinct elements of $\Gamma$. 
\end{proposition} 

\begin{remark} The inequality $\delta^\phi(\Gamma) \leq \beta$ in Proposition~\ref{prop:transverse groups when a PS measure exists} was previous established in~\cite[Prop.\ 8.1]{CZZ3} in the special case when $\delta^\phi(\Gamma) < +\infty$. 
\end{remark}

The next result shows that the theory of Patterson--Sullivan measures for transverse groups is a particular example of the theory of Patterson--Sullivan measures for GPS systems. 

\begin{proposition}\label{prop:transverse groups give GPS systems}
Suppose $\Gamma$ is non-elementary $\mathsf{P}_\theta$-transverse, $\phi\in \mathfrak{a}_\theta^*$  and 
$$
\lim_{n \rightarrow \infty} \phi(\kappa(\gamma_n)) = +\infty
$$
whenever $\{\gamma_n\}$ is a sequence of distinct elements of $\Gamma$. Define cocycles $\sigma_\phi, \bar{\sigma}_\phi \colon \Gamma \times \Lambda_\theta(\Gamma) \to \Rb$ by 
$$
\sigma_\phi(\gamma, F) = \phi(B_\theta(\gamma,F)) \quad \text{and} \quad \bar\sigma_\phi(\gamma, F) = \iota^*(\phi)(B_\theta(\gamma,F)).
$$
Then $(\sigma_\phi,\bar\sigma_{\phi}, \phi\circ G)$ is a continuous  GPS system for the action of $\Gamma$ on $\Lambda_\theta(\Gamma)$. Moreover:
\begin{enumerate}
\item  One can choose magnitude functions 
$$
\norm{\gamma}_{\sigma_\phi} = \phi(\kappa(\gamma)) \quad \text{and} \quad \norm{\gamma}_{\bar\sigma_\phi} =  \iota^*(\phi)(\kappa(\gamma)).
$$
 In particular, $\delta^\phi(\Gamma) = \delta_{\sigma_\phi}(\Gamma)$. 
\item  If $\lambda \colon \mathsf G \to \mathfrak a^+$ is the Jordan projection, then 
$$
\phi(\lambda(\gamma)) = \ell_{\sigma_\phi}(\gamma)
$$
for all $\gamma \in \Gamma$. 
\end{enumerate}
\end{proposition} 

We will also show that the length spectrum is always non-arithmetic. 

\begin{proposition}\label{prop:transverse groups non-arithmetic}
Suppose $\Gamma$ is non-elementary $\mathsf{P}_\theta$-transverse,  $\phi\in \mathfrak{a}_\theta^*$ and 
$$
\lim_{n \rightarrow \infty} \phi(\kappa(\gamma_n)) = +\infty
$$
whenever $\{\gamma_n\}$ is a sequence of distinct elements of $\Gamma$. Then 
$$
\{ \phi(\lambda(\gamma)) + \iota^*(\phi)(\lambda(\gamma)) : \gamma \in \Gamma\}
$$
generates a dense subgroup of $\Rb$. 
\end{proposition}

In light of Proposition \ref{prop:transverse groups give GPS systems}, results obtained in the framework of GPS systems can also be used to obtain other results about Patterson--Sullivan measures for ${\ms P}_\theta$-transverse groups. For instance, the following is a straightforward application of \cite[Cor.\ 14.3]{BCZZ}.

\begin{proposition} \label{prop:PS measures and length spectra}
Suppose $\Gamma$ is a non-elementary ${\ms P}_\theta$-transverse group and $\phi_1, \phi_2 \in \mathfrak{a}_\theta^*$. For $i=1,2$, let  $\mu_i$ be a $\phi_i$-Patterson--Sullivan measure of dimension $\delta_i$. If 
$$
\sum_{\gamma\in\Gamma} e^{-\delta_1 \phi_1(\kappa(\gamma))} = +\infty,
$$
then 
$\mu_1$ and $\mu_2$ are absolutely continuous if and only if $\delta_1\phi_1(\lambda(\gamma))  = \delta_2\phi_2(\lambda(\gamma)) $ for all $\gamma \in \Gamma_{\lox}$, and mutually singular otherwise.
\end{proposition}

\begin{remark} \
\begin{enumerate}
\item When $\Gamma$ is Zariski-dense, a result of Benoist~\cite{benoistasymp1} says that $\delta_1\phi_1(\lambda(\gamma))  = \delta_2\phi_2(\lambda(\gamma)) $ for all $\gamma \in \Gamma_{\lox}$ if and only if $\delta_1\phi_1 = \delta_2\phi_2$. 
\item Without the assumption that $\mu_1,\mu_2$ are supported on $\Lambda_\theta(\Gamma)$, but with additional assumptions on $\Gamma$, Proposition~\ref{prop:PS measures and length spectra} was previously established by Kim~\cite{Kim24}. 
\end{enumerate} 
\end{remark}

Next we define relatively Anosov subgroups. There are several equivalent definitions, see the discussion in~\cite[\S4]{zhu-zimmer1}, and the one we use comes from~\cite{KL}. 

If $\Gamma\subset \mathsf{G}$ is relatively hyperbolic (as an abstract group) with respect to a finite collection $\mathcal P$ of finitely generated subgroups of $\Gamma$,
then $\Gamma$ is {\em $\mathsf{P}_\theta$-relatively Anosov} if it is $\mathsf{P}_\theta$-transverse and there exists a $\Gamma$-equivariant homeomorphism from the
Bowditch boundary $\partial(\Gamma,P)$ to the limit set $\Lambda_\theta(\Gamma)$.

\begin{corollary} \label{cor:counting for rel Anosovs}
Suppose $\Gamma\subset \mathsf{G}$ is  $\mathsf{P}_\theta$-relatively Anosov with respect to $\mathcal P$. If 
$\phi\in\mathfrak{a}_\theta^*$ and $\delta:=\delta^\phi(\Gamma)<+\infty$, then
$$
\#\{[\gamma]^w\in [\Gamma_{\lox}]^w :0<\phi(\lambda(\gamma))  \leq R\} \sim \frac{e^{\delta R}}{\delta R}.
$$
\end{corollary}

\begin{proof} In~\cite[Cor.\,7.2]{CZZ4} it was shown that $\delta^\phi(H)<\delta^\phi(\Gamma)$ for any maximal parabolic subgroup of $\Gamma$. Combining this with Proposition~\ref{prop:transverse groups non-arithmetic} and Proposition~\ref{prop:transverse groups give GPS systems} shows that the hypotheses of Theorem~\ref{thm:counting for geomfin actions} are satisfied. 
\end{proof} 

\subsection{Properties of transverse groups} In this subsection we state some important properties of transverse groups. 

We first explain why a transverse groups acts as a  convergence group on its limit set. This observation appears in Kapovich--Leeb--Porti~\cite[\S5.1]{KLP1}, but since the proof is short and Kapovich--Leeb--Porti uses different terminology we include it here. 

For $F \in \Fc_\theta$, let 
$$
\mathcal{Z}_F : = \{ F' \in \Fc_\theta: F \text{ is not transverse to } F'\}.
$$
We use the following description of the action of $\mathsf{G}$ on $\Fc_\theta$ (see for instance~\cite[Prop.\ 2.6]{CZZ3}).

\begin{proposition}\label{prop:characterizing convergence in general symmetric case} Suppose $\theta \subset \Delta$ is symmetric, $F^\pm \in \Fc_\theta$ and $\{g_n\}$ is a sequence in $\mathsf{G}$. The following are equivalent:  
\begin{enumerate}
\item $U_\theta(g_n) \to F^+$, $U_\theta(g_n^{-1}) \to F^-$ and $\lim_{n \to \infty} \alpha(\kappa(g_n)) = +\infty$ for every $\alpha \in \theta$,
\item $g_n(F) \to F^+$ for all $F \in \Fc_\theta \smallsetminus \mathcal{Z}_{F^-}$, and this convergence is uniform on compact subsets of $\Fc_\theta \smallsetminus \mathcal{Z}_{F^-}$. 
\item $g_n^{-1}(F) \to F^-$ for all $F \in \Fc_\theta \smallsetminus \mathcal{Z}_{F^+}$, and this convergence is uniform on compact subsets of $\Fc_\theta \smallsetminus \mathcal{Z}_{F^+}$. 
\end{enumerate}
\end{proposition}

This immediately implies that a transverse group is a convergence group. 

\begin{observation}\label{obs:transverse implies convergent} {\rm (see also \cite[\S5.1]{KLP1})}
 If $\Gamma \subset \mathsf{G}$ is $\mathsf{P}_\theta$-transverse, then $\Gamma$ acts on $\Lambda_\theta(\Gamma)$ as a convergence group. \end{observation} 

\begin{proof} Suppose $\{\gamma_n\} \subset \Gamma$ is a sequence of distinct elements. Since $\Gamma$ is $\mathsf{P}_\theta$-divergent, $\lim_{n \to \infty} \alpha(\kappa(g_n)) = +\infty$ 
for every $\alpha \in \theta$. Since $\Fc_\theta$ is compact, we can pass to a subsequence so that $U_\theta(g_n) \to F^+$ and $U_\theta(g_n^{-1}) \to F^-$. 
Then Proposition~\ref{prop:characterizing convergence in general symmetric case} implies that $g_n(F) \to F^+$ for all $F \in \Fc_\theta \smallsetminus \mathcal{Z}_{F^-}$.
Since $\Lambda_\theta(\Gamma)$ is transverse, this implies that $g_n(F) \to F^+$ for all $F \in \Lambda_\theta(\Gamma) \smallsetminus \{ F^-\}$.
\end{proof} 

We also use the following estimate. Let $\norm{\cdot}$ denote some fixed norm on $\mathfrak{a}$.

\begin{lemma}[{Quint~\cite[Lem.\,6.5]{quint-ps}}]
\label{quintlemma}
For any $\epsilon>0$ and distance $\dist_{\Fc_\theta}$ on $\Fc_{\theta}$ induced by a Riemannian metric, there exists $C=C(\epsilon, \dist_{\Fc_\theta})>0$ such that: if $g \in \mathsf{G}$, $F \in \Fc_\theta$ and $\dist_{\Fc_\theta}\left( F, \mathcal{Z}_{U_\theta(g^{-1})}\right) > \epsilon$, then 
$$
\norm{ B_\theta(g,F)- \kappa_\theta(g)}<C.
$$
\end{lemma}

Quint's lemma implies the following estimate for transverse groups. 

\begin{lemma}\label{lem:sigma phi is expanding} Suppose $\Gamma$ is non-elementary $\mathsf{P}_\theta$-transverse and $\phi\in \mathfrak{a}_\theta^*$. Let $\dist$ be a compatible distance on $\Gamma \sqcup \Lambda_\theta(\Gamma)$. For any $\epsilon> 0$ there exists $C > 0$ such that: if $\gamma \in \Gamma$, $F \in \Lambda_\theta(\Gamma)$ and $\dist(F,\gamma^{-1}) > \epsilon$, then 
$$
\phi(\kappa_\theta(\gamma)) -C \leq \phi(B_\theta(\gamma, F)) \leq \phi(\kappa_\theta(\gamma)) +C.
$$
\end{lemma} 

\begin{proof} Suppose not. Then for every $n \geq 1$ there exist $\gamma_n \in \Gamma$ and $F_n \in \Lambda_\theta(\Gamma)$  such that $\dist(F_n,\gamma_n^{-1}) > \epsilon$ and 
$$
\abs{\sigma_\phi(\gamma_n, F_n) - \phi(\kappa_\theta(\gamma_n))} \geq n.
$$
Passing to a subsequence we can suppose that $F_n \to F$ and  $U_\theta(\gamma_n^{-1}) \to F^- \in \Lambda_\theta(\Gamma)$ in the topology on $\Fc_\theta$. Then Proposition~\ref{prop:characterizing convergence in general symmetric case} implies that $\gamma_n^{-1} \to F^-$ in the topology on $\Gamma \sqcup \Lambda_\theta(\Gamma)$ defined in Proposition~\ref{prop:compactifying}. Hence $F \neq F^-$. Since $\Gamma$ is $\mathsf{P}_\theta$-transverse, we then have $F \notin \mathcal{Z}_{F^-}$. So by  Lemma~\ref{quintlemma} there exists $C > 0$ such that 
$$
\norm{ B_\theta(\gamma_n,F_n)- \kappa_\theta(\gamma_n)}\leq C
$$
for all $n \geq 1$. So 
$$
\abs{\sigma_\phi(\gamma_n, F_n) - \phi(\kappa_\theta(\gamma_n))} \leq C \norm{\phi}
$$
for all $n \geq 1$ and we have a contradiction. 
\end{proof} 

\subsection{Proof of Proposition~\ref{prop:transverse groups when a PS measure exists}} Suppose $\Gamma$ is non-elementary $\mathsf{P}_\theta$-transverse, $\phi\in \mathfrak{a}_\theta^*$  and $\mu$ is a $\phi$-Patterson--Sullivan measure of dimension $\beta$. 

\begin{lemma} If $\{\gamma_n\}$ is a sequence of distinct elements of $\Gamma$, then 
$$
\lim_{n \rightarrow \infty} \phi(\kappa(\gamma_n)) = +\infty.
$$
\end{lemma} 

\begin{proof} Suppose for a contradiction that $\{\gamma_n\}$ is a sequence of distinct elements with 
$$
\liminf_{n \rightarrow \infty} \phi(\kappa(\gamma_n)) < +\infty.
$$
Passing to a subsequence we can suppose that 
\begin{equation}\label{eqn:sup in new proof is bounded}
\sup_{n \geq 1}  \phi(\kappa(\gamma_n)) < +\infty.
\end{equation}
Since $\Gamma$ acts on $\Lambda_\theta(\Gamma)$ as a convergence group, we can pass to a further subsequence and assume that there exist $F^+, F^- \in \Lambda_\theta(\Gamma)$ such that 
\begin{equation}\label{eqn:convergence in new proof is bounded}
\gamma_n F \rightarrow F^+
\end{equation}
for all $F \in \Lambda_\theta(\Gamma) \smallsetminus \{ F^-\}$, and the convergence is uniform on compact subsets of $ \Lambda_\theta(\Gamma) \smallsetminus \{ F^-\}$. 

Fix two disjoint compact subsets $K_1, K_2 \subset \Lambda_\theta(\Gamma) \smallsetminus \{ F^-\}$ each with positive $\mu$-measure. Lemma~\ref{lem:sigma phi is expanding} and Equation~\eqref{eqn:sup in new proof is bounded} imply that there exists $C > 0$ such that 
$$
\phi(B_\theta(\gamma_n, F)) \leq C
$$
for all $n \geq 1$ and $F \in K_1 \cup K_2$. Then
$$
\mu(\gamma_n K_i) = (\gamma_n^{-1})_* \mu(K_i) = \int_{K_i} e^{-\beta \phi(B_\theta(\gamma_n, F))}d\mu(F) \geq e^{-\beta C} \mu(K_i)
$$
for all $i=1,2$ and $n \geq 1$. 

Since $K_1$ and $K_2$ are disjoint, for every $n \geq 1$ at least one of $\gamma_n K_1$, $\gamma_n K_2$ does not contain $F^+$. So after passing to a subsequence and possibly relabelling $K_1$, $K_2$ we can assume that $F^+ \notin \gamma_n K_1$ for all $n \geq 1$. Then by~\eqref{eqn:convergence in new proof is bounded}, we can find a pairwise disjoint subsequence $\{ \gamma_{n_j} K_1\}$. But then 
$$
1 = \mu(\Lambda_\theta(\Gamma)) \geq \sum_{j=1}^\infty \mu(\gamma_{n_j} K_1) \geq \sum_{j=1}^\infty  e^{-\beta C} \mu(K_1)= +\infty
$$
and we have a contradiction. 
\end{proof} 

Now Proposition~\ref{prop:transverse groups give GPS systems} implies that $\sigma_\phi = \phi \circ B_\theta|_{\Gamma \times \Lambda_\theta(\Gamma)}$ is an expanding cocycle and hence~\cite[Prop.\,6.3]{BCZZ} implies that $\delta^\phi(\Gamma) \leq \beta$.

\subsection{Proof of Proposition~\ref{prop:transverse groups give GPS systems}} Fix a non-elementary $\mathsf{P}_\theta$-transverse subgroup $\Gamma$ and $\phi\in \mathfrak{a}_\theta^*$  where 
$$
\lim_{n \rightarrow \infty} \phi(\kappa(\gamma_n)) = +\infty
$$
whenever $\{\gamma_n\}$ is a sequence of distinct elements of $\Gamma$.  As in the statement of Proposition~\ref{prop:transverse groups give GPS systems}, define cocycles $\sigma_\phi, \bar{\sigma}_\phi \colon \Gamma \times \Lambda_\theta(\Gamma) \to \Rb$ by 
$$
\sigma_\phi(\gamma, F) = \phi(B_\theta(\gamma,F)) \quad \text{and} \quad \bar\sigma_\phi(\gamma, F) = \iota^*(\phi)(B_\theta(\gamma,F)).
$$

Equation~\eqref{eqn:involution on kappa} implies that 
$$
\lim_{n \rightarrow \infty} \iota^*(\phi)(\kappa(\gamma_n)) = \lim_{n \rightarrow \infty} \phi(\kappa(\gamma_n^{-1})) =+\infty
$$
whenever $\{\gamma_n\}$ is a sequence of distinct elements of $\Gamma$. Then Lemma~\ref{lem:sigma phi is expanding} and~\cite[Prop.\,3.2]{BCZZ} imply that $\sigma_\phi$ and $\bar\sigma_{\phi}$ are expanding cocycles. Further, we may choose 
$$
\norm{\gamma}_{\sigma_\phi} = \phi(\kappa(\gamma)) \quad \text{and} \quad \norm{\gamma}_{\bar\sigma_\phi} =  \iota^*(\phi)(\kappa(\gamma)).
$$
Then, since $\sigma_\phi, \bar\sigma_\phi$ are expanding and hence proper, Equation~\eqref{eqn:GPS property for Iwasawa} implies that the triple $(\sigma_\phi, \bar\sigma_\phi, \phi \circ G_\theta)$ is a GPS system. 

Finally, item (2) in the ``moreover'' part of Proposition~\ref{prop:transverse groups give GPS systems} follows immediately from the fact that if $\gamma$ is loxodromic, then
$\phi(\lambda(\gamma))=\phi(B_\theta(\gamma,\gamma^+))$ (see the discussion at the start of Section 9.2 in~\cite{BQ}).

\subsection{Proof of Proposition~\ref{prop:transverse groups non-arithmetic}}

We will deduce the proposition from the following general result. 

\begin{proposition}\label{prop:Nonarithmeticity Anosov}
Suppose $\Gamma\subset \mathsf G$ is a subgroup  which is not virtually solvable. If  $\phi \in \mathfrak a^*$ is positive on the cone 
$$
\overline{\bigcup_{\gamma \in \Gamma} \Rb_{> 0} \cdot \lambda(\gamma)} \smallsetminus \{0\},
$$
then the subgroup of $\Rb$ generated by $\ell_\phi(\Gamma)$ is dense in $\Rb$.
\end{proposition}

\begin{remark} Previously, Benoist \cite{benoistasymp} proved that if $\Gamma$ is a Zariski-dense \emph{semigroup} of $\mathsf G$, then the additive subgroup of $\mf a$ generated by $\lambda(\Gamma)$ is dense. In particular, the subgroup of $\Rb$ generated by $\ell_\phi(\Gamma)=(\phi \circ \lambda)(\Gamma)$ is dense in $\Rb$. As we will explain in Example~\ref{sec:examples of nonarithmetic} below, Proposition~\ref{prop:Nonarithmeticity Anosov} is not true when $\Gamma$ is only assumed to be a semigroup. 
\end{remark}

\begin{proof} By replacing $\mathsf G$ with ${\rm Ad}(\mathsf{G})$ we can assume that $\mathsf{G}$ is an algebraic subgroup of $\mathsf{SL}(d,\Rb)$ and thus speak of Zariski closures of subgroups of $\mathsf G$.

 Let $\rho\colon \Gamma\hookrightarrow\mathsf G$ be the inclusion representation.
 Let $\ms G'\subset \ms G$ be the Zariski closure of $\rho(\Gamma)$, and $\ms G'=\ms H\ltimes \mr R_u(\ms G')$ be a Levi decomposition of $\ms G'$, where $\mr R_u(\ms G')$ is the unipotent radical and $\ms H$ is a Levi factor, which is a reductive group.

 Following \cite[\S2.5.4]{GGKW}, let $\rho^{ss}$ be the semisimplification of $\rho$, obtained by composing $\rho$ with the projection on the Levi factor $\ms H$,
 so that the Zariski closure of $\rho^{ss}(\Gamma)$ is exactly $\ms H$.
 Moreover, $\rho^{ss}$ has the same length function $\ell_\phi\circ\rho^{ss}=\ell_\phi\circ\rho$ as $\rho$ by \cite[Lem.\,2.40]{GGKW}.
  
 The Lie algebra $\mf h\subset\mf g$ of $\mathsf H\subset\mathsf G$ splits as $\mf h=\mf h'+\mf z$ where $\mf h'$ is semisimple and $\mf z$ is the center of $\mf h$.
 Note that the semisimple part $\mf h'$ cannot be trivial, otherwise the Lie algebra of $\ms G'$ would be solvable, so the identity component of $\ms G'$ would be solvable, and hence the intersection of this component with $\Gamma=\rho(\Gamma)$ would be a solvable finite-index subgroup of $\Gamma$, which contradicts the assumption that $\Gamma$ is not virtually solvable.
 
 Let $\mf a'_{\mathsf H}$ be a Cartan subspace of $\mf h'$. Then $\mf a_{\mathsf H}:=\mf a'_{\mathsf H}+\mf z$ is a Cartan subspace of $\mf h$. Up to conjugating by an element of $\mathsf G$, we may assume that $\mf a_{\mathsf H} \subset \mf a$.
 
 Let $\lambda_{\mathsf H}\colon \mathsf H\to \mf a'_{\mathsf H}+\mf z$ denote the Jordan projection of $\mathsf H$. Despite the inclusion $\mf a_{\mathsf H} \subset \mf a$, we note that $\lambda_{\mathsf H}$ may not equal $\lambda|_{\mathsf H}$. However, we have the following: If $h \in \mathsf{H}$, then 
 \begin{equation}\label{eqn:relation between lambdas}
 \lambda(h) = \lambda\left(e^{\lambda_{\mathsf H}(h)}\right).
 \end{equation}
 To see this, let $h=h_e h_u h_{ss}$ denote the Jordan decomposition of $h$. Then $h_{ss}$ is conjugate to $e^{\lambda_{\mathsf H}(h)}$ and so 
$$
\lambda(h) = \lambda(h_{ss}) = \lambda\left(e^{\lambda_{\mathsf H}(h)}\right).
$$

Define $f\colon \mf a_{\mathsf H}\to \R$ by
 $$
 f(X)=\ell_\phi(e^X)=\phi\left(\lambda\left(e^X\right)\right).
 $$
Then $f$ is piecewise linear, more precisely $\mf a_{\mathsf H}=\bigcup_j W_j$ is a finite union of closed convex cones each with non-empty interior where  $f|_{W_j}$ is linear (each $W_j$ has the form $\mf a_{\mathsf H} \cap w_j \mathfrak a^+$ where $w_j$ is the Weyl group of $\mathfrak a$). 
 
 By \cite[Prop.\ 9.7]{BQ}, the intersection of 
 $$
 \Cc:=  \overline{\bigcup_{\gamma \in \rho^{ss}(\Gamma)}\Rb_{>0} \cdot \lambda_{\mathsf H}(\gamma)}
 $$
and $\mf a'_{\mathsf H}$ contains some non-zero $X_0$  (there is a typo in the reference, the intersection part of the result does not hold for semigroups, but it does hold for groups). Since $\phi$ is positive on $\overline{\bigcup_{\gamma \in \Gamma} \Rb_{> 0} \cdot \lambda(\gamma)} \smallsetminus \{0\}$, Equation~\eqref{eqn:relation between lambdas} implies that $f(X_0) > 0$. 
 
Fix $W_j$ such that $\Cc \cap W_j$ has non-empty interior in $\Cc$ and $X_0 \in W_j$. Then by \cite[\S5.1]{benoistasymp1} there exists a Zariski-dense semigroup $S\subset\rho^{ss}(\Gamma)$ where 
$$ƒ
\Cc_S:=\overline{\bigcup_{\gamma \in S}\Rb_{>0} \cdot \lambda_{\mathsf H}(\gamma)} \subset \Cc \cap W_j. 
$$
Let $f' \colon \mf a_{\mathsf H} \to \Rb$ be the linear map with $f'|_{W_j} = f|_{W_j}$. By \cite[Prop.\,9.8]{BQ} the closure of the additive group generated by $\lambda_{\mathsf H}(S)$ in $\mf a_{\mathsf H}$ contains $\mf a'_{\mathsf H}$.
 Thus the closure of the additive subgroup of $\R$ generated by 
 $$\ell_\phi(S)=f\circ\lambda_{\mathsf H}(S)=f'\circ\lambda_{\mathsf H}(S)$$
 is the image under $f'$ (which is linear) of the closure of the additive group generated by $\lambda_{\mathsf H}(S)$, which contains $\mf a'_{\mathsf H}$ and hence the line spanned by $X_0$.
 Therefore the closure of the additive subgroup of $\R$ generated by $\ell_\phi(\Gamma)$ contains $f'(\R X_0)=\R f'(X_0)=\R$ as $f'(X_0)=f(X_0)>0$, which is what we wanted to prove.
\end{proof}

It seems unlikely that in the context of transverse groups, divergence of $\phi(\kappa_\theta(\cdot))$ implies the positivity on the limit cone hypothesis needed to use Proposition~\ref{prop:Nonarithmeticity Anosov}. However, this implication is known to be true for Anosov (or more generally relatively Anosov) groups, and such groups can always be found as subgroups of transverse groups. 

\begin{proof}[Proof of Proposition~\ref{prop:transverse groups non-arithmetic}] Suppose $\Gamma$ is a $\mathsf{P}_\theta$-transverse group, $\phi\in \mathfrak{a}_\theta^*$  and
 $$
\lim_{n \rightarrow \infty} \phi(\kappa(\gamma_n)) = +\infty
$$
whenever $\{\gamma_n\}$ is a sequence of distinct elements of $\Gamma$.

Let $\psi := \phi + \iota^*(\phi)$. Then, by Equation~\eqref{eqn:involution on kappa},
 $$
\lim_{n \rightarrow \infty} \psi(\kappa(\gamma_n))  =\lim_{n \rightarrow \infty} \phi(\kappa(\gamma_n))+ \phi(\kappa(\gamma_n^{-1}))=+\infty
$$
whenever $\{\gamma_n\}$ is a sequence of distinct elements of $\Gamma$.

Using ping-pong we can fix a free subgroup $\Gamma' \subset \Gamma$ such that there is 
$\Gamma'$-equivariant homeomorphism $\partial_\infty \Gamma' \to \Lambda_\theta(\Gamma') \subset \Lambda_\theta(\Gamma)$ of the Gromov boundary and the limit set.  
Then, by definition, $\Gamma'$ is $\mathsf{P}_\theta$-relatively Anosov. 

By the equivalence of (1) and (4) in~\cite[Th.\,10.1]{CZZ4} (see also ~\cite[Lem.\,3.4.2]{sambarino-dichotomy} for a similar result in the non-relative case), $\psi$ is positive on 
$$
\overline{\bigcup_{\gamma \in \Gamma'} \Rb_{> 0} \cdot \lambda(\gamma)} \smallsetminus \{0\}.
$$
Then Proposition~\ref{prop:Nonarithmeticity Anosov} implies that  
$$
\{ \phi(\lambda(\gamma)) + \iota^*(\phi)(\lambda(\gamma)) : \gamma \in \Gamma\} \supset \{ \psi(\lambda(\gamma)) : \gamma \in \Gamma'\} 
$$
generates a dense subgroup of $\Rb$. 
\end{proof}

\begin{example}\label{sec:examples of nonarithmetic} Here we provide an example showing that Proposition~\ref{prop:Nonarithmeticity Anosov} fails when $\Gamma$ is only assumed to be a semigroup. 

Let $\Gamma$ be a convex cocompact free subgroup of $\SL_2(\R)$ with two generators $a,b\in\Gamma$.
Let $S\subset\Gamma$ be the semigroup generated by $a$ and $b$.
There exists $C\geq 1$ large enough so that for all nonnegative $n_1,m_1,\dots,n_k,m_k$, if $s=a^{n_1}b^{m_1}\cdots a^{n_k}b^{m_k}\in S$ then $\log \norm{s} \leq \tfrac C2 |s|$ where $|s|= (n_1+m_1+\dots+n_k+m_k)$ is word-length with respect to the generating set $\{a,b\}$.
Now set 
$$\rho(s)=\begin{pmatrix} s & 0 & 0 \\ 0 & e^{C |s|} & 0 \\ 0 & 0 & e^{-C |s|} \end{pmatrix}\in\mathsf G=\SL_4(\R).$$
If $\phi(\diag(x,y,z,w))=x$ for any $\diag(x,y,z,w)\in \mf a$, then $\ell_\phi(s)\in C\Zb_{\geq1}$ for any $s\in\rho(S)$. Thus the length spectrum is arithmetic, even though $\phi$ is positive on the limit cone of $\rho(S)$.
\end{example}

\end{document}